\documentclass[10pt,a4paper]{article}

\pdfoutput=1

\usepackage{hyperref}
\usepackage[utf8x]{inputenc}
\usepackage{mystyle}
\usepackage{notations}
\usepackage{mymacros}
\usepackage{dsfont}
\usepackage{url}
\usepackage{graphicx}

\usepackage{amsmath}
\usepackage[text={6in,10in},centering]{geometry}
\allowdisplaybreaks

\title{Sinkhorn Divergences \\ for Unbalanced Optimal Transport}
\author{Thibault S\'ejourn\'e\footnote{DMA, ENS, \texttt{\{jean.feydy,thibault.sejourne,gabriel.peyre\}@ens.fr}}, 
Jean Feydy$^*$\footnote{CMLA, ENS Paris-Saclay, \texttt{\{trouve,feydy\}@cmla.ens-cachan.fr}} \:, Fran\cedil{c}ois-Xavier Vialard\footnote{Universit\'e Paris-Est Marne-la-Vall\'ee, LIGM, UMR CNRS 8049, \texttt{fxvialard@normalesup.org}}\\
Alain Trouv\'e$^\dagger$, 
Gabriel Peyr\'e\footnote{CNRS}\;$^{*}$}

\date{\today}

\begin{document}

\maketitle


\abstract{
    Optimal transport (OT) distances (also called Wasserstein or Earth Mover's distances) are now routinely used to fit parametric models in data sciences.
    They define geometric loss functions to compare point clouds or more generally probability distributions.
    Their efficiency is however inhered by some lack of robustness to outliers, missing parts and sampling noise.
    In this paper, we develop and analyze a new class of loss functions which combine two keys ideas to cope with these two robustness issues:
    (i) unbalanced optimal transport which relaxes the mass conservation constraint to lower sensitivity to outliers ;
    (ii) entropic regularization, which reduces the impact of sampling (especially in high dimension)  and lends itself to fast computations using the Sinkhorn algorithm.
    %
    Our first set of contributions is the study of this new loss function, the so-called unbalanced Sinkhorn divergence, and we prove it is convex, positive, definite, and metrizes the convergence in law.
    Our second set of contributions is the analysis of the associated Sinkhorn's algorithm, and we show its linear convergence for a wide set of unbalanced settings.
    We provide numerical experiments for gradient flows and 3D scene flow estimation, showcasing the impact of this gain of robustness for applications to shape registration.
}



\section{Introduction}
\label{sec:introduction}

Many problems in imaging and learning boil down to minimizing some loss function between two positive measures $\al$ and $\be$. 
Typically $\alpha$ can thought as a deformable template while $\beta$ is some fixed dataset which in practice is a discrete measure (supported on a point cloud). 
Designing this loss function is thus of major importance. It should be robust to various sources of errors such as modeling errors, outliers, occlusions and sampling noise.
Optimal Transport (OT) approaches have emerged as a general machinery to design such loss functions which leverage some underlying ground distance (or cost) $\C(x,y)$ between the points. 
The focus of this paper is to detail a family of loss functions built on top of OT, which integrates entropic regularization and unbalanced OT, and enjoys favorable theoretical and computational properties.

\paragraph{Csiszàr $\phi$-divergences}
Arguably the simplest loss functions between measures operate pointwise comparison between the distributions. They are central in our work, since we use them to cope with outliers and missing data in the so-called unbalanced OT approaches.
Informally, these Csiszàr $\phi$-divergences~\cite{csiszar1967information} are computed by measuring how much the relative density $\frac{\d\al}{\d\be}$ is close to 1. 
Considering data defined on a space $\Xx$, and writing the set of positive measures $\Mmp(\Xx)$, this ratio is defined using the Radon-Nikodym-Lebesgue decomposition, denoted $\al = \frac{\d\al}{\d\be} \be + \al^\bot$ for any $(\al,\be)\in\Mmp(\Xx)$.
We penalize the ratio with an entropy function $\phi:(0,\infty)\rightarrow[0,\infty)$ which is assumed to be \emph{convex}, \emph{positive}, \emph{lower-semi-continuous} and such that $\phi(1)=0$.
The Csiszàr divergence is defined as
\begin{align}\label{eq-csiszar-div}
	\D_\phi(\al|\be) \eqdef \int_\Xx \phi\bigg(\frac{\d\al}{\d\be}(x)\bigg) \d\be(x) + \phi^\prime_\infty \int_\Xx \d\al^\bot(x),
\end{align}
where $\phi^\prime_\infty \eqdef \lim_{x\rightarrow\infty}\tfrac{\phi(x)}{x}$ is called the recession constant.
The entropy is extended on $\R$ by setting $\phi(x)=+\infty$ for any $x<0$.
Popular instances are the Total Variation divergence ($\TV$) when $\phi(x)=|x-1|$ and the Kullback-Leibler divergence $(\KL(\al|\be) \eqdef \int\log(\tfrac{\d\al}{\d\be})\d\al - \int\d\al + \int\d\be$) when $\phi(x)=x\log x -x + 1$.
With $\KL$ one has $\phi^\prime_\infty=+\infty$, which has a finite value if and only if the supports satisfy $\spt(\al)\subset\spt(\be)$.
For discrete distributions defined on the same grid of $N$ points, these divergences are computed in $O(N)$ operations. The downside is however that they do not take into account the geometry of the underlying spaces. They are thus not continuous with respect to translation of the points. 
To be more precise, this means that they do not metrize the weak$^*$ convergence (denoted $\al_n\rightharpoonup\al$, which means, on compact spaces, that for any continuous function $\f$, $\int\f\d\al_n\rightarrow \int\f\d\al$). 
There are sequences such that $\al_n\rightharpoonup\be$, but $\D_\phi(\al_n|\be)\nrightarrow 0$.
A striking example is when supports are disjoints. 
Take $x_n \rightarrow x$ with $x_n \neq x$, so that $\de_{x_n} \rightharpoonup \de_x$, one has $\TV(\de_{x_n}|\de_x)=2$ and $\KL(\de_{x_n}|\de_x)=+\infty$.  


\paragraph{Kernel norms} 

A natural way to cure this lack of smoothness of these $\phi$-divergences is to consider kernel norms, also called Maximum Mean Discrepencies~\cite{gretton2007kernel}. They integrate a spatial similarity measure $k:\Xx^2\rightarrow\R$, called a \textbf{kernel}, and are defined as $\norm{\al-\be}_k^2\eqdef\int_{\Xx^2}k(x,y)\d\xi(x)\d\xi(y)$ where $\xi=\al-\be$.
This definition requires the kernel to be \emph{positive}, so that $\norm{\xi}_k^2$ is indeed a positive quantity for any measure $\xi$.
Assuming the kernel $k$ is universal (i.e. functions $x\mapsto k(x,y)$ are dense), then, on sharp contrast to $\phi$-divergences, $\norm{\cdot}_k$ metrizes the weak$^*$ topology~\cite{gretton2012kernel}.
It holds for instance in $\Xx=\R^d$ endowed with the Euclidean norm $\norm{\cdot}_2$ and $k(x,y)=e^{-\norm{x-y}_2^2 / (2\sigma^2)}$ ($\sigma>0$ encoding the bandwidth of the kernel). Another example is the energy distance kernel, $k(x,y)=-\norm{x-y}_2$, in which case the kernel is only conditionally positive and the resulting norm only metrizes the space of probability measures.
%
A chief advantage of these norms with respect to OT methods detailed below is that they are computed in $O(N^2)$ operations for discrete measures with $N$ points. 
They are however quite uninformative when comparing far away distributions.
For instance, when $k(x,y)=e^{-\norm{x-y}_2^2 / (2\sigma^2)}$, one has $\norm{\de_x - \de_y}_k^2 = 2(1 - e^{-\norm{x-y}_2^2 / (2\sigma^2)})$, whose gradient w.r.t. $x$ quickly vanishes as $x-y$ increases.

\paragraph{Optimal Transport Distance}

Optimal transport avoids the global integration operated by kernel norms by rather seeking a sparse assignment between points. This leads to a better behavior of the loss when the supports are far away, at the price of an expensive optimization problem. 
For some ground cost $\C:\Xx^2\rightarrow\R$ and two probability measures $(\al,\be)$, it reads 
\begin{gather}
\begin{aligned}\label{eq-defn-primal}
&\OT(\al,\be) \eqdef \min_{\pi \in \Mmp(\Xx^2)}
\textstyle\int_{(x,y)\in\Xx^2} \C(x,y) \, \d \pi(x,y)
\\
\text{s.t.}&\;\;
\pi_1 \eqdef \textstyle\int_{y\in\Xx} \d \pi(\,\cdot\,, y) = \al,
\pi_2 \eqdef \textstyle\int_{x\in\Xx} \d \pi(x, \,\cdot\,) =  \be.
\end{aligned}
\end{gather}
The optimization variable is a transport plan $\pi$ satisfying the so-called marginal constraints, i.e. $(\pi_1,\pi_2)$ should match the input probability distributions.
When $(\Xx,d_\Xx)$ is a metric space and $\C=d_\Xx^p$ with $p\geq 1$, $\OT^{1/p}$ is called the Wassersein-$p$ distance, and on compact spaces it metrizes the weak$^*$ convergence~\cite{santambrogio2015optimal}.
One has in this case $\OT^{1/p}(\de_x,\de_y) = d_\Xx(x,y)$ which supports its favorable behavior even for far away distributions. 
%

There are however many challenges that undermine its applicability, which we aim at lifting by combining several existing approaches in a coherent framework. 
The first one is both computational and statistical: computing exactly $\OT(\al,\be)$ for discrete distribution with $N$ points requires $O(N^3\log N)$ operation. Furthermore, in $\Xx=\R^d$, the error $\OT(\al_N,\be_N)-\OT(\al,\be)$ made when considering $N$ discrete samples drawn some (unknown) distributions $(\al,\be)$ is typically of the order of $1/N^{1/d}$ \cite{dudley1969speed, weed2017sharp}.
This is in sharp contrast with kernel norms, which can be computed in $O(N^2)$ operations and whose sampling error decay like $1/\sqrt{N}$. 
A popular approach to avoid these issues is to introduce entropic regularization, using Sinkhorn's algorithm, which is also a way to interpolate between OT and kernel norms~\cite{feydy2017optimal}.
The second source of difficulties is the lack of robustness of OT to outliers and missing data. This can be alleviated by considering unbalanced OT~\cite{liero2015optimal}, which replaces the exact mass conservation constraint by a soft penalty. 
In this paper, we bring together these two streams of ideas and define unbalanced Sinkhorn divergences.


\paragraph{Unbalanced and entropic OT}


Unbalanced OT consists in relaxing the constraints $(\pi_1=\al,\pi_2=\be)$ which enforces \emph{mass conservation} (because it imposes $\int\d\al=\int\d\be$).
Following~\cite{liero2015optimal}, it uses a soft penalty $\D_\phi(\pi_1|\al) + \D_\phi(\pi_2|\be)$, so that in general $(\pi_1,\pi_2)\neq(\al,\be)$.
This combination of transportation with mass creation/destruction increase the robustness of the optimal transport plan to outliers~\cite{mukherjee2021outlier, fatras2021unbalanced}.
We focus on the unbalanced OT formulation called 'static'.
There exists another formulation called 'dynamic'~\cite{liero2015optimal, chizat2015unbalanced, kondratyev2016new, chizat2017tumor}.
Unbalanced OT proved to be successful in biology~\cite{schiebinger2017reconstruction, yang2018scalable}, videos~\cite{lee2019parallel} or prove global convergence of 2-layers neural networks~\cite{chizat2018global, rotskoff2019global}.

Unbalanced OT combines nicely with entropic regularization, and consists in adding a term $\epsilon\KL(\pi|\al\otimes\be)$ into Equation~\eqref{eq-defn-primal}.
The resulting optimization problem is strictly convex and can be solved using the popular \emph{Sinkhorn algorithm}~\cite{sinkhorn1964relationship}. This leads to a highly parallelizable computation scheme which streams well on GPUs~\cite{cuturi2013lightspeed}. 
The combination of unbalanced OT and entropic regularization reads, for any $\epsilon>0$, 
\begin{align}\label{eq-primal-unb}
\OTb(\al,\be) \eqdef \inf_{\pi \in \Mmp(\Xx^2)} \int_{\Xx^2} \C \,\d\pi + \D_\phi(\pi_1|\al) + \D_\phi(\pi_2|\be) + \epsilon \KL(\pi|\al\otimes\be).
\end{align}
This formulation was first proposed in~\cite{chizat2016scaling} with a generalized Sinkhorn algorithm to solve this approximation of unbalanced OT.
However, the convergence is only known for $\D_\phi=\KL$, and is not proved when $(\al,\be)$ are not discrete measures.
One retrieves balanced OT as a particular instance of this formulation when using $\phi=\iota_{\{1\}}$ ($\phi(1)=0$ and $+\infty$ otherwise).
Popular choices are $\D_\phi=\TV$ or $\KL$ (see Section~\ref{sec-exmp-f-div} for details).
We thus use the same notation $\OTb$ to emphasize those examples are instances of the same framework.
Concerning statistical complexity, entropic regularization transfers the curse of dimensionality into the constants.
It scales for balanced OT as $|\OTb(\al_N,\be_N)-\OTb(\al,\be)|=O(\epsilon^{-d/2}n^{-1/2})$ for compact measures~\cite{genevay2018sample}.
It can be refined for $\C = \norm{\cdot}^2$ with subgaussian measures~\cite{mena2019statistical}.

\paragraph{Unbalanced Sinkhorn divergence} 

While using $\OTb(\al,\be)$ enjoys favorable computational property, especially in high dimension, it suffers from a strong bias as $\epsilon$ becomes larger.
More precisely, $\OTb(\al,\be)$ is not a distance, in particular $\OTb(\al,\al)>0$. We show in Section~\ref{sec-pos-sink-div} that when $\epsilon \rightarrow +\infty$, then $\al$ minimizing $\OTb(\al,\be)$ degenerates to a Dirac mass when $\epsilon \rightarrow +\infty$. 
In the balanced case, this entropic bias has been studied in details and removed by considering debiased formulation~\cite{janati2020debiased}. This idea was suggested in~\cite{ramdas2017wasserstein, genevay2018learning, salimans2018improving} for balanced OT for statistical testing and generative learning.
An important contribution of our work is to extend this idea to the unbalanced OT setting. We introduce the unbalanced Sinkhorn divergence as 
\begin{align*}
\Sb(\al,\be) \eqdef \OTb(\al,\be) - \tfrac{1}{2} \OTb(\al,\al) - \tfrac{1}{2} \OTb(\be,\be) +\tfrac{\epsilon}{2} \big( m(\al) - m(\be) \big)^2,
\end{align*}
where $m(\al) = \int_\Xx \d\al \geq 0$ is the total mass of $\al$.
When $m(\al)=m(\be)$, one recovers the previously proposed formulations, which has been studied theoretically in~\cite{feydy2018interpolating}, where it was shown to be convex, positive, definite, and to metrize the convergence in law. Our contributions include the extension of these results to the general case of unbalanced OT.

\paragraph{Contributions}
Our contributions are the following:
\begin{itemize}
	\item In Section~\ref{sec-operators}, we show new theoretical results on the unbalanced Sinkhorn algorithm initially derived in~\cite{chizat2016scaling}.
		This includes a proof of linear convergence of this algorithm for general measures (not only discrete) and a general class of divergences $\D_\phi$.
		This is made possible thanks to a new formulation of the algorithm. 
	\item Section~\ref{sec-ot-prop} dwells into the continuity and differentiability properties of $\OTb$ and $\Sb$ w.r.t. weak* topology.
	\item In Section~\ref{sec-ot-prop}, we prove the main theoretical results of the paper. Theorem~\ref{thm-sink-unb} shows that $\Sb$ is a convex, positive, definite loss on $\Mmp(\Xx)$. Theorem~\ref{thm-sink-weak-cv} states that it metrizes the weak* convergence when $\D_\phi=\KL$ and $\TV$. 
	\item In Section~\ref{sec-stat-comp} we extend the results of~\cite{genevay2018sample} on statistical complexity of $\OTb$ to the unbalanced setting. More precisely, we show the error rate $|\OTb(\al_N,\be_N)-\OTb(\al,\be)|=O(\epsilon^{-d/2}n^{-1/2})$ remains true.	
	\item Section~\ref{sec-implementation} focuses on discrete measures and describes how the Sinkhorn algorithm and divergence are implemented. 
		Section~\ref{sec:numerical} provides numerical illustrations to showcase the robustness properties of this new loss function. We display gradient flows with synthetic data and a detailed quantitative analysis for optical flow estimation on real data.
\end{itemize}

\paragraph{Notations}
Here $(\Xx,d_\Xx)$ represents a metric space assumed to be compact.
We define $(\Cc(\Xx),\norm{\cdot}_\infty)$ as the space of continuous functions $\f:\Xx\rightarrow\R$ endowed with the sup-norm $\norm{\f}_\infty\eqdef\max_{x\in\Xx} |\f(x)|$.
We note the space of positive Radon measures as $\Mmp(\Xx)$.
It is in duality with $(\Cc(\Xx),\norm{\cdot}_\infty)$ and is endowed with the weak* topology.
We define it as $\al_n \rightharpoonup \al$ $\Leftrightarrow$ $\forall\f\in\Cc(\Xx),\;\int_\Xx f \d \al_n \rightarrow \int_\Xx f\d \al$.
We note $\Mmpo(\Xx)$ the space of probabilities, and $\Mmpp(\Xx)\eqdef\Mmp(\Xx)\setminus\{0\}$.
For the sake of concision we replace integrals by the duality pairing $ \dotp{\al}{\f} \eqdef \int_\Xx \f\d\al = \mathbb{E}_\al[\f]$.
We define the tensor product of measures as  $(\al\otimes\be)(x,y)\eqdef\al(x)\be(y)$ and the tensor sum of functions as $(\f\oplus\g)(x,y)\eqdef \f(x) + \g(y)$.

A kernel $k:\Xx^2\rightarrow\R$ is called positive if for any signed measure $\al$, $\norm{\al}_k^2 \eqdef \dotp{\al\otimes\al}{k} = \int_{\Xx^2} k(x,y) \d\al(x)\d\al(y)\geq 0$.
For discrete measures $\al=\sum_i \al_i \de_{x_i}$, it means that the matrix $\mathbb{K} = (k(x_i,y_j))_{i,j}$ is positive.
We assume kernels are continuous in this paper.
A kernel is called universal if the set of functions $\enscond{ x\mapsto k(x,y) }{ y\in\Xx }$ is dense in $\Cc(\Xx)$.
We also define the convolution with a measure $k\star\al$ as the continuous function $x\mapsto\int_{y\in\Xx} k(x,y)\d\al(y)$.

Concerning the cost $\C$ appearing in Program~\eqref{eq-primal-unb}, we assume it is symmetric and continuous.
We also assume it is $\gamma$-Lipschitz in the sense that for any $(x,y)\in\Xx$, $\norm{\C(x,.) - \C(y,.)}_\infty \leq \gamma\d_\Xx(x,y)$.
Finally, we define the diameter of a set $A$ as $\text{diam}(A) \eqdef \sup_{(x,y)\in A^2} \d_\Xx(x,y)$.
The diameter of a measure is the diameter of its support.

\section{Background on Csiszár-divergences, Softmin and anisotropic proximity operators}
\label{sec-operators}

We present in this section concepts and properties required to study $\OTb$, $\Sb$, and the Sinkhorn algorithm.
We start with general properties of Csiszàr divergences $\D_\phi$, then focus on two operators called Softmin and anisotropic proximity operator involved in the analysis of Sinkhorn algorithm.


\subsection{Csiszár divergences}

We recall that entropy functions and Csiszàr divergences are defined in the introduction (see Equation~\ref{eq-csiszar-div}).
Some of their main properties are detailed below:

\begin{proposition}~\cite[Corollary (2.9)]{liero2015optimal}\label{prop-f-div-liero}
For any entropy function $\phi$, the divergence $(\al,\be)\mapsto\D_\phi(\al|\be)$ is positive, jointly convex, 1-homogeneous and weak* lower semicontinuous in $(\al,\be)$.
\end{proposition}

The \emph{Legendre conjugate} $\phi^*:\R\rightarrow\R$ of an entropy function $\phi$ is defined as $\phi^*(q) \eqdef \sup_{p\geq 0} pq - \phi(p)$.
The function $\phi^*$ appears in the dual formulation of $\OTb$, and has the following properties.

\begin{proposition}[Properties of the entropy conjugate $\phi^*$]\label{prop-legendre-conj}
For any entropy function $\phi$,
\begin{enumerate}
  \item One has $\partial{\phi^*} \subset\R_+$, i.e. $\phi^*$ is non-decreasing.
  \item The domain of $\phi^*$ is $(-\infty, \phi^\prime_\infty)$.
  \item One has $\lim_{q\rightarrow -\infty} \phi^*(q) = -\phi(0)$ and $\lim_{q\rightarrow +\infty} \phi^*(q) = +\infty$.
\end{enumerate}
\end{proposition}
\begin{proof}
Take $q\leq q'$. Because $\text{dom}(\phi)\subset\R_+$, for any $x\in\text{dom}(\phi)$ one has $xq - \phi(x) \leq xq' - \phi(x)$. Taking the supremum in $x$ gives $\phi^*(q) \leq \phi^*(q')$. Since $\phi^*$ is convex and non-decreasing we get $\partial{\phi^*} \subset\R_+$.

Assume $\phi^\prime_\infty < \infty$ and take $q > \phi^\prime_\infty$, $p>0$. Then one has $\lim_{p\rightarrow +\infty} p(q - \frac{\phi(p)}{p}) = +\infty$, i.e. $q\notin\text{dom}(\phi^*)$. If $\phi^\prime_\infty = \infty$ then for any $q\in\R$ $p\mapsto pq - \phi(p)$ goes to $-\infty$ when $p\rightarrow +\infty$, which gives coercivity in $p$ and guarantees that $\phi^*(q)$ is finite, i.e. $q\in\text{dom}(\phi^*)$.

By definition one has $\phi^*(q) \geq -\phi(0)$. When $q\rightarrow -\infty$, if $p>0$ then $pq - \phi(p)\rightarrow -\infty$. Thus we necessarily have $p=0$ and in that case it gives $\lim_{-\infty} \phi^* = -\phi(0)$. when $q\rightarrow +\infty$, because $\phi$ is an entropy function, we have that $\phi^*(q) \geq q.1 -\phi(1) = q$, which gives that $\lim_{+\infty} \phi^* = +\infty$.
\end{proof}

\begin{remark}\label{rem-param-rho}
For unbalanced OT, one can add a parameter $\rho >0$ so as to tune the strength of the mass conservation, and use $\D_{\rho\phi}=\rho\D_\phi$. 
Note that $(\rho\phi)^*(q) = \rho\phi^*(q / \rho)$.
One retrieves balanced OT when $\rho \rightarrow\infty$ (provided $\phi^{-1}(\{0 \})=\{ 1 \}$).
\end{remark}

\subsection{Softmin operator}

The Softmin operator is a smoothed version of the minimum operator.

\begin{definition}[Softmin operator]\label{def-smin}
For any $\al \in \Mmpp(\Xx)$ and $\epsilon >0$, the Softmin operator $\Smin{\al}$ is such that for any $\f\in\Cc(\Xx)$
\begin{align}
\Smin{\al}(\f) \eqdef -  \epsilon \log\dotp{\al}{\exp(-\f/\epsilon)}.
\end{align}
\end{definition}

We detail some properties of this operator: these are helpful to get insights on its behaviour and are used extensively in subsequent proofs.

\begin{proposition}[Properties of the Softmin operator] \label{prop-smin-interp}
	For any $\epsilon>0$, Softmin is continuous w.r.t inputs $(\al,\f)$.
	It interpolates between a minimum operator and a sum, it is order preserving, and it is translation invariant.
	Those properties respectively read
	\begin{gather*}
	\big( \al_n \rightharpoonup\al \text{   and   } \f_n \xrightarrow{\norm{.}_\infty~}\f \big)
		\Longrightarrow\Smin{\al_n}(\f_n)\rightarrow\Smin{\al}(\f),\\
		\forall \al\in\Mmpo(\Xx),\, \dotp{\al}{\f}\xleftarrow{\epsilon\rightarrow +\infty} \Smin{\al}(\f)
		\xrightarrow{\epsilon\rightarrow 0}\min_{ x \in \Supp(\al)} \f(x), \\
		\forall(\f,\g)\in\Cc(\Xx), \, \f \leqslant\g \Longrightarrow\Smin{\al}(\f) \leqslant~\Smin{\al}(\g), \label{eq:smin_order} \\
		\forall K\in\R, \, \Smin{\al}(\f + K) = \Smin{\al}(\f) + K. \label{eq:smin_constant}
	\end{gather*}
\end{proposition}

We now mention some regularity properties of the Softmin.

\begin{lemma}[The Softmin operator is non-expansive]\label{lem-smin-lipschitz-func}
For any $\al\in\Mmpp(\Xx)$, the Softmin is $1$-Lipschitz. 
It is a non-expansive operator, with
\begin{align*}
 \forall (\f,\g)\in\Cc(\Xx), \quad
 |\Smin{\al}(\f) - \Smin{\al}(\g)| &\leq \norm{\f - \g}_\infty.
\end{align*}
\end{lemma}
\begin{proof}
Write $u_t = t(\g - \f) + \f$ for $t\in [0,1]$. 
The function $u_t$ is $\al$-measurable on a compact set, thus the function $t\mapsto \Smin{\al}(u_t)$ is differentiable. 
It gives
\begin{align*}
  |\Smin{\al}(\g) - \Smin{\al}(\f)| &= |\int_0^1 \frac{\d}{\d t} \Smin{\al}(u_t)| 
  = | \int_0^1 \dotp{\al}{(\g - \f)\frac{e^{u_t / \epsilon} }{ \dotp{\al}{ e^{u_t / \epsilon} } }} | \\
  &\leq  \int_0^1 |\dotp{\al}{(\g - \f)\frac{e^{u_t / \epsilon} }{ \dotp{\al}{ e^{u_t / \epsilon} } }} | 
  \leq \norm{\g - \f}_\infty.
\end{align*}
\end{proof}

We define two maps $\Ss_\al:\Cc(\Xx)\rightarrow\Cc(\Xx)$ and $\Ss_\be:\Cc(\Xx)\rightarrow\Cc(\Xx)$ derived from the Softmin. 
For any $(\f,\g)\in\Cc(\Xx)^2$ and $(x,y)\in\Xx^2$, the outputs $(\Ss_\al(\f), \Ss_\be(\g))$ read
\begin{align}\label{eq-defn-softmin-func}
	\Ss_\al(\f)(y)\eqdef \Smin{\al}(\C(\cdot,y) - \f),\qandq
	\Ss_\be(\g)(x)\eqdef \Smin{\be}(\C(x, \cdot) - \g).
\end{align}
Those maps are at the heart of Sinkhorn algorithm which solves the dual of~\eqref{eq-primal-unb}. We present the properties of $\Ss_\al(\f)$ (which hold analogously for $\Ss_\be(\g)$).

\begin{lemma}[Regularity of $\Ss_\al(\f)$]\label{lem-smin-cost-regular}
	Assume $\C$ is continuous on $\Xx^2$. For any $\al$-integrable function $\f$, $\Ss_\al(\f)$ is a continuous function.
	If $\C$ is $\gamma$-Lipschitz in each of its inputs, then $\Ss_\al(\f)$ is $\gamma$-Lipschitz.
\end{lemma}
\begin{proof}
	The function $\f$ is $\al$-integrable and $\C$ is continuous on $\Xx$ compact, thus $x\mapsto\C(.,x)$ is uniformly bounded w.r.t. $x$.
	The dominated convergence theorem holds and $x\mapsto\dotp{\al}{e^{\frac{\f(.) - \C(.,x)}{\epsilon}}}$ is continuous.
	Concerning the Lipschitz property, Lemma~\ref{lem-smin-lipschitz-func} gives
	\begin{align*}
	|\Smin{\al}(\C(x,.) - \f) - \Smin{\al}(\C(y,.) - \f) | &\leq \norm{\C(x,.) - \C(y,.)}_\infty \\
	&\leq \gamma\d_\Xx(x,y).
	\end{align*}
\end{proof}

\subsection{Anisotropic proximity operator}

The maps $(\Ss_\al,\Ss_\be)$ suffice to define the \emph{balanced} Sinkhorn algorithm. 
The unbalanced version also involves the \emph{anisotropic proximity operator}, introduced in~\cite{combettes2013moreau,teboulle1992entropic}. 
It generalizes the usual \emph{proximal} operator from Hilbert spaces to Banach spaces. We start with its definition.


\begin{definition}[Aprox operator]\label{def-prox}
Let $h : \R \rightarrow \R$ be a convex function and $\epsilon > 0$. The anisotropic proximity operator is defined as
\begin{align}\label{eq-def-aprox}
  \forall p \in \R, \quad
  \aprox{h}(p) \eqdef \arg\min_{q\in\R} \epsilon \exp(\tfrac{p - q}{\epsilon}) + h(q)\in \text{dom}(h)\cup\{+\infty\}.
\end{align}
\end{definition}

If there exists $x\in\text{dom}(h)$ such that $\partial h(x)\subset\R_+^*$, then for any $p\in\R$, $\aprox{h}(p) < +\infty$.
It holds when $h=\phi^*$, see Proposition~\ref{prop-legendre-conj}.

 As detailed in \cite{combettes2013moreau}, a generalized Moreau decomposition connects it with a $\KL$ (Bregman) proximity operator that reads
\begin{align*}
  \aprox{\phi^*}(p) &=  p -  \epsilon\log \proxdiv{\phi}(p),\\
  	\quad\text{where}\quad \proxdiv{\phi}(p)&\eqdef \arg\inf_{q\in\R_+} \phi(q) + \KL(q, \exp(\tfrac{p}{\epsilon}))
\end{align*}
The above $\proxdiv{\phi}$ operator is used in~\cite{chizat2016scaling} to define the Sinkhorn algorithm.
We present below one advantage of $\aprox{\phi^*}$, namely its non-expansiveness, similarly to the standard proximal operators.
It is key in Section~\ref{sec-sinkhorn} to prove convergence of Sinkhorn algorithm in wide generality.
%
%

\begin{proposition}[The aprox is non-expansive]\label{prop-nonexp}
For any entropy $\phi$, the anisotropic proximity operator is $1-$Lipschitz. 
For any $(p,q)\in\R$, one has
\begin{align*}\label{eq-aprox-nonexp}
  \norm{\aprox{\phi^*}(p) - \aprox{\phi^*}(q)}_\infty \leq |p - q|.
\end{align*}
\end{proposition}
The proof relies on properties of monotone operators, and is deferred to Appendix~\ref{appendix-proofs}.
We end with a monotonicity property on the aprox.

\begin{proposition}[The aprox is non-decreasing]\label{prop-monoton-aprox}
  For any entropy $\phi$ with Legendre transform $\phi^*$, the operator $\aprox{\phi^*}$ is non-decreasing.
\end{proposition}

\begin{proof}
  Assume $\phi^*$ is smooth, and write $g(p) \eqdef \aprox{\phi^*}(p)$. 
  The implicit function theorem holds and yields differentiability of aprox. 
  Its derivative reads
  $\phi^{*\prime}(g(p)) = e^{\frac{p - g(p)}{\epsilon}}$,
  which implies that $g^\prime(p) = \frac{\phi^{*\prime}(g(p))}{\phi^{*\prime}(g(p)) + \phi^{*\prime\prime}(g(p))} \in[0,1],$
since $\phi^*$ is convex and non-decreasing (Proposition~\ref{prop-legendre-conj}), we have $\phi^{*\prime},\phi^{*\prime\prime}\geq 0$.

If $\phi^*$ is not smooth then one can regularize it, and let the regularization go to zero. It yields a sequence of operators $(\aprox{\phi_n^*})_n$ which are non-decreasing, and that converges pointwise to $\aprox{\phi^*}$. Due to closedness of non-decreasing functions, the limit is also non-decreasing, hence the result.
\end{proof}



\section{The Sinkhorn algorithm and its convergence}
\label{sec-sinkhorn}

We now present a reformulation of the Sinkhorn algorithm to the unbalanced setting, first introduced in~\cite{chizat2016scaling}.
Crucially, the novelty is the introduction of $\aprox{\phi^*}$, which allows to prove convergence for a variety of settings.
Our general formalism allows to consider positive Radon measures.
So far convergence was only proved for discrete measures when $\D_\phi=\iota_{(=)}$ or $\rho\KL$~\cite{chizat2016scaling} via the non-linear Perron Frobenius theory~\cite{lemmens2012nonlinear}.
We emphasize that our reformulation also leads to a numerically stable algorithm, see Section~\ref{sec-implementation}.
%

\subsection{Sinkhorn iterations}

The Sinkhorn algorithm aims at solving the dual problem of $\OTb(\al,\be)$, which reads $\OTb(\al,\be)
=\sup_{(\f, \g)\in \Cc(\Xx)^2}\Ff(\f,\g)$, where
\begin{equation} \label{eq-dual-unb}
\Ff(\f,\g) \eqdef
- \dotp{\al}{\phi^*(-\f)}
- \dotp{\be}{\phi^*(-\g)}
-\epsilon \dotp{\al\otimes\be}{\fefgc - 1},
\end{equation}
and $\phi^*$ is the Legendre transform of $\phi$.
This problem is equivalent to Problem~\eqref{eq-primal-unb} thanks to Fenchel-Rockafellar theorem.
The variables $(\f,\g)$ are called \emph{dual potentials}.
The optimal plan $\pi$ and optimal $(\f,\g)$ are connected via the primal-dual optimality relation
\begin{align} \label{eq-implicit-plan}
\pi(x,y) ~&=~ \exp\big[\tfrac{1}{\epsilon}(\f(x)+ \g(y) - \C(x,y))\big]\al(x)\be(y) ~\in~ \Mmp(\Xx\times\Xx).
\end{align}



We present the dual optimality conditions. They involve the operators presented Section~\ref{sec-operators}.

\begin{proposition}[Optimality conditions for the dual problem]\label{prop-optimality-prox}
The first order optimality condition of Formulation~\eqref{eq-dual-unb} reads
\begin{gather}\label{eq-optim-cond}
	\begin{aligned}
	\f(x) &= -\aprox{\phi^*}\big(-\Ss_\be(\g)(x)\big),\quad \al-\text{a.e.} \\
	\g(y) &= -\aprox{\phi^*}\big(-\Ss_\al(\f)(y)\big), \quad \be-\text{a.e.},
	\end{aligned}
\end{gather}
where $\aprox{\phi^*}$ is applied pointwise, and $(\Ss_\al,\Ss_\be)$ are defined Equation~\eqref{eq-defn-softmin-func}.
For the sake of brevity, we define operators $(\Aa\Ss_\al,\Aa\Ss_\be)$ outputing functions in $\Cc(\Xx)$ to write Equations~\eqref{eq-optim-cond} as $\f =  \Aa\Ss_\be(\g)$ and $\g=\Aa\Ss_\al(\f)$.
\end{proposition}
\begin{proof}
The first order conditions on $\partial_\f\Ff(\f,\g)$ and $\partial_\g\Ff(\f,\g)$ read
\begin{align} \label{eq-dual-optimality}
e^{\frac{\f}{\epsilon}}\dotp{\be}{e^{\frac{( \g - \C )}{\epsilon}}} \in \partial\phi^*(-\f),\, \al\text{-a.e.}
\;\;\;\; \text{and} \;\;\;\;
e^{\frac{\g}{\epsilon}}\dotp{\al}{e^{\frac{( \f - \C )}{\epsilon}}} \in \partial\phi^*(-\g), \, \be\text{-a.e.}
\end{align}
One has $e^{-\Smin{\be}(\C(y,.)-\g) / \epsilon} = \dotp{\be}{e^{\frac{( \g - \C )}{\epsilon}}}$. 
The optimality condition of $\aprox{\phi^*}$ (Equation~\eqref{eq-def-aprox}) is $e^{\frac{p - q}{\epsilon}} \in \partial\phi^*(q)$.
For any $y\in\Xx$, when $p=-\Smin{\be}(\C(y,.)-\g)$ or $p=-\Smin{\al}(\C(.,y)-\g)$ we retrieve Equation~\eqref{eq-dual-optimality} for $q = -\f$ or $q=-\g$. 
Hence the reformulation of Equation~\eqref{eq-dual-optimality} into Equation~\eqref{eq-optim-cond}.
\end{proof}

%
%
\begin{remark}\label{rem-extrapolate-pot}
	Note that while Equations~\eqref{eq-optim-cond} only need to hold $(\al,\be)$-a.e. for optimality, they are well-defined for any $(x,y)\in\Xx^2$. It allows to extrapolate and define $(\f,\g)$ on the full space $\Xx$.
\end{remark}
One deduces Sinkhorn iterations from Proposition~\ref{prop-optimality-prox}, which perform an alternate dual maximization on $\Ff(\f,\g)$ by fixing one variable and optimizing the other.

\begin{definition}[Sinkhorn algorithm]
\label{def-sinkhorn}
Starting from some $\g_0 \in \Cc(\Xx)$, the iterations of Sinkhorn read
\begin{gather}
\begin{aligned}\label{eq-sinkhorn-iter-1}
\f_{t+1} &\eqdef \arg\max_{\f}\Ff(\f,\g_t)=\Aa\Ss_\be(\g_t),\\
\g_{t+1} &\eqdef  \arg\max_{\g}\Ff(\f_{t+1},\g)=\Aa\Ss_\al(\f_{t+1}).
\end{aligned}
\end{gather}
\end{definition}


%
The $\aprox{\phi^*}$ is frequently known in closed form, has a small computational cost (see Section~\ref{sec-exmp-f-div}), and acts pointwise on e.g. $\Ss_\al(\f)$.
Thus the time and space complexity of iterations~\eqref{eq-sinkhorn-iter-1} is the same as balanced Sinkhorn, i.e. computing $(\Ss_\al,\Ss_\be)$ is the bottleneck scaling as $O(N^2)$. 

%

We end this section with a proof that a pair of potentials is optimal if and only if it is a fixed point of the Sinkhorn algorithm.

\begin{proposition}[Link between the Sinkhorn algorithm and the unbalanced $\OTb$ problem]\label{prop-cns-optimality}
A pair of dual potentials $(\f,\g)$ is optimal if and only if it is a fixed point of the Sinkhorn mapping.
\end{proposition}
\begin{proof}
Decompose the dual functional~\eqref{eq-dual-unb} as $\Ff(\f,\g) = \Ff_1(\f) + \Ff_2(\g) + \Ff_3(\f,\g)$ where $\Ff_1(\f)=\dotp{\al}{-\phi^*(-\f)}$, $\Ff_2(\f)=\dotp{\be}{-\phi^*(-\g)}$ and $\Ff_3(\f,\g)=-\epsilon \dotp{\al\otimes\be}{\fefgc - 1}$. For any function $\Gg$ one has $\partial\Gg(\f,\g) \subseteq \partial_1\Gg(\f,\g)\times\partial_2\Gg(\f,\g)$. Nevertheless for $\Gg(\f,\g)=\Ff_1(\f)+\Ff_2(\g)$ the inclusion of subgradients becomes an equality since it is a separable function in $(\f,\g)$. Furthermore, $\Ff_3$ is a differentiable function, thus the same equality between subgradients holds. Eventually, the subgradients can be summed because $\Ff_3$ is differentiable on $\R$, thus the intersection of subgradients is non-empty, and $\partial\Ff=\partial((\Ff_1 + \Ff_2) + \Ff_3) = \partial(\Ff_1 + \Ff_2) +   \partial\Ff_3 = \partial_1\Ff\times\partial_2\Ff$.

The condition $0\in\partial\Ff$ means that the dual variable are optimal, and $0\in\partial_1\Ff\times\partial_2\Ff$ that the potentials are fixed points of the Sinkhorn mapping. The equality between those two sets means that being optimal and being fixed points is equivalent.
\end{proof}

\subsection{Examples of Csiszár divergences}
\label{sec-exmp-f-div}

We present now explicit settings, and provide in each case $(\phi,\phi^*,\aprox{\phi^*})$. 
Those settings correspond to different priors on mass variation dynamics. 
We provide below illustrations of $\aprox{\phi^*}$ and the influence of $\phi$ on $\OTb$.

\paragraph{Balanced OT} ($\D_\phi=\iota_{(=)}$) corresponds to using $\phi = \iota_{\{1\}}$, the convex indicator function which encodes the marginal constraints, i.e. $\frac{\d\pi_1}{\d\al} = 1$ and $\frac{\d\pi_2}{\d\be} = 1$. In this case we get $\phi^*(q) = q$ and $\aprox{\phi^*}(p) = p$.

\paragraph{Kullback-Leibler} ($\D_\phi=\rho\KL$) corresponds to $\phi(p) = \rho ( p \log p -p +1)$, to $\phi^*(q) = \rho(e^{q / \rho} -1)$ and $\aprox{\phi^*}(p) = (1+\tfrac{\epsilon}{\rho})^{-1} p$.
As discussed in \cite{liero2015optimal}, when $\epsilon=0$ and $d_\Xx$ a distance, unbalanced OT defines the Gaussian-Hellinger and the Kantorovitch-Hellinger distances on $\Mmp(\Xx)$ (respectively for $\C(x,y)=d_\Xx(x,y)^2$ and $\C(x,y)=-2\log\cos(d_\Xx(x,y)\wedge\pi))$).

\paragraph{Range} ($\D_\phi=RG_{[a,b]}$) is defined for $0\leq a \leq 1 \leq b$ with $\phi = \iota_{[a,b]}$ and $\phi^*(q) = \max(a q, b q)$. The proximal operator is
\begin{align*}
\aprox{\phi^*}(p) =
\text{Soft-Thresh}_{\epsilon \log a}^{\epsilon \log b}(p)=
\begin{cases}
p - \epsilon\log a & \quad\text{if } p - \epsilon\log a < 0,\\
p - \epsilon\log b & \quad\text{if } p - \epsilon\log b > 0,\\
0 & \quad\text{otherwise.}
\end{cases}
\end{align*}
Note that in this setting the problem can be infeasible, i.e. $\OTb(\al,\be)=+\infty$. We have $\OTb(\al,\be)<\infty$ if and only if $[m(\al)a,m(\al)b]\cap[m(\be)a,m(\be)b]\neq\emptyset$.
\begin{proof}
	Take $(\al,\be)$ such that $m(\al)b < m(\be)a$. The range penalty imposes on the primal $m(\al)a \leq m(\pi) \leq m(\al)b$ and $m(\be)b \leq m(\pi) \leq m(\be)b$, which is infeasible. A similar proof holds if $m(\be)b < m(\al)a$.
	Conversely, take $k\in[m(\al)a,m(\al)b]\cap[m(\be)a,m(\be)b]\neq\emptyset$. Then one can verify that $\pi=(k / m(\al)m(\be))\al\otimes\be$ is a feasible plan, which guarantees that both the primal and the dual are finite.
\end{proof}

\paragraph{Total Variation} ($\D_\phi=\rho TV$) corresponds to $\phi(p) = \rho|p-1|$ and for $q\leq \rho$, $\phi^*(q) = \max(-\rho, q)$ with $\text{dom}(\phi^*) = (-\infty, \rho]$. The anisotropic operator reads
\begin{align*}
\aprox{\phi^*}(p) =
\text{Clamp}_{[-\rho, +\rho]}(p)=
\begin{cases}
-\rho & \quad\text{if } p < -\rho \\
p & \quad\text{if } p \in [-\rho, \rho]\\
\rho & \quad\text{if } p > \rho.
\end{cases}
\end{align*}
In this case, unbalanced OT (i.e. when $\epsilon=0$) is a Lagrangian version of partial optimal transport~\cite{figalli2010optimal}, where only some fraction of the total mass is transported. 
When $\C$ is a distance, it is also equivalent to the flat norm (the dual norm of bounded Lipschitz functions)~\cite{hanin1999extension,hanin1992kantorovich,schmitzer2019framework}.

\paragraph{Power entropies} divergences are parametrized by $s\in\R\setminus\{0,1\}$ and $r \eqdef s / (s-1)$. When $s<1$ it correspond to
\begin{align*}
\phi(p) = \frac{\rho}{s(s-1)}\big(p^s -s(p-1) -1 \big) \qandq 
\phi^*(q) = \rho\frac{r-1}{r}\left[\big(1 + \frac{q}{\rho(r-1)}\big)^r -1\right].
\end{align*}
Special cases include \emph{Hellinger} with $s=1/2$, and \emph{Berg entropy} as the limit case $s=0$, defined by $\phi(p) = \rho(p - 1 - \log p)$ and $\phi^*(q) = - \rho\log( 1 - q / \rho)$ with $\text{dom}(\phi^*) = (-\infty, \rho)$. 
Kullback-Leibler is the limit $s=1$. 
We refer to~\cite{liero2015optimal} for more details.
The following proposition summarizes important properties of this divergence needed for the analysis of Sinkhorn interates.

\begin{proposition}[Properties of the power entropy]\label{prop-aprox-power-ent}
	For any dual exponent $r<1$, $\phi^*$ is strictly convex and $\partial\phi^*(x)\rightarrow 0$ when $x\rightarrow -\infty$. The proximal operator satisfies
	\begin{align}\label{eq-lambert}
	\aprox{\phi^*}(p) = \rho(1-r) - \epsilon (1-r) W\left(\frac{\rho}{\epsilon} \exp\big(\frac{-p + \rho(1-r)}{\epsilon(1-r)}\big)\right),
	\end{align}
	where $W$ is the Lambert function, which satisfies for any $p\in\R_+$ $W(p)e^{W(p)} = p$, see~\cite{corless1996lambertw}.
	It is a non-expansive operator, and it is a contraction on compact sets.
\end{proposition}

\begin{proof}
	The strict convexity and the limit of the gradient is immediate.
	For any input $p$, $q = \aprox{\phi^*}(p)$ satisfies
	\begin{align*}
	& e^{\tfrac{p-q}{\epsilon}} = (1 - \tfrac{q}{\rho(1-r)})^{r-1}\\
	&\Leftrightarrow \quad  \epsilon(1-r) \log(1 - \tfrac{q}{\rho(1-r)}) + p - q = 0 \\
	&\Leftrightarrow \quad \epsilon(1-r)\log(Q) + \rho(1-r) Q + (p -\rho(1-r)) = 0 \text{  with  } Q = 1 - \tfrac{q}{\rho(1-r)}\\
	&\Leftrightarrow \quad q(p) = \rho(1-r) - \epsilon (1-r) W(\Delta(p))
	\text{  with  } \Delta(p) = \tfrac{\rho}{\epsilon}e^{\frac{-p + \rho(1-r)}{\epsilon(1-r)}}.
	\end{align*}
	We now show that the above mapping is indeed 1-Lipschitz. We first note that $\frac{\d\Delta}{\d p} = -\frac{\Delta(p)}{(1-r)\epsilon}$. The derivative of the Lambert function gives
	\begin{align*}
	\frac{\d q(p)}{\d p} = -\epsilon(1-r) \frac{\d\Delta}{\d p}\frac{\d W}{\d \Delta} = \epsilon(1-r) \frac{\Delta}{\epsilon(1-r)} \frac{W(\Delta)}{\Delta(1 + W(\Delta))}= \frac{W(\Delta)}{(1 + W(\Delta))} <1.
	\end{align*}
	Because $\Delta(p) >0$ we have $W(\Delta)\geq 0$, and $W(\Delta)\rightarrow +\infty$ when $p\rightarrow -\infty$. 
	Thus $\aprox{\phi^*}$ is 1-lipschitz and contractive when iterations are restricted to a compact set.
\end{proof}

Note that Formula~\eqref{eq-lambert} enables a fast evaluation of $\aprox{\phi^*}$. 
The Lambert function is computable via the cubically converging Halley's algorithm~\cite{alefeld1981convergence}. 
It is also computable on GPU devices.

\begin{figure}[p]
	\begin{minipage}{\textwidth}
	\centering
	\resizebox{!}{5cm}{\begin{tikzpicture}

    \pgfplotstableread{
    p balanced range tv
    -3.000 -3.000 -2.307 -1.000
    -2.950 -2.950 -2.257 -1.000
    -2.900 -2.900 -2.207 -1.000
    -2.850 -2.850 -2.157 -1.000
    -2.800 -2.800 -2.107 -1.000
    -2.750 -2.750 -2.057 -1.000
    -2.700 -2.700 -2.007 -1.000
    -2.650 -2.650 -1.957 -1.000
    -2.600 -2.600 -1.907 -1.000
    -2.550 -2.550 -1.857 -1.000
    -2.500 -2.500 -1.807 -1.000
    -2.450 -2.450 -1.757 -1.000
    -2.400 -2.400 -1.707 -1.000
    -2.350 -2.350 -1.657 -1.000
    -2.300 -2.300 -1.607 -1.000
    -2.250 -2.250 -1.557 -1.000
    -2.200 -2.200 -1.507 -1.000
    -2.150 -2.150 -1.457 -1.000
    -2.100 -2.100 -1.407 -1.000
    -2.050 -2.050 -1.357 -1.000
    -2.000 -2.000 -1.307 -1.000
    -1.950 -1.950 -1.257 -1.000
    -1.900 -1.900 -1.207 -1.000
    -1.850 -1.850 -1.157 -1.000
    -1.800 -1.800 -1.107 -1.000
    -1.750 -1.750 -1.057 -1.000
    -1.700 -1.700 -1.007 -1.000
    -1.650 -1.650 -0.957 -1.000
    -1.600 -1.600 -0.907 -1.000
    -1.550 -1.550 -0.857 -1.000
    -1.500 -1.500 -0.807 -1.000
    -1.450 -1.450 -0.757 -1.000
    -1.400 -1.400 -0.707 -1.000
    -1.350 -1.350 -0.657 -1.000
    -1.300 -1.300 -0.607 -1.000
    -1.250 -1.250 -0.557 -1.000
    -1.200 -1.200 -0.507 -1.000
    -1.150 -1.150 -0.457 -1.000
    -1.100 -1.100 -0.407 -1.000
    -1.050 -1.050 -0.357 -1.000
    -1.000 -1.000 -0.307 -1.000
    -0.950 -0.950 -0.257 -0.950
    -0.900 -0.900 -0.207 -0.900
    -0.850 -0.850 -0.157 -0.850
    -0.800 -0.800 -0.107 -0.800
    -0.750 -0.750 -0.057 -0.750
    -0.700 -0.700 -0.007 -0.700
    -0.650 -0.650 0.000 -0.650
    -0.600 -0.600 0.000 -0.600
    -0.550 -0.550 0.000 -0.550
    -0.500 -0.500 0.000 -0.500
    -0.450 -0.450 0.000 -0.450
    -0.400 -0.400 0.000 -0.400
    -0.350 -0.350 0.000 -0.350
    -0.300 -0.300 0.000 -0.300
    -0.250 -0.250 0.000 -0.250
    -0.200 -0.200 0.000 -0.200
    -0.150 -0.150 0.000 -0.150
    -0.100 -0.100 0.000 -0.100
    -0.050 -0.050 0.000 -0.050
    0.000 0.000 0.000 0.000
    0.050 0.050 0.000 0.050
    0.100 0.100 0.000 0.100
    0.150 0.150 0.000 0.150
    0.200 0.200 0.000 0.200
    0.250 0.250 0.000 0.250
    0.300 0.300 0.000 0.300
    0.350 0.350 0.000 0.350
    0.400 0.400 0.000 0.400
    0.450 0.450 0.045 0.450
    0.500 0.500 0.095 0.500
    0.550 0.550 0.145 0.550
    0.600 0.600 0.195 0.600
    0.650 0.650 0.245 0.650
    0.700 0.700 0.295 0.700
    0.750 0.750 0.345 0.750
    0.800 0.800 0.395 0.800
    0.850 0.850 0.445 0.850
    0.900 0.900 0.495 0.900
    0.950 0.950 0.545 0.950
    1.000 1.000 0.595 1.000
    1.050 1.050 0.645 1.000
    1.100 1.100 0.695 1.000
    1.150 1.150 0.745 1.000
    1.200 1.200 0.795 1.000
    1.250 1.250 0.845 1.000
    1.300 1.300 0.895 1.000
    1.350 1.350 0.945 1.000
    1.400 1.400 0.995 1.000
    1.450 1.450 1.045 1.000
    1.500 1.500 1.095 1.000
    1.550 1.550 1.145 1.000
    1.600 1.600 1.195 1.000
    1.650 1.650 1.245 1.000
    1.700 1.700 1.295 1.000
    1.750 1.750 1.345 1.000
    1.800 1.800 1.395 1.000
    1.850 1.850 1.445 1.000
    1.900 1.900 1.495 1.000
    1.950 1.950 1.545 1.000
    2.000 2.000 1.595 1.000
    2.050 2.050 1.645 1.000
    2.100 2.100 1.695 1.000
    2.150 2.150 1.745 1.000
    2.200 2.200 1.795 1.000
    2.250 2.250 1.845 1.000
    2.300 2.300 1.895 1.000
    2.350 2.350 1.945 1.000
    2.400 2.400 1.995 1.000
    2.450 2.450 2.045 1.000
    2.500 2.500 2.095 1.000
    2.550 2.550 2.145 1.000
    2.600 2.600 2.195 1.000
    2.650 2.650 2.245 1.000
    2.700 2.700 2.295 1.000
    2.750 2.750 2.345 1.000
    2.800 2.800 2.395 1.000
    2.850 2.850 2.445 1.000
    2.900 2.900 2.495 1.000
    2.950 2.950 2.545 1.000
    3.000 3.000 2.595 1.000
    }\datatable

  \begin{axis}[width=.5\textwidth,
              grid=both,ymin=-3, ymax=3, xmin=-3, xmax=3,
              xlabel={\scriptsize $p$}, 
              ylabel={\scriptsize $-\aprox{\phi^*}(-p)$}, 
              xtick={-3,-2,-1,0,1,2,3},
              ytick={-3,-2,-1,0,1,2,3},
              legend pos = south east,
              grid=major,
            legend cell align={left},
            axis background/.style={fill=white},
            label style={font=\tiny},
            tick label style={font=\tiny},
            axis equal image,]
      \addplot[red!80, very thick,] table[x=p,y=balanced]  {\datatable};
      \addplot[blue!50, very thick, dashed] table[x=p,y=range]  {\datatable};
      \addplot[green!80!black, very thick, dotted] table[x=p,y=tv]  {\datatable};
      \addlegendentry{{\small Balanced}}
      \addlegendentry{{\small Range [0.5, 1.5]}}
      \addlegendentry{{\small TV}}
  \end{axis}
  \end{tikzpicture}}
	\quad
	\resizebox{!}{5cm}{\begin{tikzpicture}

    \pgfplotstableread{
    p balanced range tv
    -3.000 -1.500 -1.191 -0.880
    -2.950 -1.475 -1.176 -0.875
    -2.900 -1.450 -1.161 -0.869
    -2.850 -1.425 -1.147 -0.863
    -2.800 -1.400 -1.132 -0.857
    -2.750 -1.375 -1.116 -0.850
    -2.700 -1.350 -1.101 -0.844
    -2.650 -1.325 -1.085 -0.837
    -2.600 -1.300 -1.070 -0.830
    -2.550 -1.275 -1.054 -0.822
    -2.500 -1.250 -1.037 -0.815
    -2.450 -1.225 -1.021 -0.807
    -2.400 -1.200 -1.005 -0.798
    -2.350 -1.175 -0.988 -0.790
    -2.300 -1.150 -0.971 -0.781
    -2.250 -1.125 -0.954 -0.772
    -2.200 -1.100 -0.937 -0.762
    -2.150 -1.075 -0.919 -0.753
    -2.100 -1.050 -0.902 -0.743
    -2.050 -1.025 -0.884 -0.732
    -2.000 -1.000 -0.866 -0.722
    -1.950 -0.975 -0.848 -0.710
    -1.900 -0.950 -0.829 -0.699
    -1.850 -0.925 -0.811 -0.687
    -1.800 -0.900 -0.792 -0.675
    -1.750 -0.875 -0.773 -0.663
    -1.700 -0.850 -0.754 -0.650
    -1.650 -0.825 -0.735 -0.637
    -1.600 -0.800 -0.715 -0.623
    -1.550 -0.775 -0.695 -0.610
    -1.500 -0.750 -0.676 -0.595
    -1.450 -0.725 -0.656 -0.581
    -1.400 -0.700 -0.635 -0.566
    -1.350 -0.675 -0.615 -0.550
    -1.300 -0.650 -0.594 -0.535
    -1.250 -0.625 -0.574 -0.519
    -1.200 -0.600 -0.553 -0.502
    -1.150 -0.575 -0.532 -0.485
    -1.100 -0.550 -0.511 -0.468
    -1.050 -0.525 -0.489 -0.451
    -1.000 -0.500 -0.468 -0.433
    -0.950 -0.475 -0.446 -0.415
    -0.900 -0.450 -0.424 -0.396
    -0.850 -0.425 -0.402 -0.377
    -0.800 -0.400 -0.379 -0.358
    -0.750 -0.375 -0.357 -0.338
    -0.700 -0.350 -0.334 -0.318
    -0.650 -0.325 -0.311 -0.297
    -0.600 -0.300 -0.288 -0.276
    -0.550 -0.275 -0.265 -0.255
    -0.500 -0.250 -0.242 -0.234
    -0.450 -0.225 -0.219 -0.212
    -0.400 -0.200 -0.195 -0.190
    -0.350 -0.175 -0.171 -0.167
    -0.300 -0.150 -0.147 -0.144
    -0.250 -0.125 -0.123 -0.121
    -0.200 -0.100 -0.099 -0.097
    -0.150 -0.075 -0.074 -0.074
    -0.100 -0.050 -0.050 -0.049
    -0.050 -0.025 -0.025 -0.025
    0.000 0.000 0.000 0.000
    0.050 0.025 0.025 0.025
    0.100 0.050 0.050 0.051
    0.150 0.075 0.076 0.076
    0.200 0.100 0.101 0.102
    0.250 0.125 0.127 0.129
    0.300 0.150 0.153 0.155
    0.350 0.175 0.179 0.182
    0.400 0.200 0.205 0.210
    0.450 0.225 0.231 0.237
    0.500 0.250 0.258 0.265
    0.550 0.275 0.284 0.293
    0.600 0.300 0.311 0.321
    0.650 0.325 0.338 0.350
    0.700 0.350 0.365 0.379
    0.750 0.375 0.392 0.408
    0.800 0.400 0.419 0.437
    0.850 0.425 0.447 0.467
    0.900 0.450 0.474 0.497
    0.950 0.475 0.502 0.527
    1.000 0.500 0.530 0.557
    1.050 0.525 0.558 0.588
    1.100 0.550 0.586 0.619
    1.150 0.575 0.614 0.650
    1.200 0.600 0.643 0.681
    1.250 0.625 0.671 0.712
    1.300 0.650 0.700 0.744
    1.350 0.675 0.729 0.776
    1.400 0.700 0.758 0.808
    1.450 0.725 0.787 0.840
    1.500 0.750 0.816 0.873
    1.550 0.775 0.845 0.905
    1.600 0.800 0.875 0.938
    1.650 0.825 0.904 0.971
    1.700 0.850 0.934 1.005
    1.750 0.875 0.964 1.038
    1.800 0.900 0.993 1.072
    1.850 0.925 1.023 1.105
    1.900 0.950 1.054 1.139
    1.950 0.975 1.084 1.174
    2.000 1.000 1.114 1.208
    2.050 1.025 1.145 1.242
    2.100 1.050 1.175 1.277
    2.150 1.075 1.206 1.312
    2.200 1.100 1.237 1.347
    2.250 1.125 1.268 1.382
    2.300 1.150 1.299 1.417
    2.350 1.175 1.330 1.453
    2.400 1.200 1.362 1.488
    2.450 1.225 1.393 1.524
    2.500 1.250 1.424 1.560
    2.550 1.275 1.456 1.596
    2.600 1.300 1.488 1.632
    2.650 1.325 1.520 1.668
    2.700 1.350 1.552 1.705
    2.750 1.375 1.584 1.741
    2.800 1.400 1.616 1.778
    2.850 1.425 1.648 1.815
    2.900 1.450 1.680 1.852
    2.950 1.475 1.713 1.889
    3.000 1.500 1.745 1.926
    }\datatable

  \begin{axis}[width=.5\textwidth,grid=both,ymin=-3, ymax=3, xmin=-3, xmax=3,
              xlabel={\scriptsize $p$}, 
              xtick={-3,-2,-1,0,1,2,3},
              ytick={-3,-2,-1,0,1,2,3},
              legend pos = south east,
              grid=major,
            legend cell align={left},
            axis background/.style={fill=white},
            label style={font=\tiny},
            tick label style={font=\tiny},
            axis equal image,]
      \addplot[red!80, very thick,] table[x=p,y=balanced]  {\datatable};
      \addplot[blue!50, very thick, dashed] table[x=p,y=range]  {\datatable};
      \addplot[green!80!black, very thick, dotted] table[x=p,y=tv]  {\datatable};
      \addlegendentry{{\small KL}}
      \addlegendentry{{\small Hellinger}}
      \addlegendentry{{\small Berg}}
  \end{axis}
  \end{tikzpicture}}
	\caption{Display of the 1-Lipschitz operator $p \mapsto -\aprox{\phi^*}(-p)$ in the six major settings of Table~\ref{tab-entropies-aprox}, using $\epsilon=1$ and $\rho=1$.}
	\label{fig-aprox}
	\vspace*{.7cm}
	
	\newcommand{\myfigE}[1]{\includegraphics[height=.14\linewidth]{images/compare_entropy/comparison_entropy_#1-c}}

	\centering
	\begin{tabular}{c@{\hspace{1mm}}c@{\hspace{1mm}}c@{\hspace{1mm}}}%
		\myfigE{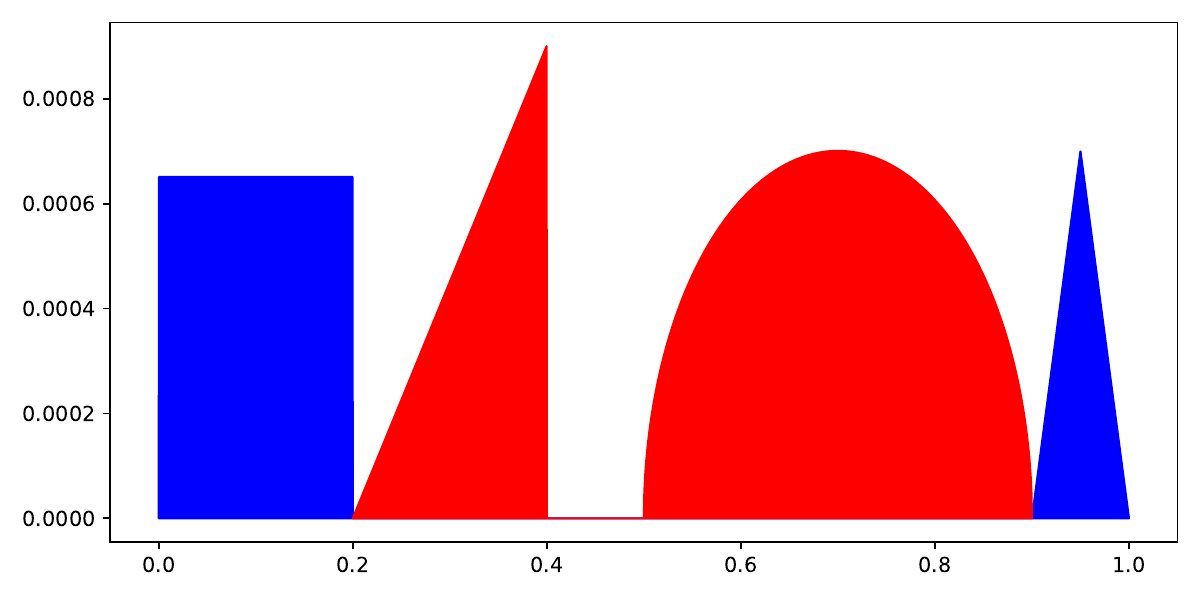} &
		\myfigE{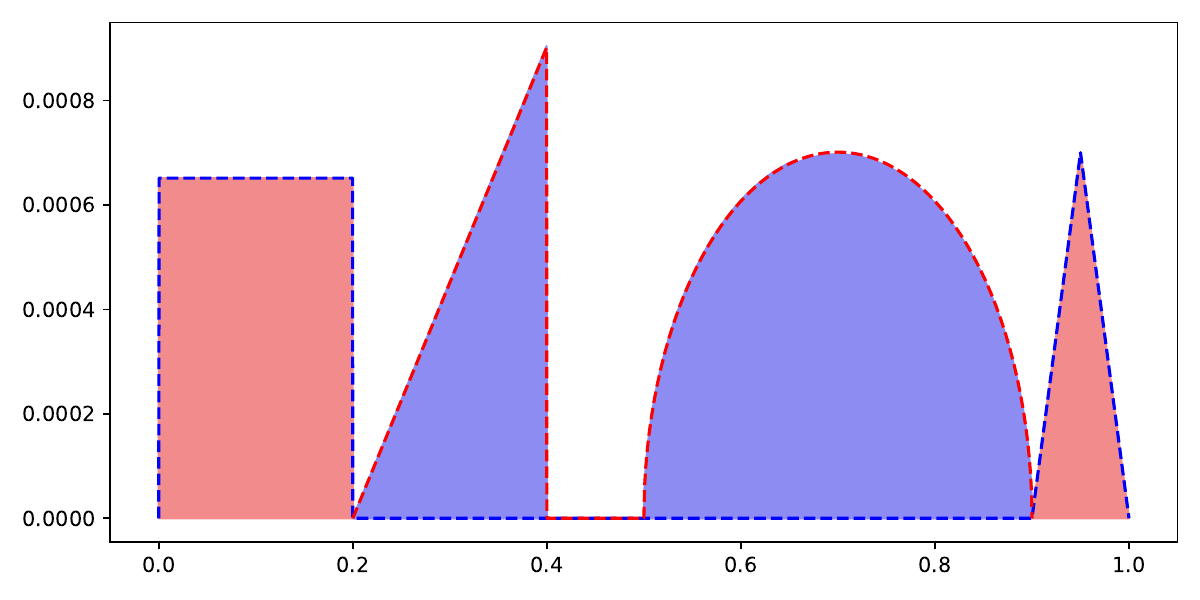} &
		\myfigE{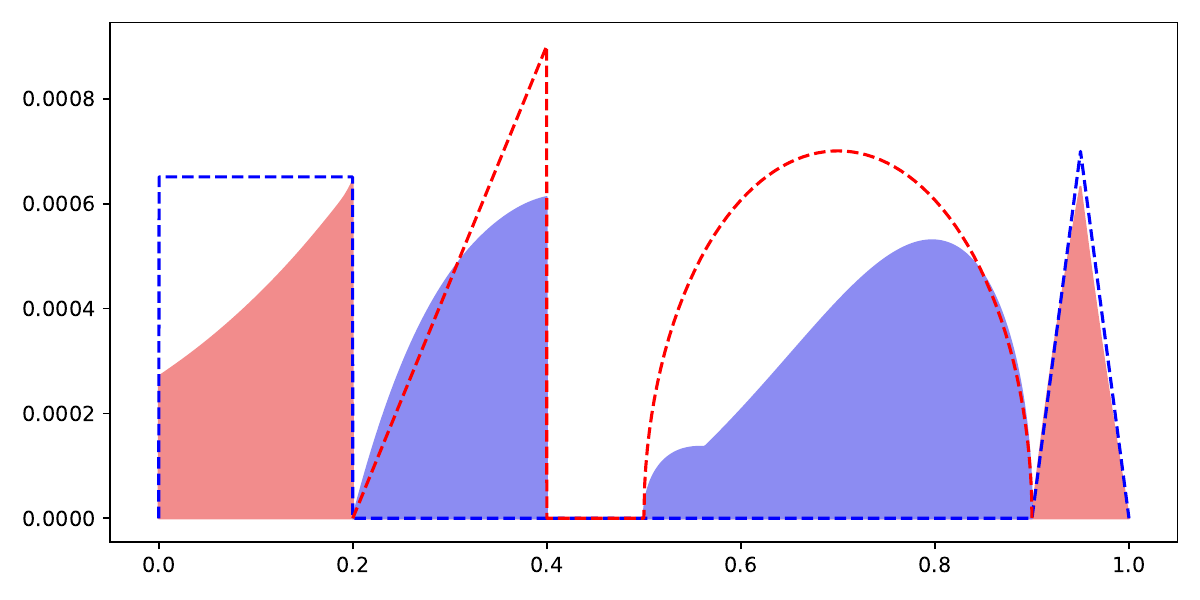} \\[-1mm]
		Inputs $(\al,\be)$ & $\iota_{(=)}$ & $10^{-1} * \KL$ \\[1mm]
		\myfigE{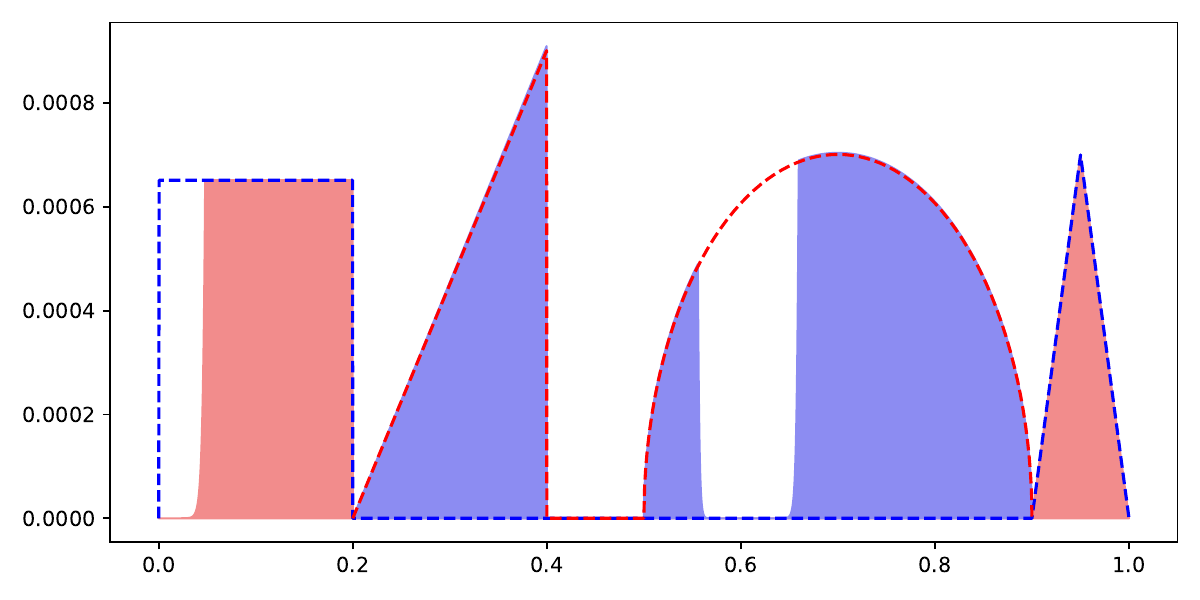} &
		\myfigE{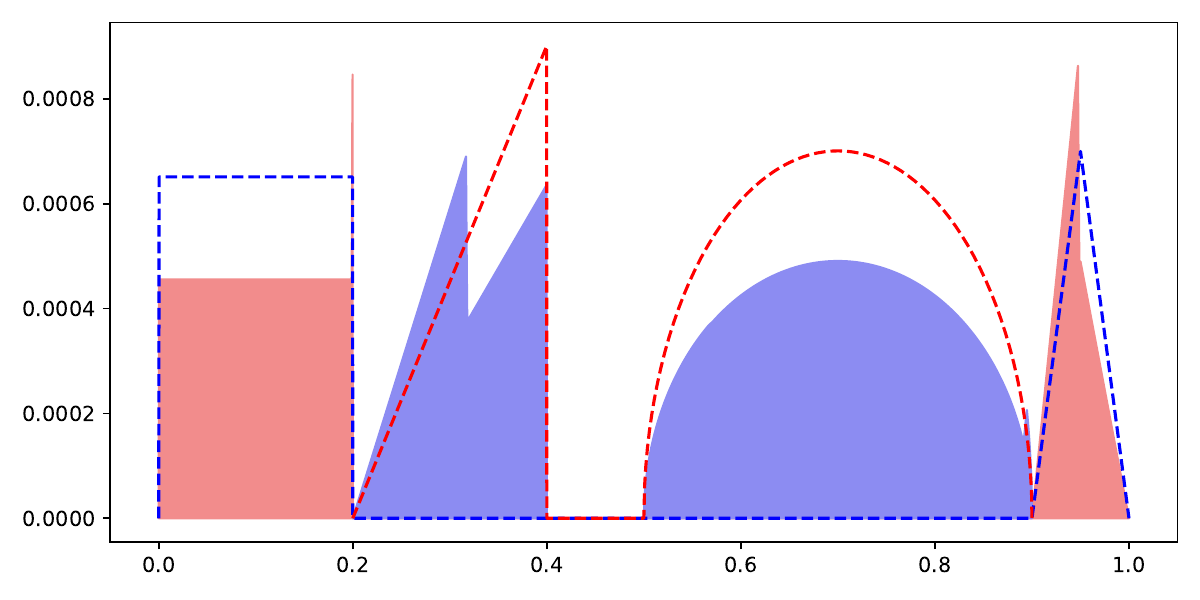} &
		\myfigE{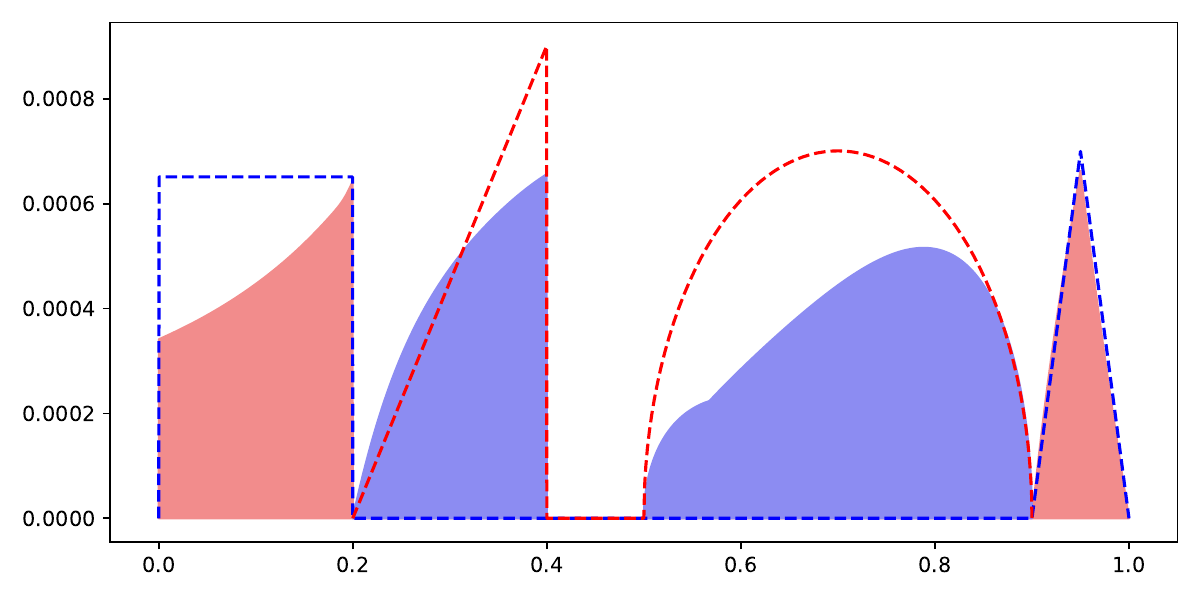} \\[-1mm]
		$10^{-1} * \TV$ & $\RG_{[0.7, 1.3]}$ & $10^{-1} * \text{Berg}$
	\end{tabular}
	\caption{Display of optimal marginals $(\textcolor{myred2}{\pi_1},\textcolor{myblue2}{\pi_2})$ depending on $\phi$. 
		The inputs $(\textcolor{myblue1}{\al},\textcolor{myred1}{\be})$ are 1D.
		Measures $(\textcolor{myblue1}{\al},\textcolor{myred1}{\be})$ and $(\textcolor{myred2}{\pi_1},\textcolor{myblue2}{\pi_2})$ are respectively plotted as dashed lines and filled colorings. 
		We use a regularization $\sqrt{\epsilon}=\sqrt{10^{-3}}$ on $[0,1]$.}
	\label{fig-comp-ent}
	\vspace*{.7cm}
	
	\newcommand{\myfig}[1]{\includegraphics[width=.23\linewidth]{images/compare_reach/comparison_#1-c}}
	\centering
	\begin{tabular}{@{}c@{\hspace{1mm}}c@{\hspace{1mm}}c@{\hspace{1mm}}c@{\hspace{1mm}}c@{}}
		\rotatebox{90}{\small $\D_\phi=\rho\KL$} & \myfig{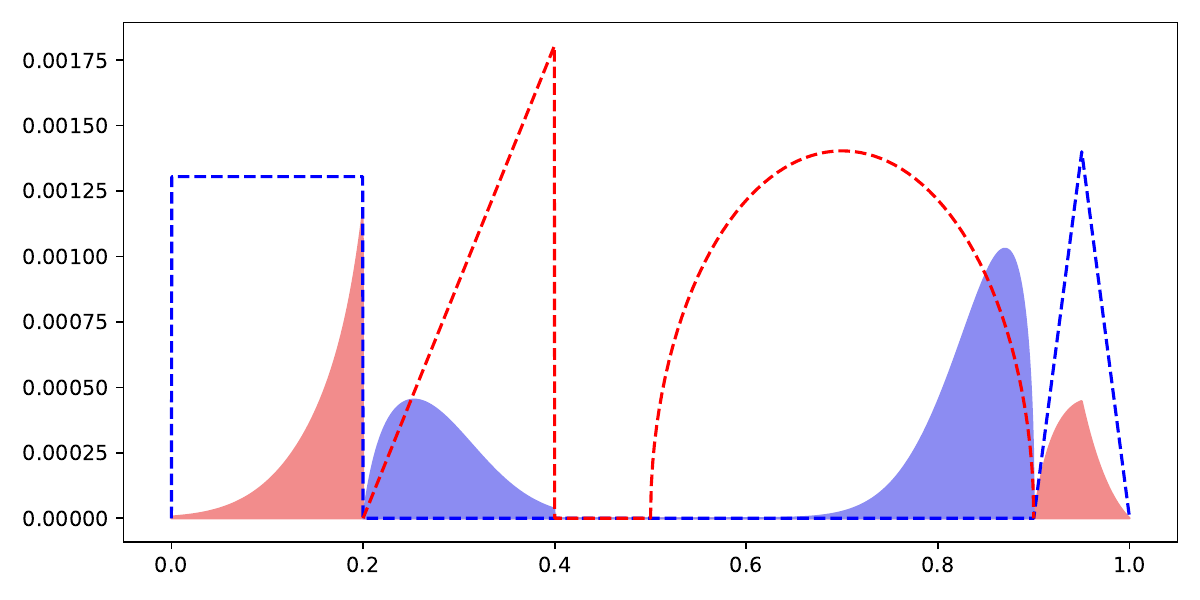} & \myfig{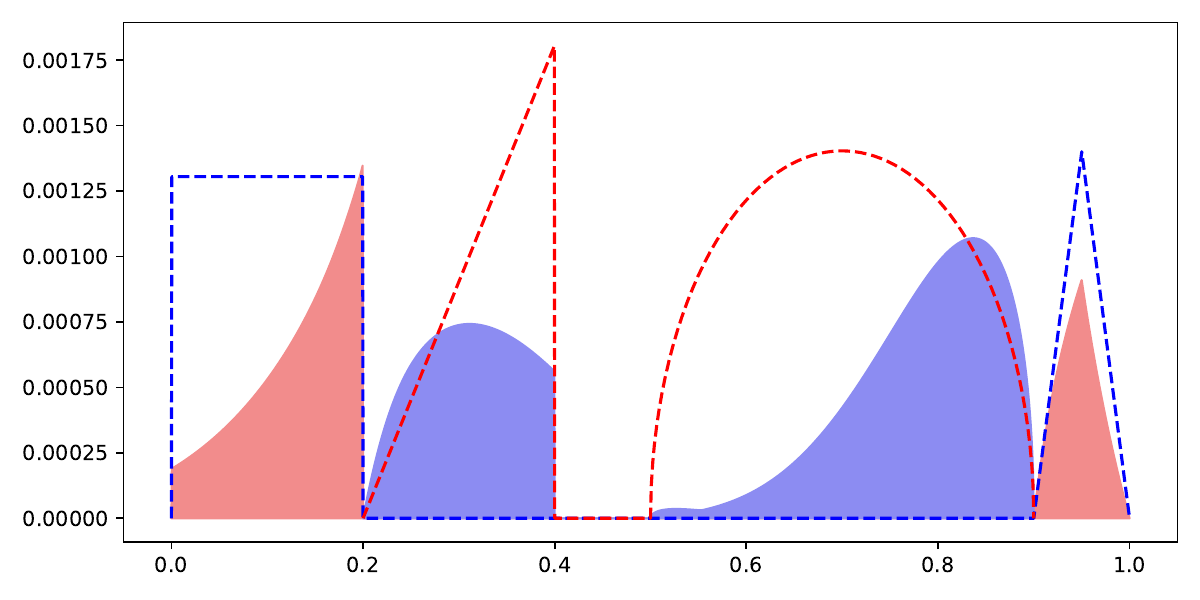} & \myfig{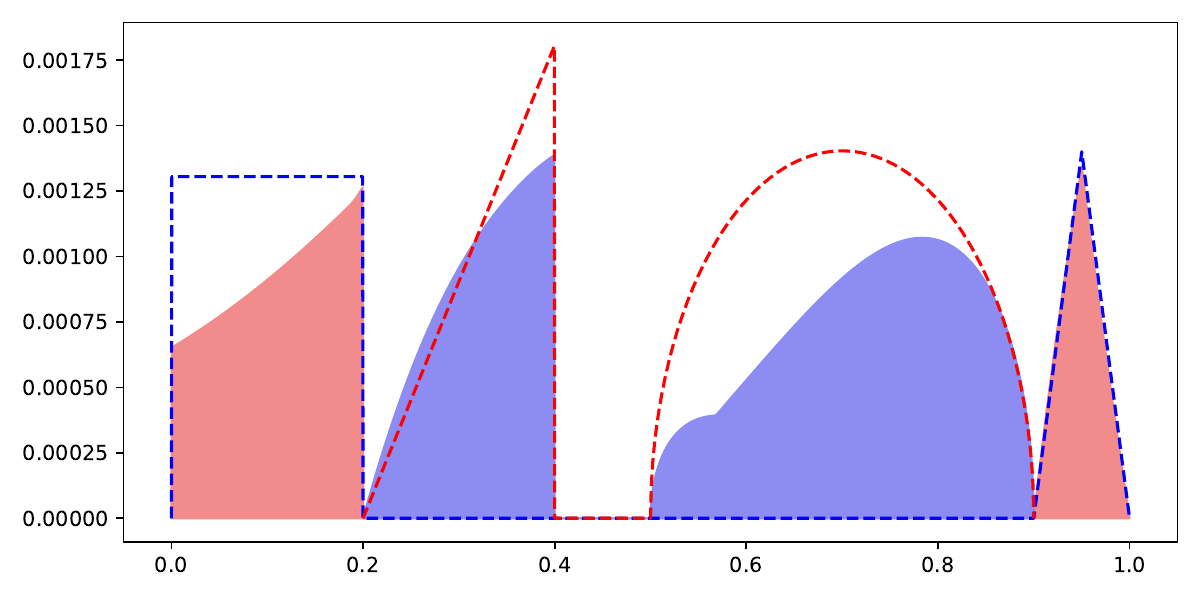} & \myfig{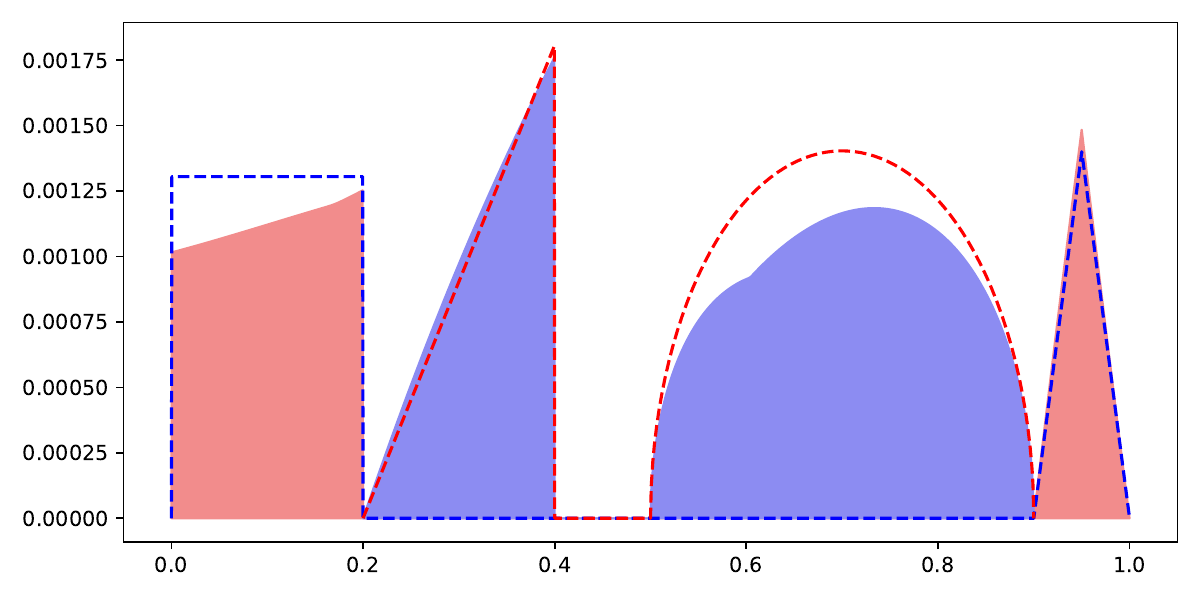} \\
		\rotatebox{90}{\small $\D_\phi=\rho\TV$\;\;} & \myfig{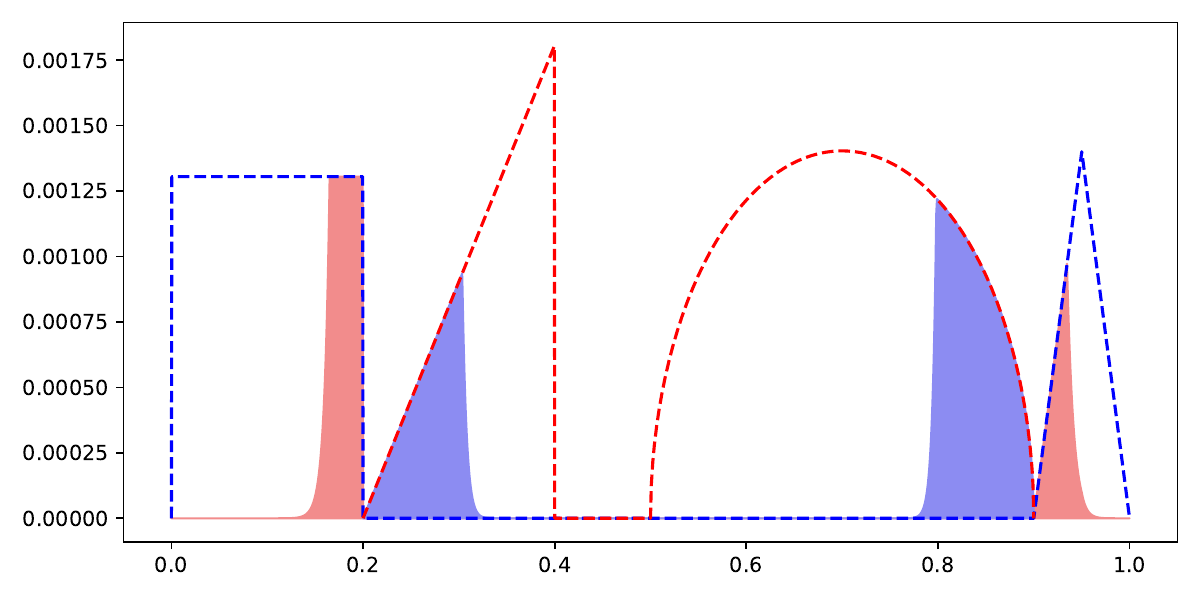} & \myfig{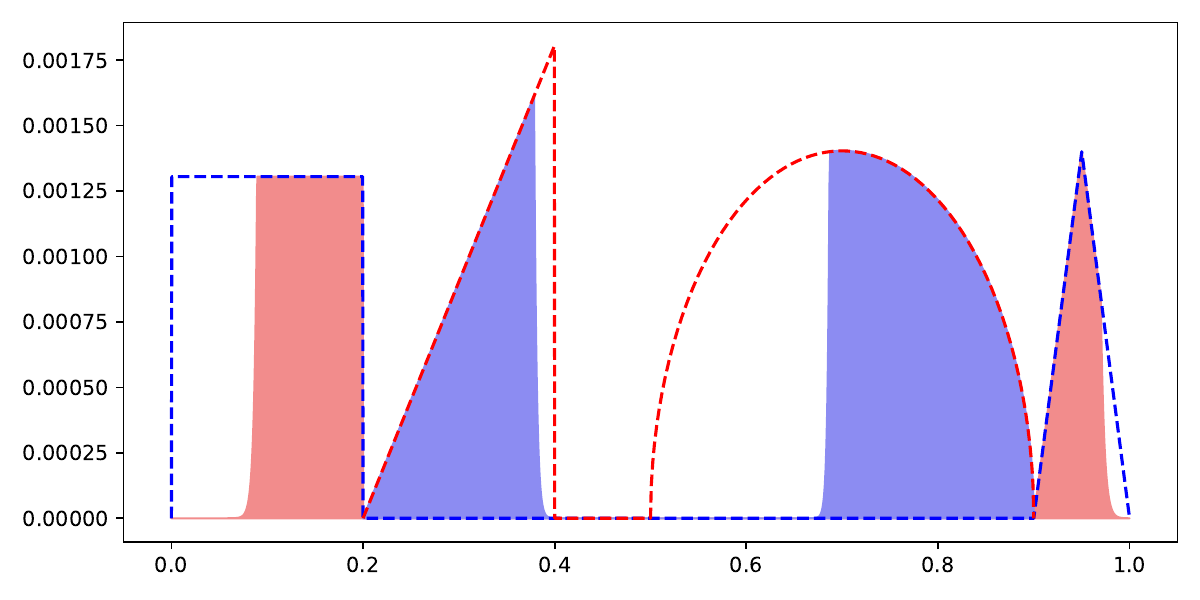} & \myfig{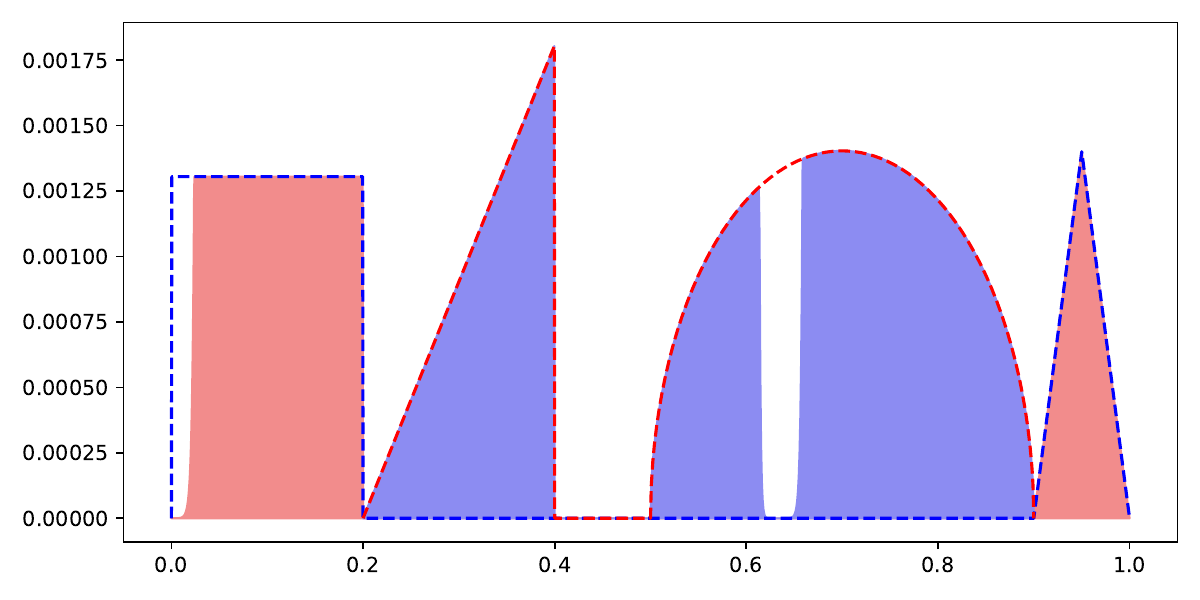} & \myfig{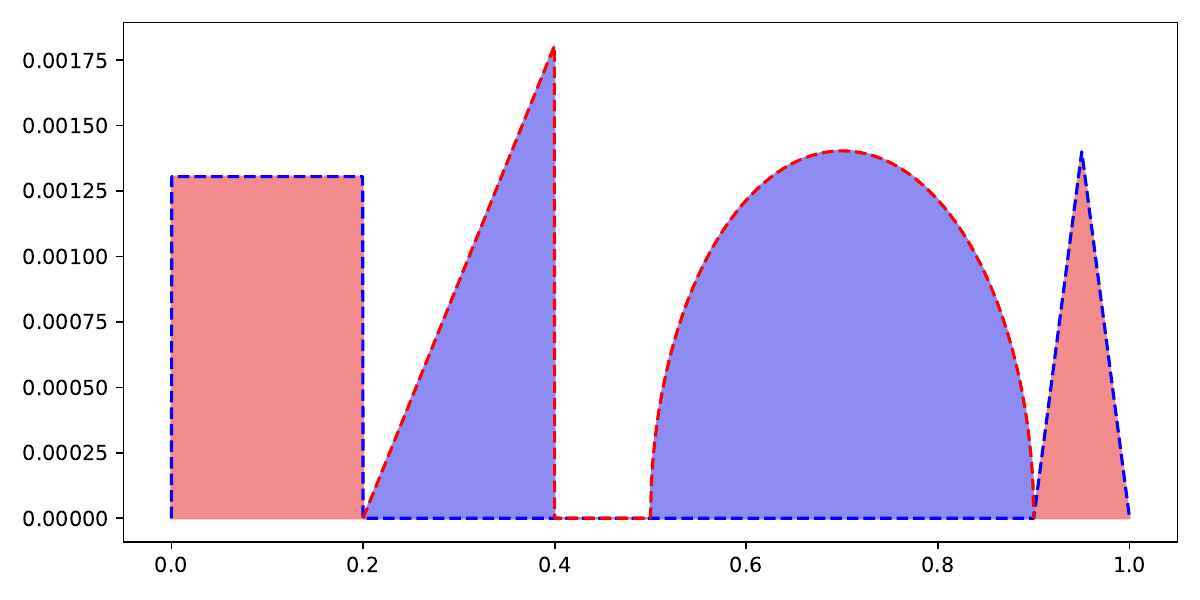} \\
		& $\rho=0.01$ & $\rho=0.03$ & $\rho=0.13$ & $\rho=0.5$
	\end{tabular}
	\caption{Display of marginals $(\textcolor{myred2}{\pi_1},\textcolor{myblue2}{\pi_2})$ depending on parameter $\rho$.
		We use the same inputs $(\textcolor{myblue1}{\al},\textcolor{myred1}{\be})$ from Figure~\ref{fig-comp-ent}. 
		First line corresponds to $\rho\KL$ and the second to $\rho\TV$.}
	\label{fig-impact-reach}
\end{minipage}
\end{figure}

\paragraph{Overview.}
	We give an informal interpretation of Sinkhorn iterations for different divergences based on Proposition~\ref{prop-optimality-prox}, to illustrate the role of $\aprox{\phi^*}$.
	Optimality conditions have a compositional structure.
	Operators $(\Ss_\al,\Ss_\be)$ characterize optimal \emph{balanced} potentials as fixed points, and $\aprox{\phi^*}$ updates such fixed point by \emph{saturating} ($\D_\phi=\TV$) or \emph{dampening} ($\D_\phi=\KL$) dual potentials, see Figure~\ref{fig-aprox}.
	It indirectly impacts the plan via Equation~\eqref{eq-implicit-plan} by blocking or reducing transportation, see Figure~\ref{fig-impact-reach}.

Figure~\ref{fig-comp-ent} displays the impact of $\phi$ on the optimal plan $\pi$.
Here marginals $(\pi_1, \pi_2)$ are compared to the input marginals $(\al,\be)$.
Informally speaking, $\TV$ has 'sharp' marginals, i.e. it either transport s.t. $\pi_1(x)=\al(x)$ or destroys mass s.t. $\pi_1(x)=0$.
Marginals with $\KL$ are 'smooth' in the sense that it progressively transitions between transportation and destruction as $\C(x,y)$ increases.
Marginals for $\RG_{[a,b]}$ are less interpretable due to the box constraint, but we see that $\tfrac{\d\pi_1}{\d\al}\in\{a,b\}$.
The result of Berg entropy is similar to $\KL$, probably because they are both power entropies.
%

Figure~\ref{fig-impact-reach} shows the impact of the parameter $\rho$ on $(\pi_1,\pi_2)$ (see Remark~\ref{rem-param-rho}). 
It illustrates that $\rho$ acts as a characteristic radius beyond which it is preferable to destroy mass than transport it. 
This phenomenon is sharp in the case of $\TV$ (it is known when $\epsilon=0$ that $\spt(\pi)\subset\{(x,y), \C(x,y)\leq 2\rho\}$) while there is a smooth dampening as $\C$ increases for $\KL$.

%

\begin{table}
	\centering
	\def\arraystretch{1.2}
	\begin{adjustbox}{center}
		\begin{tabular}{|c@{}c@{\hspace{2mm}}c@{}c|}
			\hline
			Setting & Parameters & $\phi(p)$ & $-\aprox{\phi^*}(-p)$\\
			\hline
			Balanced & None & $0$ if $p=1$, $+\infty$ otherwise & $p$ \\
			Range & $0 \leq a \leq 1 \leq b$ & $0$ if $p \in [a, b]$, $+\infty$ otherwise &
			$\text{Soft-Thresh}_{\epsilon \log a}^{\epsilon \log b}(p)$ \\
			TV & $\rho > 0$ & $\rho\, |p-1|$ & $\text{Clamp}_{[-\rho, +\rho]}(p)$  \\
			KL & $\rho > 0$ & $\rho\, (p \log p - p + 1)$ & $\tfrac{\rho}{\rho+\epsilon}\, p$ \\
			Hellinger & $\rho > 0$ & $4 \rho\, (1 + (p-1)/2 - \sqrt{p})$ &  $2 \epsilon W(\tfrac{\rho}{\epsilon} \exp(\tfrac{\rho+p/2}{\epsilon})) - 2\rho $ \\
			Berg & $\rho > 0$ & $\rho\, (p - 1 - \log p)$ & $\epsilon W(\tfrac{\rho}{\epsilon} \exp(\tfrac{\rho+p}{\epsilon})) - \rho$ \\
			\hline
		\end{tabular}
	\end{adjustbox}
	\caption{Summary of the information that is required to implement
		the generalized Sinkhorn algorithm in six common settings.
	}
	\label{tab-entropies-aprox}
\end{table}

\subsection{Convergence analysis and compactness of potentials}
\label{sec-convergence}

For discrete measures, alternate maximization is known to converge to maximizers for smooth problems~\cite{tseng2001convergence}, 
but convergence speeds known in the litterature depend on the number of samples.
Until now, there was no proof for general (continuous) measures.
In this section, we work over the \emph{infinite dimensional} space $\Mmp(\Xx)$ to overcome these limitations.
We prove linear convergence of the unbalanced Sinkhorn algoritm in full generality in Theorem~\ref{thm-cv-sink-compact} which is the main result of this section.

\subsubsection{General convergence result}

Theorem~\ref{thm-cv-sink-compact} states convergence of Sinkhorn iterates, provided they remain in a compact subset of $\Cc(\Xx)^2$.
We then prove that this compactness hypothesis holds in a variety of settings, including Section~\ref{sec-exmp-f-div}.
A first setting studied in Section~\ref{subsubsec-convergence-compact} assumes $\phi^*$ is strictly convex, and holds in wide generality.
The settings of balanced OT, TV and Range are convex but not strictly. They are treated separately in Section~\ref{sec-compact-balanced-tv-range}.

\begin{theorem}[The Sinkhorn algorithm solves the $\OTb$ problem]\label{thm-cv-sink-compact}
If the cost $\C$ is $\gamma$-Lipschitz, and if the dual program~\eqref{eq-dual-unb} can be restricted to a compact subset of $\Cc(\Xx)^2$, then there exists an optimal pair of dual potentials and the Sinkhorn algorithm converges towards a pair of optimal potentials.
In particular we have convergence for all settings of Section~\ref{sec-exmp-f-div}.
\end{theorem}

\begin{proof}
Consider a sequence $(\f_n,\g_n)_n$ approaching $\OT(\al,\be)=\sup\Ff$. 
Compactness in $\Cc(\Xx)$ allows to extract $(\f_{n_k},\g_{n_k})\rightarrow(\f,\g)$, where $(\f,\g)\in\Cc(\Xx)^2$ are optimal, i.e. $\OT(\al,\be)=\Ff(\f,\g)$.

Now write $(\f_t,\g_t)$ the Sinkhorn iterates~\eqref{eq-sinkhorn-iter-1} for some $\f_0\in\Cc(\Xx)$.
Iterates $(\f_t,\g_t)$ are $\gamma$-Lipschitz (Proposition~\ref{lem-smin-cost-regular}), thus equicontinuous on $\Xx$. 
Furthermore, non-expansivity of $\Aa\Ss$ (Propositions~\ref{lem-smin-lipschitz-func} and~\ref{prop-nonexp}) implies $\norm{\f_t - \f}_\infty\leq\norm{\f_{0} - \f}_\infty$.
Since $\f\in\Cc(\Xx)$ and $\Xx$ compact, then $\norm{\f}_\infty < \infty$. 
Thus $\norm{\f_t}_\infty \leq\norm{\f_{0} - \f}_\infty + \norm{\f}_\infty$.

Ascoli-Arzela theorem holds and Sinkhorn iterates $(\f_t,\g_t)_t$ are a compact sequence in $\Cc(\Xx)$.
Take any subsequence $\f_{t_k}\rightarrow \f_*$, and $\eta >0$. There exists $k$ such that $\norm{\f_{k} - \f_*}_\infty < \eta$. 
Non-expansivity of $\Aa\Ss$ implies again that $\forall t\geq k,\, \norm{\f_{t} - \f_*}_\infty\leq\norm{\f_{k} - \f_*}_\infty < \eta$.
The same fact holds for $(\g_t)$
This inequality is the definition of the convergence of $(\f_t)$.
Thus any subsequence verifies $\f_{t_k}\rightarrow\f_*$ and then $\f_t\rightarrow\f_*$. 
Thus Sinkhorn iterates converge towards $(\f_*,\g_*)$ and are fixed point of the Sinkhorn maps.
Thus they are optimal (Proposition~\ref{prop-cns-optimality}).

Thanks to Lemmas~(\ref{lem-compact-dual},\ref{lem-compact-balanced},\ref{lem-compact-tv},\ref{lem-compact-range}), we can restrict Problem~\eqref{eq-dual-unb} to a compact set, hence the convergence for all settings of Section~\ref{sec-exmp-f-div}.
\end{proof}

Theorem~\ref{thm-cv-sink-compact} reduces proofs of \emph{convergence} to proofs of \emph{compactness} of the sequence $(\f_n,\g_n)$. 
We detail these results in Sections~\ref{subsubsec-convergence-compact} and~\ref{sec-compact-balanced-tv-range}.
We give before a sufficient condition of convergence when $\aprox{\phi^*}$ is contractive.
It holds for $\KL$ and some Power entropies (see Proposition~\ref{prop-aprox-power-ent}).

\begin{proposition}\label{prop-contractive-sink}
If $\aprox{\phi^*}$ is a contraction on compact sets w.r.t.~$\norm{\cdot}_\infty$ and if $\C$ is $\gamma$-Lipschitz, then the Sinkhorn algorithm converges linearly towards a unique fixed point w.r.t $\norm{\cdot}_\infty$.
\end{proposition}

\begin{proof}
The Softmin is non-expansive (Lemma~\ref{lem-smin-lipschitz-func}) and Lemma~\ref{lem-smin-cost-regular} gives the continuity of $\Ss_\al(\f)$ and $\Ss_\be(g)$, which are bounded on compact sets. Thus composing with $\aprox{\phi^*}$ gives a contractive mapping with respect to $\norm{.}_\infty$.
\end{proof}

\begin{remark}\label{rem-cv-balanced-sink}
	A similar contraction theorem holds for \emph{balanced} OT.
	The Birkhoff-Hopf theorem from non-linear Perron-Frobenius theory~\cite{lemmens2012nonlinear} states that $(\Ss_\al,\Ss_\be)$ are contractive w.r.t. the Hilbert pseudo-norm.
\end{remark}

\subsubsection{Lemmas on compactness of potentials}
\label{subsubsec-lemma-factor-compact}

This section reduces the proof of compactness in two parts thanks to the structure of $\Ff(\f,\g)$.
Lemma~\ref{lem-uniq-tensor-sum} below states $(\f,\g)$ are optimal up to translations $(\f+\lambda,\g-\lambda)$ for $\lambda\in\R$.
This invariance allows to assume $\f(x_0)=0$ for some $x_0\in\Xx$, and to build a compact set in Lemma~\ref{lem-compact-anchor}.
It then remains to prove that admissible translations $\lambda$ lie in a compact set, which is treated in Sections~\ref{subsubsec-convergence-compact} and~\ref{sec-compact-balanced-tv-range}.

\begin{lemma}[Uniqueness of the optimal dual pair]\label{lem-uniq-tensor-sum}
For any $(\al,\be)\in\Mmpp(\Xx)$, there is uniqueness of optimal potentials $(\f,\g)$ for the dual program~\eqref{eq-dual-unb} in the following sense: if there are two optimal solutions $(\f_1,\g_1)$ and $(\f_2,\g_2)$ then $\f_1\oplus\g_1 = \f_2\oplus\g_2$, $\alpha \otimes \beta$-a.e.. Thus, given optimal potentials $\f\oplus\g$, all other optimal ones can only be $\alpha \otimes \beta$-a.e. of the form $(\f + \lambda, \g-\lambda)$ for some $\lambda\in\R$.
\end{lemma}
\begin{proof}
	Write $(\f_1,\g_1)$ and $(\f_2,\g_2)$ two optimal pairs for~\eqref{eq-dual-unb}, and define $\f_t=t\f_1 + (1-t)\f_2$ and $\g_t=t\g_1 + (1-t)\g_2$ with $t\in[0,1]$. Write
	\begin{align*}
		a_1 &= \dotp{\al}{-\phi^*(-\f_t)} + \dotp{\be}{-\phi^*(-\g_t)},\\
		a_2 &= \dotp{\al}{-t\phi^*(-\f_1) - (1-t)\phi^*(-\f_2)} + \dotp{\be}{-t\phi^*(-\g_1) - (1-t)\phi^*(-\g_2)},\\
		b_1 &= -\epsilon\dotp{\al\otimes\be}{e^{(\f_t\oplus\g_t-\C) / \epsilon} - 1},\\
		b_2 &= -\epsilon\dotp{\al\otimes\be}{t e^{(\f_1\oplus\g_1-\C) / \epsilon} + (1-t)e^{(\f_2\oplus\g_2-\C) / \epsilon} - 1}.
	\end{align*}
	By optimality and convexity of the problem one has $a_1 + b_1 = a_2 + b_2$ as well as $a_1\geq a_2$ and $b_1\geq b_2$, thus necessarily $a_1=a_2$ and $b_1=b_2$. In particular the equality $b_1=b_2$ is an integral against a positive measure whose integrand verifies pointwise $e^{(\f_t(x)\oplus\g_t(y)-\C) / \epsilon}\leq te^{(\f_1(x)\oplus\g_1(y)-\C) / \epsilon} + (1-t)e^{(\f_2(x)\oplus\g_2(y)-\C) / \epsilon}$, thus the inequaliy becomes a pointwise equality holding $\al\otimes\be$-a.e. Eventually, the strict convexity of the exponential yields $\al\otimes\be$-a.e. that $t (f_1(x)+g_1(y)) + (1-t)(f_2(x) + g_2(y)) = f_1(x)+g_1(y) = f_2(x) + g_2(y)$.
\end{proof}

We warn that not all $\lambda$ yield an optimal pair.
There exists a unique $\lambda$ for strictly convex $\phi^*$, while any $\lambda\in\R$ is optimal for balanced OT.

\begin{lemma}[Compact Anchoring of $(\f,\g)$]\label{lem-compact-anchor}
	Assume $\C$ is $\gamma$-Lipschitz, and define 
	$\Pp_{x_o} \eqdef \{ (\f,\g)\in\Aa\Ss_\be(\Cc(\Xx))\times\Aa\Ss_\al(\Cc(\Xx)),\, \f(x_0) = 0,\,\exists M\in\R,\, \norm{\f\oplus\g}_\infty \leq M \}$ 
	for some $x_0\in\Xx$.
	Then one can restrict the dual~\eqref{eq-dual-unb} as a supremum over $\Pp_{x_o}+\R \eqdef\{(\f + \lambda,\g-\lambda),\, (\f,\g)\in\Pp_{x_0},\, \lambda\in\R\}$.
	Furtermore the set $\Pp_{x_o}$ is relatively compact in $\Cc(\Xx)$.
\end{lemma}
\begin{proof}
	Optimality of $(\f,\g)$ is equivalent to have $(\f,\g)=(\Aa\Ss_\be(\g), \Aa\Ss_\al(\f))$ (Proposition~\ref{prop-cns-optimality}), hence the restriction to $\Aa\Ss_\be(\Cc(\Xx))\times\Aa\Ss_\al(\Cc(\Xx))$.
	Such potentials are $\gamma$-Lipschitz (Lemma~\ref{lem-smin-cost-regular}).
	
	We show that in $\Aa\Ss_\be(\Cc(\Xx))\times\Aa\Ss_\al(\Cc(\Xx))$, there exists $\tilde{M}$ s.t. $\norm{\f\oplus\g}_\infty \leq \tilde{M}$.
	Consider a sequence $(\f_n,\g_n)_n$ such that $\norm{\f_n\oplus\g_n}_\infty\rightarrow+\infty$. 
	We have $(\f_n,\g_n)\in\Cc(\Xx)$ with $\Xx$ compact, thus $\exists(x_n,y_n)\in\Xx^2,\,\norm{\f_n\oplus\g_n}_\infty=(\f_n\oplus\g_n)(x_n,y_n)$. 
	We have $(\f_n\oplus\g_n)(x_n,y_n)\rightarrow+\infty$. 
	Since $(\f_n,\g_n)$ are $\gamma$-Lipschitz, $\forall(x,y)\in\Xx^2,\, |(\f_n\oplus\g_n)(x_n,y_n) - (\f_n\oplus\g_n)(x,y)|\leq 2\gamma \text{diam}(\Xx)$, thus $\norm{\f_n\oplus\g_n}_\infty - 2\gamma \text{diam}(\Xx) \leq (\f_n\oplus\g_n)(x,y)$. 
	When $(x,y)\in\spt(\al)\times\spt(\be))$, $(\f_n\oplus\g_n)(x,y)\rightarrow+\infty$ and $\Ff(\f_n,\g_n)\rightarrow-\infty$.
	Similarly, if $(\f_n\oplus\g_n)(x_n,y_n)\rightarrow-\infty$ then $\Ff(\f_n,\g_n)\rightarrow-\infty$. 
	Hence we have $\norm{\f\oplus\g}_\infty \leq \tilde{M}$.
	
	Assume $\norm{\f\oplus\g}_\infty \leq \tilde{M}$.
	Potentials $(\f+\lambda,\g-\lambda)$ with $\lambda\in\R$ have the same bound. 
	Thus w.l.o.g. $\f(x_0)=0$ for some $x_0\in\Xx$. 
	Since $\f$ is $\gamma$-Lipschitz, we have $\norm{\f}_\infty \leq\gamma \text{diam}(\Xx)$ because $\f(x_0)=0$.
	Thus $\norm{\g}_\infty \leq \norm{\f} + \norm{\f\oplus\g}_\infty \leq \gamma \text{diam}(\Xx) + \tilde{M} = M.$
	Thus potentials in $E$ satisfy all properties.
	
	In $\Pp_{x_0}$ potentials are uniformly equicontinuous because $\norm{\f}_\infty,\norm{\g}\leq M$.
	Ascoli-Arzela theorem holds, and $\Pp_{x_0}$ is relatively compact in $\Cc(\Xx)$.
\end{proof}

\subsubsection{Compactness for strictly convex entropies}
\label{subsubsec-convergence-compact}

We prove compactness of potentials under two fairly general assumptions.

\begin{assumption}\label{as:1}
  The function $\phi^*$ is strictly convex.
\end{assumption}
\begin{assumption}\label{as:2}
  There exists a sequence $(x_n)_n \subset \text{\upshape dom}(\phi^*)$ and $s_n\in\partial\phi^*(x_n)$ such that $s_n$ converges either to zero or $+\infty$. 
\end{assumption}
The cases of $\KL$ and Power entropies mentioned Section~\ref{sec-exmp-f-div} satisfy those assumptions.
%
This setting ensures existence and uniqueness of optimal $(\f,\g)$, the latter being key for weak* differentiability of $\OTb$ and $\Sb$.
%

\begin{lemma}[Restriction of the dual $\OTb$ problem to a compact set]\label{lem-compact-dual}
Let $\C$ be a $\gamma$-lipschitz cost function. Under Assumption~\ref{as:2}, the dual problem~\eqref{eq-dual-unb} can be restricted to a supremum over the compact set $\overline{\Pp_{x_0} + I}$ where
\begin{align*}
  \Pp_{x_0} + I = \{ (\f+\lambda,\g-\lambda),\, (\f,\g)\in\Pp_{x_0}, \lambda\in I \}
\end{align*}
with $I$ being a compact set. Furthermore, the compact interval $I$ only depends on $(m(\al),m(\be))$ in a neighborhood of $(\al,\be)$ and this dependency is continuous.
\end{lemma}

\begin{proof}
Lemma~\ref{lem-compact-anchor} applies, thus we consider potentials $(\f+\lambda,\g-\lambda)$ with $(\f,\g)\in\Pp_{x_0}$ (relatively compact) and $\lambda\in\R$.
It remains to prove that the dual program~\eqref{eq-dual-unb} is coercive w.r.t. $\lambda$.

Since $\phi^*$ is convex one has for any $q$ such that $-q\in\text{dom}(\phi^*)$ and $s\in\partial\phi^*(-q)$
\begin{align*}
  \dotp{\al}{-\phi^*(-\f-\lambda)} &\leq \dotp{\al}{-\phi^*(-q) +s(\f + \lambda-q)} \\
  &\leq m(\al) \big(  -\phi^*(-q) +s(\norm{\f}_\infty + \lambda-q) \big).
\end{align*}
From this and the similar inequality for $\be$, we deduce for any $(-q,-\tilde{q})\in\text{dom}(\phi^*)$ and $(s,\tilde{s})\in\partial\phi^*(-q)\times\partial\phi^*(-\tilde{q})$ that
\begin{align*}
\Ff(\f+\lambda,\g-\lambda)=\dotp{\al}{-\phi^*(-\f - \lambda)} &+ \dotp{\be}{-\phi^*(-\g + \lambda)}\leq  R(s, \tilde{s})\lambda + K,
\end{align*}
where $R(s, \tilde{s})\eqdef m(\al)s - m(\be)\tilde{s}$ and
\begin{align*}
	K \eqdef m(\al)\big(-\phi^*(-q) + s(\norm{\f}_\infty -q) \big)+ m(\be)\big(-\phi^*(-\tilde{q}) + \tilde{s}(\norm{\g}_\infty -\tilde{q}) \big).
\end{align*}
To be coercive in $\lambda$, we need to find points $(q_1,\tilde{q}_1)\in\text{dom}(\phi^*)^2$, $(s_1, \tilde{s}_1)\in\partial\phi^*(q_1)\times\partial\phi^*(\tilde{q}_1)$ and $(q_2,\tilde{q}_2)\in\text{dom}(\phi^*)^2$, $(s_2, \tilde{s}_2)\in\partial\phi^*(q_2)\times\partial\phi^*(\tilde{q}_2)$ such that $K<+\infty$ (it holds on $\text{dom}(\phi^*)$), such that $R(s_1,\tilde{s}_1) > 0$ and $R(s_2,\tilde{s}_2) < 0$. 
Assumption~\ref{as:2} proves it. 
We have $\partial\phi^*\subset\R_+$. 
If  $\exists s_n\in\partial\phi^*(x_n)$ with $s_n\rightarrow 0$, take $q_1\in\text{dom}(\phi^*)$ and $\tilde{q}_1=x_n$.
There exists $n_0$, such that $s_n$ is small enough for $n\geq n_0$ and $R(s_1, s_n)>0$. 
Similarly we find some $R(s_2, \tilde{s}_2)<0$. 
The same approach holds for $\partial\phi^*(x_n)\rightarrow +\infty$. 
Thus $\Ff(\f+\lambda,\g-\lambda)\rightarrow-\infty$ when $\lambda\rightarrow\pm\infty$ by taking either $R<0$ or $R>0$.

Note that $(R,K)$ depends continuously on $(m(\al),m(\be))$. Thus, on a neighbourhood of $(\al,\be)$, one still has $R(s_1, \tilde{s}_1) > 0$, $R(s_2, \tilde{s}_2) < 0$ and $|K|<\infty$.

Coercivity holds and $\lambda$ is in a compact interval $I$ that is constant in a neighborhood of $(\al,\be)$. 
Thus the optimal potentials can be taken in a set $\Pp_{x_0}+I$. 
The potentials inside this set remain equicontinuous and uniformly bounded. 
The Ascoli-Arzelà theorem applies and $\Pp_{x_0}+I$ is relatively compact in $\Cc(\Xx)$.
\end{proof}

Lemma~\ref{lem-compact-dual} and Assumptions~\ref{as:1} ensure existence and uniqueness of optimal $(\f,\g)$.
We now prove $(\f,\g)$ depend continuously in $(\al,\be)$, which is key for the weak* regularity of $\OTb$ studied in Section~\ref{sec-ot-prop}. 

\begin{proposition}[The dual potentials vary continuously with the input measures]
\label{prop-uniform-conv}
Let $\C$ be a $\gamma$-lipschitz cost function.
Let $\al_n \rightharpoonup \al$ and
$\be_n \rightharpoonup \be$ be weakly converging
sequences of measures in $\Mmpp(\Xx)$.
Write $(\f_n,\g_n)$ the (unique) sequence of optimal
potentials for $\OTb(\al_n,\be_n)$.

Under Assumptions~\ref{as:1} and \ref{as:2}, $\f_n$ and $\g_n$ converge uniformly towards the
unique pair of optimal potentials $(\f,\g)$ for $\OTb(\al,\be)$:
\begin{align*}
\big( \al_n \rightharpoonup \al, \,
\be_n \rightharpoonup \be
\big)\Longrightarrow
\big( \f_n \xrightarrow{\|\cdot\|_\infty} \f, \,
\g_n \xrightarrow{\|\cdot\|_\infty} \g
\big).
\end{align*}
\end{proposition}

\begin{proof}
  Thanks to Theorem~\ref{thm-cv-sink-compact}, for all $(\al_n,\be_n)$, $\exists !(\f_n,\g_n)$ optimal in $\OTb(\al_n,\be_n)$. 
  Applying Lemma~\ref{lem-compact-dual}, we assume $(\f_n,\g_n)\in\Pp_{x_0}+I_n$.
  The interval $I_n$ depends continuously in $(m(\al_n),m(\be_n))$ in a neighborhood of $(\al_n,\be_n)$. 
  Since $(\al_n,\be_n)\rightharpoonup(\al,\be)$ with $m(\al),m(\be)>0$,
  there exists $(-q_i, -\tilde{q}_i)\in\text{dom}(\phi^*)^2$, $(s_i,\tilde{s}_i)\in\partial\phi^*(-q_i)\times\partial\phi^*(-\tilde{q}_i)$, $\eta >0$ and $n_0$ such that for all $n \geq n_0$
\begin{gather*}
  0 < m(\al) -\eta < m(\al_n) < m(\al) +\eta \; \text{and} \; 0 < m(\be) -\eta < m(\be_n) < m(\be) +\eta\\
  \Rightarrow R_+ \eqdef (m(\al)-\eta)s_1 - (m(\be)+\eta)\tilde{s}_1 < R(s_1,\tilde{s}_1)<0,\\
   \Rightarrow R_- \eqdef (m(\al)+\eta)s_2 - (m(\be)-\eta)\tilde{s}_2 > R(s_1,\tilde{s}_1)>0,
\end{gather*}
where $R$ is defined in the proof of Lemma~\ref{lem-compact-dual}. 
Again, Assumption~\ref{as:2} guarantees that we find points in $\text{dom}(\phi^*)$, independent of $n$, such that $\forall n \geq n_0$, $R_+ <0$ and $R_- >0$. 
We then build a compact subset $\Pp_{x_0}+I$ with $I$ compact and independent of $n$ such that for all $n$, $(\f_n,\g_n)\in\Pp_{x_0}+I$.
Ascoli-Arzelà theorem holds.
   One  extracts a subsequence $(\f_{n_k},\g_{n_k})\rightarrow(\f,\g)$, where $(\f,\g)$ are unique optimal in $\OTb(\al,\be)$. 
   All subsequences converge to the same limit due to uniqueness of optimal potentials. 
   Hence we get $(\f_n,\g_n)\rightarrow(\f,\g)$.
\end{proof}

\subsubsection{Compactness for Balanced, Total Variation and Range entropies}
\label{sec-compact-balanced-tv-range}

The settings of balanced, TV and Range OT do not satisfy Assumption~\ref{as:1}. 
We prove compactness below with Lemmas~(\ref{lem-compact-balanced},\ref{lem-compact-tv},\ref{lem-compact-range}) dedicated to each setting.
They guarantee Theorem~\ref{thm-cv-sink-compact} holds.


\begin{lemma}[Balanced OT]\label{lem-compact-balanced}
In the setting of balanced OT where $\phi^*(x)=x$, the dual program can be restricted to the compact set $\Pp_{x_0}$.
\end{lemma}
\begin{proof}
Lemma~\ref{lem-compact-anchor} applies, thus we consider potentials $(\f+\lambda,\g-\lambda)$ with $(\f,\g)\in\Pp_{x_0}$ and $\lambda\in\R$.
In this setting $\Ff(\f+\lambda, \g-\lambda)=\Ff(\f,\g)$. This invariance allows us to quotient $\Cc(\Xx)$ w.r.t. such translations.
Thus w.l.o.g. we can assume $(\f,\g)\in\Pp_{x_0}$, which is compact according to Lemma~\ref{lem-compact-anchor}.
\end{proof}


\begin{lemma}[Total Variation]\label{lem-compact-tv}
In the setting of total variation OT where $\phi^*(x)=\max(-\rho, x)$ with $\text{dom}(\phi^*) = (-\infty, \rho]$ and $\rho >0$, the dual program can be restricted to a set of functions which is compact.
\end{lemma}
\begin{proof}
%
Sinkhorn iterates are equicontinuous (Lemma~\ref{lem-smin-cost-regular}).
When $\D_\phi=\rho\TV$, $\aprox{\phi^*}$ imposes $\norm{\f}_\infty,\norm{\g}_\infty\leq\rho$, for any $(\f,\g)\in\Aa\Ss_\be(\Cc(\Xx))\times\Aa\Ss_\al(\Cc(\Xx))$.
Thus iterates are uniformly equicontinuous, Ascoli-Arzelà theorem holds, hence the compactness.
\end{proof}

Finally, we prove the compactness in the limit setting of the range divergence.

\begin{lemma}[Range divergence]\label{lem-compact-range}
In the setting of range OT where $\phi^*(x) = \max(a x, b x)$ with $0\leq a \leq 1 \leq b$ and for any $(\al,\be)\in\Mmpp(\Xx)$ such that $ E = [a m(\al), b m(\al)]\cap[a m(\be), b m(\be)] \neq\emptyset$, the dual program can be restricted to a compact set of dual potentials.
\end{lemma}
\begin{proof}
Lemma~\ref{lem-compact-anchor} applies, thus we consider potentials $(\f+\lambda,\g-\lambda)$ with $(\f,\g)\in\Pp_{x_0}$ (relatively compact) and $\lambda\in\R$.
It remains to prove that $\Ff_\lambda\lambda\mapsto\Ff(\f+\lambda,\g-\lambda)$ is coercive.
Lemma~\ref{lem-compact-anchor} gives $\norm{\f}_\infty \leq M$ and $\norm{\g}_\infty \leq M$, thus we have
\begin{align*}
  \exists\lambda_0,\, \forall\lambda\geq\lambda_0,\, \dotp{\al}{-\phi^*(-\f - \lambda)} &+ \dotp{\be}{-\phi^*(-\g + \lambda)} = \kappa + \lambda(a m(\al) - b m(\be) ) , \\
  \exists\lambda_1,\, \forall\lambda\leq\lambda_1,\, \dotp{\al}{-\phi^*(-\f - \lambda)} &+ \dotp{\be}{-\phi^*(-\g + \lambda)} = \kappa + \lambda(b m(\al) - a m(\be) ) .
\end{align*}
The terms independent of $\lambda$ are considered as constants, denoted by $\kappa$ that changes from line to line.

Because $E\neq\emptyset$, we have $\lambda(a m(\al) - b m(\be) )\leq 0$ and $\lambda(b m(\al) - a m(\be) )\leq 0$ when $\lambda\geq\lambda_0$ or $\lambda\leq\lambda_1$.
There are two cases.
Assume first $E$ is not a singleton.
Then both slopes are negative and the $\Ff(\f+\lambda,\g-\lambda)\rightarrow-\infty$ when $\lambda\rightarrow\pm\infty$.
It yields a compact set of potentials for the same reasons as Lemma~\ref{lem-compact-dual}.

Assume now $E$ is a singleton. 
One slope is then zero, while the other is negative, e.g. $a m(\al) = b m(\be)$. 
Then $\Ff_\lambda$ is not coercive and attains a plateau when $\lambda\rightarrow +\infty$. 
However, by concavity of $\Ff$, any $(\f+\lambda,\g-\lambda)$ attaining this plateau is optimal.
Here any potential such that $\f +\lambda>0$ and $\g-\lambda <0$ reaches the optimal plateau.
Since we have $\norm{\f}_\infty, \norm{\g}_\infty \leq M$, then $\lambda\in I=[-M, M]$ is enough to have such optimal functions in the compact set $\Pp_{x_0} + I$. 
It allows to restrict the dual program on $\Pp_{x_0} + I$.
The same holds if $b m(\al) = a m(\be)$.
\end{proof}

\begin{remark}\label{rem-asym-entropy}
	All the above proofs of compactness could be extended to \emph{asymmetric} penalties $\D_{\phi_1}(\pi_1|\al)$ and $\D_{\phi_2}(\pi_2|\be)$. 
	Theorem~\ref{thm-cv-sink-compact} would hold in such setting.
\end{remark}

We end with an example where Assumption~\ref{as:1} is not satisfied using the Range divergence. 
Uniqueness of $(\f,\g)$ no longer holds, and $\Ff$ can be constant on some domain when $R(a,b)=m(\al)b - m(\be)a=0$.
The set of optimizers may even be unbounded, which is why we consider the Range divergence as a limit setting of this theory.
Note that Theorem~\ref{thm-cv-sink-compact} holds even in this setting, i.e. Sinkhorn iterates converge to finite $(\f,\g)\in\Cc(\Xx)^2$.

\begin{example}\label{exmp-range2}
Consider $\D_\phi=RG_{[a,b]}$, $\epsilon=1$, $\al=a \delta_x$ and $\be=b \delta_y$ where $(a,b)$ are the parameters of the Range divergence. 
Assume that $\C=\C(x,y)\in[-\log b, -\log a]$. 
We have $\Ss_\al(\f)=\C - \f -\log a$ and $\Ss_\be(\g) = \C - \g-\log b$.
Taking $(\f_0,\g_0)=(0,0)$, the assumption on $\C$ gives after using $\aprox{\phi^*}$ that $(\f_1,\g_1)=(0,0)$, thus $(\f_0,\g_0)$ are optimal. 
If we consider  $(\f_0+\lambda,\g_0-\lambda)$ we have for $\lambda\in\R_-$, $\dotp{\al}{\phi^*(\lambda)} = a(b\lambda)$ and $\dotp{\be}{\phi^*(-\lambda)} = b(-a\lambda)$.
Thus $\Ff(\lambda,-\lambda)=-\epsilon\dotp{\al\otimes\be}{e^{-\C / \epsilon} - 1}$ is constant and optimal $\forall\lambda\leq 0$. 
The set of optimal potentials contains all pairs $(-\lambda,\lambda)$ and is unbounded.
\end{example}

\section{Properties of entropic unbalanced optimal transport functionals}
\label{sec-ot-prop}

We focus here on topological properties of $\OTb$ and functionals derived from it.
Our main result are Theorem~\ref{thm-sink-unb} and~\ref{thm-sink-weak-cv} stating that $\Sb$ is convex, positive, definite, and metrizes the weak* topology.
It means that $\Sb$ satisfies more metric properties than $\OTb$.

%
%

\subsection{Weak* regularity of unbalanced OT}
\label{sec-weak-regularity-ot}
We detail the regularity of $\OTb$.
The Sinkhorn divergence $\Sb$ inherits those properties.

\begin{theorem}[Convexity and continuity of $\OTb$]
	\label{thm-continuity-unb}
For any entropy $\phi$, $\OTb$ is convex on $\Mmp(\Xx)$ in $\al$ and $\be$ but not jointly convex.
Assume $\phi$ is continuous and satisfies Asssumptions~\ref{as:1} and~\ref{as:2}.
If $(\al_n,\be_n)\rightharpoonup(\al,\be)\in\Mmpp(\Xx)^2$, then $\OTb(\al_n,\be_n)\rightarrow\OTb(\al,\be)$.
\end{theorem}
\begin{proof}
The functional $\OTb$ is a supremum of functions which are linear in $\al$ and linear in $\be$, but not jointly convex in $(\al,\be)$, hence the convexity result.
Concerning continuity, Theorem~\ref{thm-sink-weak-cv} and Proposition~\ref{prop-uniform-conv} hold.
There exists $(\f_n,\g_n)_n$ and $(\f,\g)$ such that $\OTb(\al_n,\be_n)=\Ff(\f_n,\g_n)$, $\OTb(\al,\be)=\Ff(\f,\g)$ and $(\f_n,\g_n)\rightarrow(\f,\g)$.
Because $\phi$ is continuous, so is $\Ff$ on $\Cc(\Xx)^2$.
Thus $\Ff(\f_n,\g_n)\rightarrow\Ff(\f,\g)$, hence the continuity of $\OTb$.
\end{proof}

\begin{remark}\label{rem-continuity-not-strictly}
	A similar result holds for balanced, TV or Range but requires dedicated proofs detailed in Appendix~\ref{appendix-proofs}.
	In the Range setting $\OTb$ would only be continuous on its domain.
	For instance, with $\D_\phi=\RG_{[1,2]}$, we have $\OTb((1-\epsilon)\al, (2+\epsilon)\al)=+\infty$ for any $\al$ and $\epsilon>0$, even though $\OTb(\al,\al)<+\infty$.
\end{remark}

We focus on the differentiability of $\OTb$. We start with subdifferentials defined for any setting, then study differentiability under additional assumptions.

\begin{definition}[Subdifferential on a space of measures]\label{def-subdif}
Let $\Ff$ be any functional defined on $\Mmp(\Xx)$. The subdifferential of $\Ff$ at $\al\in\Mmp(\Xx)$ is defined as
\begin{align*}
  \partial\Ff(\al) \eqdef \{p\in\Cc(\Xx),\,\, \forall\be\in\Mmp(\Xx),\, \Ff(\be)\geq \Ff(\al) + \dotp{\be - \al}{p} \}
\end{align*}
If $\partial\Ff(\al) \neq\emptyset$, we say that $\Ff$ is subdifferentiable at $\al$.
\end{definition}

\begin{proposition}[Subdifferential of $\OTb$]\label{prop-subdif}
Let us assume that Assumption~\ref{as:2} holds or consider the case of balanced, TV and Range unbalanced optimal transport. 
For any $(\al,\be)\in\Mmpp(\Xx)$ such that $\OTb(\al,\be) < \infty$, note $(\f,\g)$ optimal potentials. Then subdifferentials are nonempty, and
\begin{align*}
        - \phi^*(-\f) -\epsilon\dotp{\be}{ \tefgc} + \epsilon m(\be) \in\partial_1\OTb(\al,\be), \\
        - \phi^*(-\g) -\epsilon\dotp{\al}{ \tefgc} + \epsilon m(\al) \in\partial_2\OTb(\al,\be).
\end{align*}
\end{proposition}
\begin{proof}
The proof is similar for both coordinates: let us show it for the first one. Take $(\bar{\al},\be)$, and compare $\OTb(\bar{\al},\be)$ with $\OTb(\al,\be)$. The pair $(\f,\g)$ is suboptimal in $\OTb(\bar{\al},\be)$, thus
\begin{align*}
  \OTb(\bar{\al},\be) &\geq - \dotp{\bar{\al}}{\phi^*(-\f)} - \dotp{\be}{\phi^*(-\g)} - \epsilon \dotp{\bar{\al}\otimes\be}{\tefgc - 1}\\
                      &\geq \dotp{ \bar{\al} }{-\phi^*(-\f)-\epsilon\dotp{\be}{ \tefgc -1}} - \dotp{\be}{\phi^*(-\g)}\\
                      &\geq \OTb(\al,\be) + \dotp{\bar{\al} - \al}{-\phi^*(-\f)-\epsilon\dotp{\be}{ \tefgc - 1}}.                      
\end{align*}
Since $(\f,\g)$ is optimal in $\OTb(\al,\be)$, we get that $- \phi^*(-\f) -\epsilon\dotp{\be}{ \tefgc } + \epsilon m(\be)\in\partial_1\OTb(\al,\be)$. The similar property holds for $\partial_2\OTb(\al,\be)$.
\end{proof}

We now consider Assumptions~(\ref{as:1},\ref{as:2}) hold.
We prove stronger differentiability properties of $\OTb$ in this setting.
First we define it for functionals defined on $\Mm(\Xx)$.

\begin{definition}[Differentiability in $\Mmp(\Xx)$] \label{def-diff-meas}
Let $\Ff$ be any functional defined on $\Mmp(\Xx)$. We say that it is differentiable in the sense of measures if for any $\al\in\Mmp(\Xx)$, there exists a function $\nabla \Ff(\al) \in \Cc(\Xx)$ such that for any $t$ in a neighborhood of $0$ and for any $\de\al\in\Mm(\Xx)$ with $\al + t\de\al \in \Mmp(\Xx)$,
\begin{align*}
  \Ff(\al+t\de\al) = \Ff(\al) + t \dotp{\de\al}{\nabla\Ff(\al)} + o(t).
\end{align*}
If such property holds, we call $\nabla\Ff(\al)$ the gradient of $\Ff$ at $\al$.
\end{definition}

We now present our main theorem on the regularity of $\OTb$. 
Note that it does not hold for balanced OT. 
This case requires a separate proof, detailed in~\cite{feydy2018interpolating}.

\begin{theorem}
	\label{thm-diff-unb}
	Let $\C$ be a $\gamma$-lipschitz cost function.
	Under Assumptions~\ref{as:1} and \ref{as:2}, $\OTb$ is differentiable on
	$\Mmpp(\Xx)^2$ in the sense of Definition~\ref{def-diff-meas}.
	For any $(\al,\be)$, write $(\f,\g)$ the unique potentials verifying $(\f,\g)=(\Aa\Ss_\be(\g), \Aa\Ss_\al(\f))$ everywhere on $\Xx$ (see Remark~\ref{rem-extrapolate-pot}). Then the gradients read
    \begin{align*}
       \nabla_\al \OTb  &= - \phi^*(-\f) -\epsilon\dotp{\be}{ \tefgc -1} \\
       \nabla_\be \OTb  &= - \phi^*(-\g) -\epsilon\dotp{\al}{ \tefgc -1}.
    \end{align*}
    Furthermore, if $\phi^*$ is differentiable, one can simplify formulas using $\nabla\phi^*(-\f)=\dotp{\be}{ \tefgc }$ and $\nabla\phi^*(-\g)=\dotp{\al}{ \tefgc }$.
\end{theorem}

\begin{proof}
The proof is deferred in Appendix~\ref{appendix-proofs}. It is a generalization of~\cite{santambrogio2015optimal} and~\cite{feydy2018interpolating}.
\end{proof}

The last point of Theorem~\ref{thm-diff-unb} is important from a computational perspective. 
Computing $\dotp{\be}{e^{(\f\oplus\g-\C)/ \epsilon}}$ takes $O(N^2)$ time, while $\dotp{\al}{\nabla\phi^*(-\f)}$ takes $O(N)$ time because $\nabla\phi^*$ is applied pointwise.


We give as a corollary the formulas in the popular case $\D_\phi=\rho\KL$. 

\begin{corollary}[Gradient of $\OTb$ for $\rho\KL$]
\label{cor-diff-KL}
When $\D_\phi=\rho\KL$, $\OTb$ is differentiable in the sense of Theorem~\ref{thm-diff-unb}. For any measures $(\al,\be)$ whose (existing and unique) potentials are noted $(\f,\g)$:
\begin{gather}
\begin{aligned}
\nabla_\al \OTb(\al,\be) = (\rho+\epsilon m(\be)) - (\rho+\epsilon)\exp(-\f/\rho),\\
\nabla_\be \OTb(\al,\be) = (\rho+\epsilon m(\al)) - (\rho+\epsilon)\exp(-\g/\rho).
\end{aligned}
\end{gather}
\end{corollary}
\begin{proof}
Theorem~\ref{thm-diff-unb} holds in this setting. 
It then suffices to compute formulas with $\phi^*(x)=\rho(e^{x / \rho} - 1)$.
\end{proof}

\subsection{Sinkhorn divergence, Sinkhorn entropy and Hausdorff divergence}
\label{subsec-sink-functionals}
We present in this section functionals derived from $\OTb$.
Most importantly, we generalize the balanced Sinkhorn divergence~\cite{ramdas2017wasserstein,genevay2018learning,feydy2018interpolating} to the unbalanced setting.
%

\begin{definition}\label{def-sink-div-unb}
The \emph{Unbalanced Sinkhorn divergence} is defined as
\begin{align}
  \Sb(\al,\be) \eqdef \OTb(\al,\be) - \tfrac{1}{2} \OTb(\al,\al) - \tfrac{1}{2} \OTb(\be,\be) +\tfrac{\epsilon}{2} \big( m(\al) - m(\be) \big)^2.
\end{align}
The \emph{Unbalanced Sinkhorn Entropy} is defined as
\begin{gather}
\begin{aligned}
\Fb(\al) \eqdef &-\sup_{\f\in\Cc(\Xx)} \dotp{\al}{-\phi^*(-\f)} - \tfrac{\epsilon}{2}\dotp{\al\otimes\al}{e^{\frac{\f\oplus\f-\C}{\epsilon}}}
= &- \tfrac{1}{2}\OTb(\al,\al) + \tfrac{\epsilon}{2}m(\al)^2,
\end{aligned}
\end{gather}
where the last relation holds thanks to Proposition~\ref{lem-sym-pot}.
Under Assumptions~(\ref{as:1},\ref{as:2}) $\OTb$ is differentiable, and the symmetric Bregman divergence associated to $\Fb$ is well-defined. 
We call it the \emph{Hausdorff divergence}. It reads
\begin{align}
  \Hb(\al,\be) \eqdef \dotp{\al - \be}{\nabla\Fb(\al) - \nabla\Fb(\be) }. 
\end{align}
From now on, we write $(\f_{\al\be},\g_{\al\be}, \f_\al, \g_\be)$ optimal potentials such that $\OTb(\al,\be)=\Ff(\f_{\al\be},\g_{\al\be})$, $\OTb(\al,\al)=\Ff(\f_\al,\f_\al)$ and $\OTb(\be,\be)=\Ff(\g_\be,\g_\be)$.
\end{definition}

Those divergences can be explicited as functions of $(\f_{\al\be},\g_{\al\be}, \f_\al, \g_\be)$, allowing a simple numerical computation (see Section~\ref{sec-implementation}). 
We provide formulas for $\D_\phi = \rho\KL$. 

\begin{proposition}[Dual formulas for the Sinkhorn costs]\label{prop-funct-kl}
Assuming the cost $\C$ to be symmetric and $\gamma$-Lipschitz. For $\D_\phi = \rho\KL$  and $(\al,\be)\in\Mmpp(\Xx)$ one has 
\begin{align}
  \OTb(\al,\be) &= \dotp{\al}{  \rho - (\rho + \tfrac{\epsilon}{2}) e^{-\tfrac{\f_{\al\be}}{\rho}}  } 
  + \dotp{\be}{  \rho - (\rho + \tfrac{\epsilon}{2}) e^{-\tfrac{\g_{\al\be}}{\rho}}  } + \epsilon m(\al)m(\be), \nonumber\\
  \Sb(\al,\be) &= \dotp{\al}{ - (\rho + \tfrac{\epsilon}{2}) \big( e^{-\tfrac{\f_{\al\be}}{\rho}} - e^{-\tfrac{\f_{\al}}{\rho}}  \big) }
  + \dotp{\be}{ - (\rho + \tfrac{\epsilon}{2}) \big( e^{-\tfrac{\g_{\al\be}}{\rho}} - e^{-\tfrac{\g_{\be}}{\rho}}  \big) }, \nonumber \\
  \Hb(\al,\be) &= \dotp{\al}{ - (\rho + \epsilon) \big( e^{-\tfrac{\f_{\al}}{\rho}} - e^{-\tfrac{\g_{\be}}{\rho}}  \big) }
  + \dotp{\be}{ - (\rho + \epsilon) \big( e^{-\tfrac{\g_{\be}}{\rho}} - e^{-\tfrac{\f_{\al}}{\rho}}  \big) }. \nonumber
\end{align}
\end{proposition}
\begin{proof}
	For $(\al,\be)\in\Mmpp(\Xx)$ Theorem~\ref{thm-cv-sink-compact} applies and yields existence and uniqueness of the potentials. The result is a calculation derived from the dual program~\eqref{eq-dual-unb} and the formulas from Theorem~\ref{thm-diff-unb}, i.e. $e^{-\f_{\al\be} / \rho} = \dotp{\be}{e^{(\g_{\al\be} - \C) / \epsilon}}$ and $e^{-\g_{\al\be} / \rho} = \dotp{\al}{e^{(\f_{\al\be} - \C) / \epsilon}}$.
\end{proof}
\begin{remark}
	Balanced OT satisfies the relation 
	$$\OTb(\al,\be) = \dotp{\al}{\nabla_\al\OTb(\al,\be)} + \dotp{\be}{\nabla_\be\OTb(\al,\be)}.$$
	The previous result shows that this relation does not hold for unbalanced OT.
\end{remark}


\subsection{Properties of the Sinkhorn entropy}
\label{subsec-sink-entropy}

We first focus on the Sinkhorn entropy. 
A key result is Proposition~\ref{thm-Feps_uniqueness} stating that $\Fb$ is convex.
It means that $\al\mapsto\OTb(\al,\al)$ is concave, which contrasts with the convexity of $\al\mapsto\OTb(\al,\be)$.
%
Theorem~\ref{thm-Feps_uniqueness} states the regularity of $\Fb$, but we need to reformulate it with the next result to ease its study.

\begin{proposition}[Change of variables in the symmetric $\OTb$ problem]\label{prop-entropy-reform}
Assuming that $\C$ is symmetric and such that the kernel $k_\epsilon=e^{-\C / \epsilon}$ is positive, one has
\begin{align*}
  \Fb(\al) = \inf_{\mu\in\Mmp(\Xx)}  \dotp{\al}{\phi^*\big(- \epsilon \log\big( \frac{\d\mu}{\d\al}\big)\,\big)} + \tfrac{\epsilon}{2} \Vert \mu \Vert^2_{k_\epsilon}.
\end{align*}
\end{proposition}
\begin{proof}
Similar to~\cite{feydy2018interpolating} for balanced OT, we perform a change of variable $\mu = e^{\f / \epsilon} \al$ to get
\begin{align*}
  \Fb(\al) 
  &= - \sup_{\f\in \Cc(\Xx)} -  \dotp{\al}{\phi^*(-\f)} -\tfrac{\epsilon}{2} \dotp{\al\otimes\al}{e^{\frac{\f\oplus \f - \C }{\epsilon}}}\\
  &= \inf_{\mu\sim\al}  \dotp{\al}{\phi^*\big(- \epsilon \log\big( \frac{\d\mu}{\d\al}\big)\,\big)} + \tfrac{\epsilon}{2} \Vert \mu \Vert^2_{k_\epsilon}\label{eq-constrained-entropy-program}
  = \inf_{\mu\in\Mmp(\Xx)}  \dotp{\al}{\phi^*\big(- \epsilon \log\big( \frac{\d\mu}{\d\al}\big)\,\big)} + \tfrac{\epsilon}{2} \Vert \mu \Vert^2_{k_\epsilon}.
\end{align*}
We relax the constraint $\mu\sim\al$ at the last line. 
The kernel $k_\epsilon$ is positive, so $\mu\ll\al$ can be removed.
Then $\al\ll\mu$ holds since $\lim_{q\rightarrow +\infty} \phi^*(q) = +\infty$. 
Otherwise, there is a $\al$-non-negligible set A s.t. $\frac{\d\mu}{\d\al}(x)=0$ $\alpha$-ae on $A$, so $\log \frac{\d\mu}{\d\al}=-\infty$.
\end{proof}

We can use Proposition~\ref{prop-entropy-reform} to prove the following theorem.

\begin{theorem}[Properties of the symmetric $\OTb$ problem]\label{thm-Feps_uniqueness}
Assume $\C$ is symmetric and $k_\epsilon=e^{-\C / \epsilon}$ is a positive universal kernel.
Then there exists a unique $\mu_\al\in\Mmp(\Xx)$ which attains the infimum, i.e. is such that
\begin{align*}
\Fb(\al)~=~  
\langle\al,\phi^*(-\epsilon\log \tfrac{\d\mu_\al}{\d\al})\rangle
+\tfrac{\epsilon}{2}
\langle \mu_\al, 
k_\epsilon\star \mu_\al \rangle.
\end{align*}
Moreover, $\al \sim \mu_\al$, and $\f = \epsilon\log\tfrac{\d\mu_\al}{\d\al}$ is the optimal dual potential for $\Fb(\al)$. Furthermore, $\Fb$ is strictly convex and is weak* continuous for all settings of Section~\ref{sec-exmp-f-div}. 
This implies that the Hausdorff divergence $\Hb$ is positive definite.
\end{theorem}

\begin{proof}
The proof is a generalization of~\cite{feydy2018interpolating} and is deferred in Appendix~\ref{appendix-proofs}.
\end{proof}

Proposition~\ref{prop-entropy-reform} involves a new variable $\mu_\al=e^{\f_\al / \epsilon}\al$.
The map $\al\mapsto\mu_\al$ appears to be injective, which matters for the definiteness of $\Sb$. 

\begin{lemma}[Injectivity of the symmetric $\OTb$ problem]\label{lem-injective-sym-pot}
	Note $\f_\al$ the optimal potential of $\Fb(\al)$. Assume $\C$ is symmetric.
	Then $\al\mapsto\al e^{\f_\al / \epsilon}$ is injective.
\end{lemma}
\begin{proof}
	Symmetry of $\C$ yields the optimality condition $\f=\Aa\Ss_\al(\f)$ on $\f$.
	Assume $\al e^{\f_\al / \epsilon} = \be e^{\g_\be / \epsilon}$ for some $(\al,\be)$. 
	This equality implies that $\dotp{\al}{e^{(\f_\al-\C(x,.)) / \epsilon}} = \dotp{\be}{e^{(\g_\be-\C(x,.)) / \epsilon}}$. After composing with the log and the aprox, we get	$-\aprox{\phi^*}(-\Ss_\al(\f_\al)) = -\aprox{\phi^*}(-\Ss_\be(\g_\be))$.
	By optimality of $(\f_\al,\g_\be)$ we have $\f_\al = \g_\be$, thus the relation $\al e^{\f_\al / \epsilon} = \be e^{\g_\be / \epsilon}$ implies $\al=\be$.
\end{proof}

We state the link between $\Fb$ and $\OTb$ when the problem is symmetric.
The properties of $\Sb$ rely heavily on it.
As discussed in~\cite{knight2014symmetry,feydy2020thesis} for balanced OT, the potential $\f_\al$ is much faster to compute in the symmetric setting, which matters to compute $\OTb(\al,\al)$.
%

\begin{proposition}
	\label{lem-sym-pot}
If $\C$ is symmetric, Then $\Fb(\al)=-\tfrac{1}{2}\OTb(\al,\al) + \tfrac{\epsilon}{2}m(\al)^2$.
\end{proposition}
\begin{proof}
There exists an optimal potential $\f$ for $\Fb$ (Theorem~\ref{thm-Feps_uniqueness}). 
The pair $(\f,\f)$ is suboptimal for $\OTb(\al,\al)$, thus $\Fb(\al)\leq -\tfrac{1}{2}\OTb(\al,\al) + \tfrac{\epsilon}{2}m(\al)^2$.

We use a primal suboptimality argument to get the other inequality. 
Consider the plan $\pi = e^{(\f\oplus\f-\C) / \epsilon}\al\otimes\al$. 
The marginals read $\pi_1=\pi_2=\dotp{\al}{e^{(\f-\C) / \epsilon}}e^{\f / \epsilon}\al$ since $\C$ is symmetric. 
Thanks to symmetry, the optimality condition on $\Ff(\f,\f)$ reads ($ \frac{\d\pi_i}{\d\al} = \dotp{\al}{e^{(\f-\C) / \epsilon}}e^{\f / \epsilon} \in \partial\phi^*(-\f)$.
It is equivalent to $\phi^*(-\f(x)) + \phi(\frac{\d\pi_i}{\d\al}(x)) = -\f(x)\frac{\d\pi_i}{\d\al}(x)$ for any $x\in\Xx$.
The suboptimality of $\pi$ reads
\begin{align*}
    \OTb(\al,\al) \leq \dotp{\pi}{\C} + \rho\D_\phi(\pi_1|\al) + \rho\D_\phi(\pi_2|\al) + \epsilon\KL(\pi|\al\otimes\al),
\end{align*}
where
   $ \D_\phi(\pi_i|\al) = \dotp{\al}{\phi(\frac{\d\pi_i}{\d\al})}
    = \dotp{\al}{-\f\frac{\d\pi_i}{\d\al} - \phi^*(-\f)},$
and
\begin{align*}
    \epsilon\KL(\pi|\al\otimes\al) &= \epsilon\dotp{\pi}{\log\frac{\d\pi}{\d\al\d\al}} - \epsilon m(\pi) + \epsilon m(\al)^2\\
    &= \dotp{\pi}{\f\oplus\f-\C} - \epsilon\dotp{\al\otimes\al}{e^{\tfrac{\f\oplus\f-\C}{\epsilon}} - 1}\\
    &= \dotp{\pi_1}{\f} + \dotp{\pi_2}{\f} - \dotp{\pi}{\C}- \epsilon\dotp{\al\otimes\al}{e^{\tfrac{\f\oplus\f-\C}{\epsilon}} - 1}\\
    &= \dotp{\al}{\f\frac{\d\pi_1}{\d\al}} + \dotp{\al}{\f\frac{\d\pi_2}{\d\al}} - \dotp{\pi}{\C}- \epsilon\dotp{\al\otimes\al}{e^{\tfrac{\f\oplus\f-\C}{\epsilon}} - 1}.
\end{align*}
Summing everything, we get the inequality $\OTb(\al,\al)\leq -2\Fb(\al) + \epsilon m(\al)^2$, hence the desired equality.
\end{proof}

\subsection{Bounds and  Asymptotics of $\Sb$}
\label{subsec-sink-div}

We present properties of $\Sb$ which are key to prove the main Theorems~\ref{thm-sink-unb} and~\ref{thm-sink-weak-cv}.
We provide two bounds.
The first one involves $\Hb$ and extends~\cite{feydy2018interpolating}.
The second is new and involves a kernel norm, thus hilighting the connection between entropic OT and Reproducing Kernel Hilbert Spaces (RKHS).


\begin{proposition}[The Sinkhorn divergence is bounded from below by a ``soft'' Hausdorff divergence]\label{prop-Seps-ineq-bregman}
Under Assumptions~\ref{as:1} and \ref{as:2}, for any $(\al,\be)\in\Mmpp(\Xx)$, one has
\begin{align*}
  \Sb(\al,\be) \geq \tfrac{1}{2} \Hb(\al,\be)\geq 0.
\end{align*}
\end{proposition}
\begin{proof}
The functional $\OTb$ is convex in $\al$ and in $\be$. 
Theorem~\ref{thm-diff-unb} holds, Thus $\OTb$ is differentiable. 
The first order convexity inequality gives
\begin{align*}
  &\OTb(\al,\be) \geq \OTb(\be,\be) + \dotp{\al-\be}{\nabla_1 \OTb(\be,\be)}\\
  &\OTb(\al,\be) \geq \OTb(\al,\al) + \dotp{\be-\al}{\nabla_2 \OTb(\al,\al)}\,.
\end{align*}
Applying Theorem~\ref{thm-diff-unb} and Lemma~\ref{lem-sym-pot}, gradients $\nabla_1 \OTb(\be,\be)$ and $\nabla_2 \OTb(\al,\al)$ verify $\nabla\Fb(\al) = -\nabla_1 \OTb(\al,\al) + \epsilon m(\al)=-\nabla_2 \OTb(\al,\al) + \epsilon m(\al)$.
Summing the above inequalities thus yields
\begin{align*}
   2 \OTb(\al,\be) &\geq \OTb(\al,\al) + \OTb(\be,\be) \\
  &+  \dotp{\al-\be}{-\nabla\Fb(\be) +\epsilon m(\be)} + \dotp{\be-\al}{-\nabla\Fb(\al) +\epsilon m(\al)},\\
   2 \OTb(\al,\be) &\geq \OTb(\al,\al) + \OTb(\be,\be) \\
   &+ \dotp{\al - \be}{\nabla\Fb(\al) - \nabla\Fb(\be)} - \epsilon (m(\al) - m(\be))^2,\\
  \Sb(\al,\be) &\geq \tfrac{1}{2} \Hb(\al,\be).
\end{align*}
Finally, we apply Proposition~\ref{thm-Feps_uniqueness}. 
The Hausdorff divergence is a Bregman divergence associated to $\Fb$ which is convex. 
Thus $\Hb$ is positive. 
\end{proof}

\begin{proposition}[The Sinkhorn divergence is bounded from below by a kernel norm]\label{prop-Seps-ineq-norm}
For any entropy $\phi$, write $(\f_\al,\g_\be)$ optimal symmetric potentials for $\Fb(\al)$ and $\Fb(\be)$. 
Then, one has
\begin{align}
\Sb(\al,\be) \geq \tfrac{\epsilon}{2} \Vert  \al e^{\frac{\f_{\al}}{\epsilon}} - \be e^{\frac{\g_{\be}}{\epsilon}} \Vert^2_{k_\epsilon}.
\end{align}
\end{proposition}
\begin{proof}
The pair $(\f_\al,\g_\be)$ is suboptimal in $\OTb(\al,\be)\geq\Ff(\f_\al,\g_\be)$. The detailed calculation is deferred in Appendix~\ref{appendix-proofs}.
\end{proof}

\begin{remark}\label{rem-bound-sinkdiv}
	When $(\al,\be)=(\de_x,\de_y)$, the first bound is sharp, while the second is not and approaches $0$ as $\C(x,y)\rightarrow\infty$.
\end{remark}

Finally, we show how the entropic regularization impacts the behaviour of $\Sb$ when $\epsilon\rightarrow\infty$.

\begin{proposition}[Behaviour of the Sinkhorn divergence when $\epsilon$ tends to infinity]\label{prop-asymptotic-unb}
For any entropy $\phi$, any measures $(\al,\be)$ such that $(m(\al),m(\be))\in \text{\upshape dom}(\phi)$. 
One has when $\epsilon\rightarrow\infty$,
\begin{align*}
  \OTb(\al,\be) \rightarrow\, &\dotp{\al}{\C\star\be} + m(\al)\phi(m(\be)) + m(\be)\phi(m(\al)).\\
   \Sb(\al,\be) =  &\norm{\al - \be}^2_{-\C} + (m(\al) - m(\be))(\phi(m(\be)) - \phi(m(\al)))\\
  &+ \tfrac{\epsilon}{2}(m(\al) - m(\be))^2  + o(1).
\end{align*}
\end{proposition}
\begin{proof}
The plan $\pi=\al\otimes\be$ is suboptimal in the primal~\eqref{eq-primal-unb}. 
One has $\pi_1 = m(\be)\al$ and $\pi_2=m(\al)\be$.
Since $(m(\al),m(\be))\in \text{\upshape dom}(\phi)$, it yields
\begin{align*}
  \OTb(\al,\be) \leq \dotp{\al\otimes\be}{\C} + m(\al)\phi(m(\be)) + m(\be)\phi(m(\al)).
\end{align*}
Second, let us focus on the dual formulation~\eqref{eq-dual-unb}. 
Consider constant potentials $(\f^*,\g^*)$ such that $-\f^* \in\nabla\phi(m(\be))$ and $-\g^* \in\nabla\phi(m(\al))$
They satisfy the optimality condition~\ref{eq-dual-optimality} when $\epsilon\rightarrow\infty$. 
Such $(\f^*,\g^*)$ exist because $(m(\al),m(\be))\in \text{\upshape dom}(\phi)$, thus $\partial\phi\neq\emptyset$. 
This is equivalent to 
\begin{align}
  \phi(m(\be)) = -\f^* m(\be) - \phi^*(-\f^*)\qandq
  \phi(m(\al)) = -\g^* m(\al) - \phi^*(-\g^*).\label{eq-suboptim-cond-legendre}
\end{align}
By suboptimality of $(\f^*,\g^*)$ we have $\OTb(\al,\be)\geq\Ff(\f^*,\g^*)$. 
It holds for any $\epsilon>0$, thus at the limit $\epsilon\rightarrow\infty$, a Taylor expansion of $\epsilon (e^{(\f\oplus\g-\C)/\epsilon} - 1)$ yields
\begin{align}
  &\lim_{\epsilon\rightarrow\infty} \OTb(\al,\be)
   	\geq \dotp{\al}{-\phi^*(-\f^*)} + \dotp{\be}{-\phi^*(-\g^*)} + \dotp{\al\otimes\be}{\C - (\f^*\oplus\g^*)} \label{eq-asymptotic-dual1}\\
  &\qquad\geq \dotp{\al\otimes\be}{\C} + m(\al)\big(-\f^* m(\be) - \phi^*(-\f^*)\big) + m(\be)\big(-\g^* m(\al) - \phi^*(-\g^*)\big)\label{eq-asymptotic-dual2}\\
  &\qquad\geq \dotp{\al\otimes\be}{\C} + m(\al)\phi(m(\be)) + m(\be)\phi(m(\al)).\label{eq-asymptotic-dual3}
\end{align}
Equation~\eqref{eq-asymptotic-dual2} is a simplification of Equation~\eqref{eq-asymptotic-dual1} because the potentials $(\f^*,\g^*)$ are constant. 
Equation~\eqref{eq-asymptotic-dual3} applies Equation~\eqref{eq-suboptim-cond-legendre}.
We have $\lim_{\epsilon\rightarrow\infty}\OTb(\al,\be)$.
Summing all the terms of the Sinkhorn divergence gives the second formula.
\end{proof}

This result shows that $\Sb(\al,\be)$ diverges as $\epsilon \rightarrow +\infty$ when $m(\al) \neq m(\be)$.
Note that this proof avoids $\Gamma$-convergence arguments.
For balanced OT one would take (not constant) $\f^*=\C\star\be - \tfrac{1}{2}\dotp{\al\otimes\be}{\C}$ and $\g^*=\C\star\al - \tfrac{1}{2}\dotp{\al\otimes\be}{\C}$.

\subsection{Positive definiteness of the Sinkhorn divergence}
\label{sec-pos-sink-div}

We present now the main results of this section on $\Sb$.

\begin{theorem}[The Sinkhorn divergence $\Sb$ is positive, definite and convex] \label{thm-sink-unb}
  Assume $\C$ is symmetric, $\gamma$-Lipschitz, and that $k_\epsilon = e^{- \C / \epsilon}$ is a positive universal kernel. 
  For any $\epsilon >0$, for any entropy $\phi$, the Sinkhorn divergence $\Sb(\al,\be)$ is positive, definite and convex in $\al$ and $\be$ (but not jointly).
\end{theorem}
\begin{proof}
The kernel $k_\epsilon$ is positive, thus it defines a positive kernel norm. 
Applying Proposition~\ref{prop-Seps-ineq-norm}, we get that $\forall(\al,\be)\in\Mmp(\Xx),\, \Sb(\al,\be)\geq 0$.

The function $(\al,\be)\mapsto(m(\al) - m(\be))^2$ is convex,  $\Fb$ is convex (Theorem~\ref{thm-Feps_uniqueness}) and $\OTb$ is convex in each of its inputs (Theorem~\ref{thm-continuity-unb}). 
Summing everything proves $\Sb$ is convex in $\al$ and in $\be$.

Proving definiteness holds with Propositions~\ref{prop-Seps-ineq-norm} and Lemma~\ref{lem-injective-sym-pot}. 
If $\Sb(\al,\be)=0$, then so is the kernel norm. 
Since $k_\epsilon$ is universal, we get $\al e^{\f_{\al} / \epsilon} = \be e^{\g_{\be} / \epsilon}$, which implies that $\al=\be$ thanks to Lemma~\ref{lem-injective-sym-pot}.
\end{proof}

This last theorem focuses on properties of $\Sb$ with respect to the weak* topology when $\D_\phi=\rho\KL$ or $\rho\TV$.
While taking such $\D_\phi$ seems restrictive, they are the two settings most frequently studied in the litterature.

\begin{theorem}[The Sinkhorn divergence $\Sb$ metrizes the convergence in law]
	\label{thm-sink-weak-cv}
When $\D_\phi = \rho\KL$ or $\rho\TV$,
$\Sb$ metrizes the convergence in law: for any sequence $(\al_n)_n$ in
$\Mmpp(\Xx)$, we have $\al_n\rightharpoonup\al \Longleftrightarrow \Sb(\al_n,\al) \rightarrow 0.$
\end{theorem}
\begin{proof}
Assume $\al_n\rightharpoonup\al$. 
Theorem~\ref{thm-diff-unb} and Proposition~\ref{prop-uniform-conv} gives that $\Sb$ is weak* continuous.
By definition $\Sb(\al,\al)=0$, thus $\Sb(\al_n,\al)\rightarrow 0$.

Conversely, Assume $\Sb(\al_n,\al)\rightarrow 0$. 
Assume $(m(\al_n))_n$ is uniformly bounded.s
Since $\Xx$ is compact, Banach-Alaoglu theorem gives $(\al_n)$ is a compact sequence.
Take any weak limit $\al_{n_\infty}$ of a subsequence $(\al_{n_k})_k$.
By continuity $\Sb(\al_{n_\infty},\al)=0$, and definiteness implies $\al_{n_\infty}=\al$
Thus $(\al_n)$ has a unique limit and converges to $\al$.

It remains to prove $(m(\al_n))_n$ is uniformly bounded. 
Write $\bar{\al}_n = \al_n / m(\al_n)$. 
When $\D_\phi = \rho\KL$, noting $\f_n$ and $\bar{\f}_n$ optimal potentials of $\Fb(\al_n)$ and $\Fb(\bar{\al}_n)$.
Linking optimality conditions of $(\al_n, \bar{\al}_n)$, one obtains the relation $\f_n = \bar{\f}_n -\tfrac{\rho\epsilon}{2\rho + \epsilon}\log(m(\al_n))$.

Since $\Sb(\al_n,\al)\rightarrow 0$, Proposition~\ref{prop-Seps-ineq-norm} gives that $\norms{\al_n e^{\f_n / \epsilon}}_{k_\epsilon}\rightarrow \norms{\al e^{\f / \epsilon}}_{k_\epsilon}$ where $\f$ is optimal for $\Fb(\al)$. 
The optimality condition of $\f_n$ reads $e^{-\f_n / \rho}=e^{\f_n / \epsilon}\dotp{\al_n}{e^{(\f_n - \C) / \epsilon}}$. 
Thus we reformulate the kernel norm as
\begin{align*}
    \norm{\al_n e^{\f_n / \epsilon}}^2_{k_\epsilon} = \dotp{\al_n\otimes\al_n}{e^{(\f_n\oplus\f_n-\C)/\epsilon}}
    = \dotp{\al_n}{e^{-\f_n / \rho}}
     =\dotp{\bar{\al}_n}{e^{-\bar{\f}_n / \rho}} m(\al_n)^{\frac{\epsilon}{2\rho + \epsilon} + 1}.
\end{align*}
The sequence $\norms{\al_n e^{\f_n / \epsilon}}_{k_\epsilon}$ converges, so it is bounded. 
If $m(\al_n)\rightarrow\infty$, it imposes that $\dotp{\bar{\al}_n}{e^{-\bar{\f}_n} / \rho}$ converges to $0$.
 Since $\bar{\al}_n$ is a probability on a compact space, it imposes $\norm{\bar{\f}_n}_\infty\rightarrow\infty$, which contradicts the coercivity of $\Ff$ and the optimality of $\bar{\f}_n$ since $\Fb(\bar{\al}_n)>\infty$.
For $\D_\phi=\rho\TV$, $\aprox{\phi^*}$ imposes $\f_n\geq -\rho$, thus one has $M > \norms{\al_n e^{\f_n / \epsilon}}_{k_\epsilon} \geq e^{(-2\rho-\text{diam}(\Xx))/\epsilon}m(\al_n)^2$.
In both cases the mass is necessarily bounded, which ends the proof on the weak* metrization of $\Sb$.
\end{proof}

\subsection{Case of the null measure}
\label{sec-null-meas}

The case $\al=0$ needs to be treated separately because dual potentials lack regularity. 
Indeed, If $\al=0$ then $\al\otimes\be=0$ and the regularization $\KL(.,\al\otimes\be)$ imposes that the only feasible plan is $\pi=0$.
Thus the primal cost is equal to $\OTb(\al,\be) = m(\be) \phi(0)$. 
Note that we assume $\phi(0) < +\infty$, otherwise $\OTb(0,\be)=+\infty$. 
In that case the primal is well-defined with an explicit formula. 
Concerning the dual, it reads $\OTb(\al=0,\be) = \sup_{\g\in\Cc(\Xx)} \dotp{\be}{-\phi^*(-\g)}$.
When $\D_\phi = \KL$, the dual program is equal to the primal, but the sup is not attained in $\Cc(\Xx)$ because $\g=+\infty$ is optimal.
Thus, we cannot use the regularity of dual potentials given in Proposition~\ref{prop-uniform-conv} to prove the regularity of OT when any of the input measures is null.
Nevertheless, it is possible to prove via the primal that OT functionals are regular when the input measures go to zero.

\begin{proposition}[Continuity of unbalanced OT at the null measure]
Assume $\phi$ is a continuous entropy with $\text{\upshape dom}(\phi) = \R_+$. 
Take $(\al_n,\be_n)\rightharpoonup(0,\be)$ with $\be\in\Mmp(\Xx)$. 
Then $\OTb$ is weak* continuous at $(0,\be)$, $\Fb$ is weak* continuous and $\Sb$ is weak* continuous and positive at $(0,\be)$ under the assumptions of Theorem~\ref{thm-sink-unb}.
\end{proposition}
\begin{proof}
The plan $\pi_n = \al_n\otimes\be_n$ is feasible (since $\text{dom}(\phi) = \R_+$) and suboptimal. 
It yields an upper bound on $\OTb$. 
Jensen inequality on $\D_\phi$ (which is also positive) gives a lower bound. 
They read
\begin{align*}
  \OTb(\al_n,\be_n) &\geq  \inf_{\pi\in\Mmp(\Xx^2)} \dotp{\pi}{\C} + m(\al_n)\phi( \frac{m(\pi_{n})}{m(\al_n)} ) + m(\be_n)\phi( \frac{m(\pi_{n})}{m(\be_n)} ),\\
  \OTb(\al_n,\be_n) &\leq  \dotp{\al_n\otimes\be_n}{\C} + m(\al_n)\phi(m(\be_n)) + m(\be_n)\phi(m(\al_n)).
\end{align*}
The lower bound is an infimum on a lower semicontinuous functional and is thus bounded from below by the infimun of the limit $\al_n\rightharpoonup 0$ and $\be_n\rightharpoonup 0$. 
Thus $\pi=0$, since other plans yield an infinite cost, and $\OTb(\al_n,\be_n)\geq m(\be) \phi(0)$. 
The upper bound gives at the limit $\OTb(\al_n,\be_n)\rightarrow m(\be) \phi(0) = \OTb(0,\be)$ (because $\phi$ is continuous), which proves the weak* continuity of $\OTb$.
Concerning $\Fb$, the same proof holds using the suboptimal plan $\pi=\al_n\otimes\al_n$.
The Sinkhorn divergence $\Sb$ is positive for strictly positive measures and weak* continuous as a sum of weak* continuous functions. Thus when $\al_n\rightharpoonup 0 $ the positivity remains at the limit.
\end{proof}

\subsection{Extensions}
We detail here extensions of the theory we developped above.
They were not considered for the sake of simplicity, but should be worth considering for applications.
We provide motivated examples for such ideas, with details on how our theory should be adapted.

\subsubsection{Assymetric marginal penalties}
As suggested Remark~\ref{rem-asym-entropy}, one could want to consider penalties $\D_{\phi_1}(\pi_1|\al)$ and $\D_{\phi_2}(\pi_2|\be)$ with $\phi_1\neq\phi_2$.
For instance, take $\D_{\phi_1}=\iota_{(=)}$ and $\D_{\phi_2}=\rho\KL$.
It is relevant in e.g. domain adaptation where $\al$ is a source dataset on which a predictor was trained, and $\be$ is a similar but shifted dataset on which we want to transfer the learned predictor.

In this setting it is possible to define a Sinkhorn divergence which would be positive, but no longer symmetric.
It then reads
\begin{align*}
	\Sb^{(\phi_1,\phi_2)}(\al,\be) &= \OTb^{(\phi_1,\phi_2)}(\al,\be) - \OTb^{(\phi_1,\phi_1)}(\al,\al) - \OTb^{(\phi_2,\phi_2)}(\be,\be) 
	+ \tfrac{\epsilon}{2}(m(\al) - m(\be))^2,
\end{align*}
where $\OTb^{(\phi_1,\phi_2)}$ is the regularized OT program penalized with $(\D_{\phi_1},\D_{\phi_2})$.
Using the above formula, it is straightforward to prove Proposition~\ref{prop-Seps-ineq-norm}, hence the positivity of $\Sb^{(\phi_1,\phi_2)}$.
To compute $\OTb^{(\phi_1,\phi_2)}$ the only change is to consider two operators $\aprox{\phi^*_i}$ ($i\in\{1,2\}$) such that optimal potentials satisfy $\f = -\aprox{\phi^*_1}(-\Ss_\be(\g))$ and $\g = -\aprox{\phi^*_2}(-\Ss_\al(\f))$.

\subsubsection{Spatially varying $\phi$-divergences}
Recall Csiszàr divergence integrate pointwise penalties on $\tfrac{\d\al}{\d\be}$.
It is thus possible to generalize $\D_\phi$ as
\begin{align}\label{eq-csiszar-div-varying}
\D_\phi(\al|\be) \eqdef \int_\Xx \phi\big(\frac{\d\al}{\d\be}(x),x\big) \d\be +  \int_\Xx \phi^\prime_\infty(x) \d\al^\bot(x),
\end{align}
where $\phi(\cdot,x)$ is an entropy function for each location $x \in \Xx$ with associated recession value  $\phi^\prime_\infty(x)$.
Some regularity is however required to avoid measurability issues and be able to apply Legrendre duality.
It is well-defined when the function (defined on $\mathbb{R}_+^2 \times \Xx$) $\Phi : (r,s,x) \mapsto \phi(r/s,x) s$
(properly extended when $s=0$ using $\phi^\prime_\infty(x)$) is a so-called normal-integrant~\cite[chap.14]{rockafellar2009variational}.
For instance, this is ensured if $\Phi$ is lower-semi-continuous.

A typical example of such divergence consists in using a spatially varying parameter $\rho(x)$, such that for e.g. $\KL$ penalties one takes $\phi(p,x)=\rho(x) ( p \log p -p +1)$.
It allows to modulate the strength of the conservation of mass constraint over the spatial domain $\Xx$.
Such situation appears e.g. in biology where the frequency of cell duplications $\rho(x)$ depends on the cell functionality $x$.

Concerning computations, one takes a spatially varying map $\aprox{\phi^*(\cdot,x)}$ at each $x\in\Xx$.
Note that when $\phi(p,x)=\rho(x)\phi(p)$ one has $\phi^*(p,x) = \rho(x)\phi^*(p / \rho(x))$.
The full Sinkhorn update outputs the function $x\mapsto -\aprox{\phi^*(\cdot,x)}(-\Ss_\al(\f)(x))$, and for $\KL$ penalties mentioned above, it reads $\aprox{\phi^*(\cdot,x)}(q) = (1 + \tfrac{\epsilon}{\rho(x)})^{-1}q$.


\section{Statistical Complexity of Unbalanced Transport}
\label{sec-stat-comp}

A common assumption in statistics, machine learning and imaging is that one does not have directly access to the distributions $(\al,\be)$, but rather that the data is composed of a set of $n$ samples from these measures. 
Thus an important theoretical and practical question is the study of the discretization error made when approximating $\OTb(\al,\be)$ with $\OTb(\al_n,\be_n)$. 

More precisely, we wish to establish the convergence rate of $|\OTb(\al_n,\be_n)- \OTb(\al,\be)|$ as $n\rightarrow\infty$ so as to know how many samples are needed to reach a desired tolerance error. 
For unregularized OT the rate is $O(n^{-1 / d})$ when $\Xx=\R^d$~\cite{dudley1969speed}. 
It was refined in~\cite{weed2017sharp} to be $O(n^{-1 / d^*})$ where $d^*$ is a quantification of the intrinsic dimension of the measure.
Entropic regularization has been proved to mitigate this curse of dimensionality, yielding in $\R^d$ when $\epsilon\rightarrow 0$ a rate of $O(\epsilon^{- \lfloor d / 2 \rfloor}n^{-1/2})$~\cite{genevay2018sample}, with an improvement of the dependency with $\epsilon$ of the constant in~\cite{mena2019statistical} (which also extends this result from compact domains to sub-Gaussian measures).

A recent work shows the statistical and time benefits of using $\Sb$ in the Balanced case instead of $\OTb$~\cite{chizat2020faster}.
It allows to obtain accurate approximations of $\OT$ while allowing a larger regularization $\epsilon$ compared to using $\OTb$.
%
%
This proof relies on a dynamic formulation of entropic OT.
An unbalanced entropic dynamic formulation was recently developed in~\cite{baradat2021regularized}, but its only connected to $\OTb$ when $\D_\phi=\TV$ and $\C(x,y)=|x-y|^2$.
In this section we consider general (but smooth) $\phi$ and $\C$, which excludes the TV case.
For this reason, our results which focuses on $\OTb$ instead of $\Sb$ remain of interest to the community.

This section extends the results of~\cite{genevay2018sample,mena2019statistical} to the framework of unbalanced OT. 
We suppose in addition with all the previous assumptions that the cost $\C$ and the function $\phi^*$ are $\Cc^\infty$. We assume the space $\Xx$ is a compact Lipschitz domain of $\R^d$. 

We denote by $(\al,\be)\in\Mmp(\Xx)$ the input positive measures, by $(\bar{\al},\bar{\be})\in\Mmpo(\Xx)$ their normalized versions and by $(\al_n,\be_n)$ their empirical counterparts with $n$ points, i.e.
\begin{align*}
  \al = m(\al)\bar{\al}, \quad \al_n = \frac{m(\al)}{n} \sum_{i=1}^n \de_{X_i} \qquad \be = m(\be)\bar{\be}, \quad \be_n = \frac{m(\be)}{n} \sum_{i=1}^n \de_{Y_i},
\end{align*}
where $(X_1,...,X_n)$ and $(Y_1,...,Y_n)$ are $n$ points in $\Xx$ sampled from the normalized probability distributions $(\bar{\al},\bar{\be})$. 
Note that we assume for simplicity that the masses of $(\al,\be)$ are known, so that the total masses of $(\al_n,\be_n)$ are the same as those of $(\al,\be)$.

The main result of this section is the following theorem. Its proof is very technical and is detailed in Appendix~\ref{appendix-statistical-complexity}.
\begin{theorem}[Sample complexity of the unbalanced transport cost]\label{thm-sample-complexity-unb}
Assume $\phi^*$ and  $\C$ are $\Cc^\infty$ and that Assumptions~(\ref{as:1}, \ref{as:2}) hold. 
Then there exists a rational fraction $\Qq(\epsilon)$ whose coefficients only depend on the norms $\norms{\C^{(k)}}_\infty$ and $\norms{\phi^{*(k)}}$, respectively evaluated on compact sets $\Xx$ and $\Yy$ where $\Yy$ is a compact independent of $\epsilon$, such that
\begin{align*}
  For\, any \,\, \epsilon,\;\; \mathbb{E}_{\bar{\al}\otimes\bar{\be}}\big[|\OTb(\al,\be) - \OTb(\al_n,\be_n)|\big] = O\bigg(\frac{m(\al) + m(\be)}{\sqrt{n}}\Qq(\epsilon)\bigg).
\end{align*}
Furthermore the rational fraction has the following asymptotics.
\begin{align*}
\Qq(\epsilon) = O_{\epsilon\rightarrow 0}(\epsilon^{-\lfloor d/2 \rfloor}) \qandq \Qq(\epsilon) = O_{\epsilon\rightarrow\infty}(1).
\end{align*}
\end{theorem}

The proof of this result relies on several lemmas presented in Appendix~\ref{appendix-statistical-complexity}. 
In particular, we show dual potentials are smooth and belong to a Sobolev space $\Hh^s_\al(\Xx)$, which is a RKHS when $s > \lfloor \tfrac{d}{2} \rfloor$. %
Note that for a given dimension $d$, it suffices to assume $\C$ and $\phi^*$ are $\Cc^{\lfloor d /2 \rfloor + 1}$. 
We show that the potentials lie in a ball of $\Hh^s_\al(\Xx)$ endowed with its corresponding Sobolev norm, 
Then we apply standard results from the PAC-learning theory in Reproducing Kernel Hilbert Spaces.
\begin{proof}
We first start by applying Proposition~\ref{prop-ineq-ot-sob}
\begin{align*}
 \mathbb{E}_{\bar{\al}\otimes\bar{\be}}\big[|\OTb(\al,\be) - \OTb(\al_n,\be_n)|\big]  &\leq 2 \mathbb{E}_{\bar{\al}}\big[\sup_{\f\in\Hh^s_{\al,\lambda}(\R^d)} |\dotp{\al - \al_n}{\f}|\big] \\
  &+2 \mathbb{E}_{\bar{\be}}\big[\sup_{\f\in\Hh^s_{\be,\lambda}(\R^d)} |\dotp{\be - \be_n}{\g}|\big].
\end{align*}

Write $\al = m(\al) \bar{\al}$ and $\be = m(\be) \bar{\be}$. We apply Proposition~\ref{prop-pac-rkhs} with $B=1$ to the Sobolev space $\Hh^s_{\al,\lambda}(\R^d)$ with $s = \lfloor \tfrac{d}{2}\rfloor + 1$, such that $\Hh^s_{\al,\lambda}(\R^d)$ and $\Hh^s_{\be,\lambda}(\R^d)$ are RKHS. It yields for the normalized measure $\bar{\al}\in\Mmpo(\Xx)$
\begin{align*}
  \mathbb{E}_{\bar{\al}}\bigg[ \sup_{\f\in\Hh^s_{\al,\lambda}(\R^d)} |\dotp{\bar{\al} - \bar{\al}_n}{\f}| \bigg] \leq \frac{2\lambda}{\sqrt{n}}
\end{align*}
where $\lambda=\Qq(\epsilon)$ is the radius of the Sobolev ball bounding the potentials (Proposition~\ref{prop-dual-pot-sob}). 
We get the desired result by multiplying by $m(\al)$, and summing with the similar term obtained fo $\be$.
\end{proof}

\section{Implementation}
\label{sec-implementation}

We detail here how to compute all divergences defined Section~\ref{sec-ot-prop} when $(\al,\be)$ are discrete measures.
It takes two steps.
Firstly $(\f_{\al\be},\g_{\al\be}, \f_\al, \g_\be)$ are computed using the Sinkhorn algorithm~\ref{def-sinkhorn}.
Secondly, potentials are summed against the input measures as described for instance in Proposition~\ref{prop-funct-kl} when $\D_\phi=\rho\KL$.

\subsection{Sinkhorn algorithm}

\paragraph{Discrete Setting}
We write discrete measures as $\al = \sum_{i=1}^\N \AL_i \de_{x_i}$ and $\be=\sum_{j=1}^\M \BE_j \de_{y_j}$, where $(\AL_i)_i,(\BE_j)_j \in\R_+^\N$ are vectors of non-negative masses and $(x_i)_{i},(y_j)_{j} \in \Xx^\N$ are two sets of points.
Potentials $(\f_i) = (\f(x_i))$ and $(\g_j) = (\g(y_j))$ become two vectors of $\R^\N$ and $\R^\M$.
The cost $\C_{ij}=\C(x_i,y_j)$ and the transport plan $\pi_{ij}=\pi(x_i,y_j)$ become matrices of $\R^{\N\times\M}$.
The latter can be computed with Equation~\eqref{eq-implicit-plan} which becomes $\pi_{ij} =  \exp\tfrac{1}{\epsilon}[\f_i + \g_j - \C_{ij}]\al_i \be_j$.
Once potentials $(\f,\g)$ are computed by the Sinkhorn algorithm, functionals of Definition~\ref{def-sink-div-unb} involve discrete sums such as $\dotp{\al}{\phi^*(-\f)} = \sum_{i=1}^N \AL_i \phi^*(-\f_i)$.

\paragraph{Computational routines}
The Sinkhorn algorithm allows parallel computations, and is thus ideally suited to modern computing hardware (e.g. GPU).
In practice, we rely on the standard NumPy~\cite{numpy} and PyTorch~\cite{pytorch} libraries for array manipulations and display our results using Matplotlib~\cite{matplotlib}.
When using GPUs, we rely on the KeOps library~\cite{charlier2020kernel,feydy2020fast} to perform fast computations, with a negligible memory footprint -- which is often a bottleneck with GPUs.

Contrary to~\cite{chizat2016scaling} where $\proxdiv{\phi}$ is used (see Section~\ref{sec-operators}), the use of $\aprox{\phi^*}$ allows to perform iterations on \emph{log-domain}.
We update $(\f,\g)$ instead of $(e^{\f / \epsilon}, e^{\g/\epsilon})$, which is key for numerical stability.
Indeed, operators $(\Ss_\al,\Ss_\be)$ (see Equation~\eqref{eq-defn-softmin-func}) are Log-Sum-Exp reductions, an operation which can be stabilized as
\begin{align*}
	\LSE_{i=1}^\N (u_i) \eqdef \log\textstyle\sum_{i=1}^\N \exp(u_i) =  \max_{k} u_k + \log\textstyle\sum_{i=1}^\N \exp(u_i - \max_{k} u_k).
\end{align*}
Such expression avoids numerical overflows of exponentials since $u_i - \max_{k} u_k\leq 0$.
It also avoids underflows in the sense that updates of $(e^{\f / \epsilon}, e^{\g/\epsilon})$ involve the matrix $(e^{-\C_{ij}/\epsilon})$ whose coordinates are numerically underflowing to $0$ for small $\epsilon$.


\paragraph{Algorithm}
We detail the implementation in Algorithm~\ref{alg:sinkhorn}.
We emphasize that the only change from \emph{balanced} Sinkhorn is the extra composition with the operator $\aprox{\phi^*}$ with the Log-Sum-Exp reduction.
As detailed in Section~\ref{sec-exmp-f-div}, $\aprox{\phi^*}$ is often cheap too compute, thus not impacting the computation cost of Sinkhorn.


{\centering
\begin{minipage}{\linewidth}
\begin{algorithm}[H]
	\caption{~~~~\, {Sinkhorn  Algorithm: Sink($(\AL_i)_i$, $(x_i)_i$, $(\BE_j)_j$, $(y_j)_j$)} \label{alg:sinkhorn}}
	\textbf{Parameters~:}~~ symmetric cost function $\C(x,y)$, regularization $\epsilon > 0$ \\
	\textbf{Input~~~~\,:}~~ source $\alpha = \sum_{i=1}^\N \AL_i\delta_{x_i}$,~target~ $\beta = \sum_{j=1}^\M \BE_j\delta_{y_j}$\\
	\textbf{Output~~\,:}~~ vectors $(\f_i)_i$ and $(\g_j)_j$, equal to the optimal potentials\\
	\vspace*{-1.0em}
	\begin{algorithmic}[1]
		\STATE $\f_i\gets \text{zeros}(\M)$~~;~~$\g_j\gets \text{zeros}(\N)$ \COMMENT{Vectors of size $\M$ and $\N$}
		\vspace{.1cm}\WHILE{updates $>$ tol}\vspace{.1cm}
		\STATE $\g_j \gets - \,\epsilon \LSE_{i=1}^\N \big[ \log(\AL_i) + (\f_i - \C(x_i,y_j))\,/\,\epsilon\,\big]$
		\STATE $\g_j \gets -\aprox{\phi^*}(-\g_j)$
		\label{alg:sinkhorn:line_a}
		\STATE $\f_i \gets - \,\epsilon \LSE_{j=1}^\M \big[ \log(\BE_j) + (\g_j - \C(x_i,y_j))\,/\,\epsilon\,\big]$
		\STATE $\f_i \gets -\aprox{\phi^*}(-\f_i)$
		\label{alg:sinkhorn:line_b}
		\ENDWHILE
		\RETURN{~~$(\f_i)_i,~~(\g_j)_j$}
	\end{algorithmic}
\end{algorithm}
\end{minipage}
}

\begin{remark}
	It is possible to implement Remark~\ref{rem-extrapolate-pot} to extrapolate potentials, which matters to compute the Hausdorff divergence.
	For instance, take $(f_i)_i$ s.t. $\f=\Aa\Ss_\al(\f)$ with $\al = \sum_{i=1}^\N \AL_i \de_{x_i}$.
	To evaluate at some $y$, we compute
	\begin{align*}
		\f(y) = -\aprox{\phi^*}\big(\epsilon\log \sum_{i=1}^\N \exp \big[ \log(\AL_i) + (\f_i - \C(x_i,y))\,/\,\epsilon\,\big]\big).
	\end{align*}
\end{remark}

\section{Numerical illustrations}
\label{sec:numerical}

\subsection{A synthetic example with gradient flows}
We now present numerical examples and applications of our results based on the algorithm of Section~\ref{sec-implementation}. Our implementation of the functionals and the code needed to reproduce the experiments below is available at:\\
\centerline{\url{https://github.com/thibsej/unbalanced-ot-functionals}.}

We present numerical experiments on gradient flows.
Given a set of particles $\theta = \{(x_i, r_i)_i \}$ with coordinates $x_i \in\R^d$ and masses $r_i \in\R_+$, one wishes to study their trajectories under a potential $\theta\mapsto F(\theta)$.
The particles intialized at $t=0$ by $\theta_0$ undergo the dynamic $\partial_t\theta(t) = - \nabla F(\theta(t))$.
Such flows have been extensively studied for Partial Differential Equations.
They also gained attention in machine learning to study convergence of neural networks~\cite{chizat2018global}.

We consider the same setting as~\cite{chizat2019sparse}.
We consider a target measure $\be\in\Mmp(\R^d)$ and the potential $\al\mapsto\Sb(\al,\be)$, as one would do in e.g. generative learning in imaging~\cite{arjovsky2017wasserstein}.
The measure $\al$ represents the model we train, parameterized as $\al_\theta = \sum_i^n  r_i^2 \de_{x_i}$ with parameter $\theta = ((x_i, r_i))_i\in(\R^2\times\R_+)^n$.
Minimizing $\Sb(\cdot,\be)$ amounts to run the following gradient steps
\begin{align*}
	x_i^{(t+1)} & = x_i^{(t)} - \eta_x \nabla_{x_i} \Sb(\al_\theta^{(t)}, \be),                  \\
	r_i^{(t+1)} & = r_i^{(t)}.\exp\big( - 2\eta_x \nabla_{r_i} \Sb(\al_\theta^{(t)}, \be) \big),
\end{align*}
where $(\eta_x, \eta_r)>0$ are two learning steps.
The update on $r_i$ is called a mirror descent step, and is used to enforce that $r_i\geq 0$.
We retrieve the exact gradient flow when $(\eta_x, \eta_r)\rightarrow 0$.
Using such model $\al_\theta$ and such updates is proved in~\cite{chizat2019sparse} to be equivalent to a gradient flow in the space $\Mmp(\Xx)$, in contrast with classical flows optimizing over $\Mmp_1(\Xx)$.

%
%

We run the experiments in several settings.
Wa always take the Euclidean distance $\C(x,y)=\norm{x-y}^2_2$ on the unit square $[0,1]^2$, constant learning rates $(\eta_x, \eta_r) = (60, 0.3)$, a radius $\sqrt{\rho}=\sqrt{10^{-1}}$, and a default blur radius of $\sqrt{\epsilon}=\sqrt{10^{-3}}$.
In each timeframe we display iterations $[5, 10, 20, 50, 300]$ of the gradient descent steps.
Each dot represents a particle, and the diameter represents its mass.

Figures~\ref{fig-flow-reg} (rows 1 and 2) show the difference between using $\OTb$ and the (debiased) Sinkhorn divergence $\Sb$.
Note that for $\OTb$ (row 1) the model $\al_\theta$ concentrates (i.e. suffers the \emph{entropic bias}) while for $\Sb$ it approaches $\be$ up to details of size $\sqrt{\epsilon}$.
One the same figure, comparing rows 2 and 3 shows the influence of $\epsilon$, which operates a low pass smoothing.
If $\epsilon$ is chosen too large then $\al_\theta$ discards finer details.
Figure~\ref{fig-flow-tv} shows the impact of changing $\D_\phi$ on the mass variation dynamics.
For instance, one retrieves a partial transport behaviour for $\rho\TV$.

\newcommand{\myrot}[1]{ \rotatebox{90}{\small \hspace{2mm} #1}}

\newcommand{\myimg}[1]{\includegraphics[width=0.16\textwidth]{images/#1}}
\newcommand{\myimgRow}[1]{ %
	\myimg{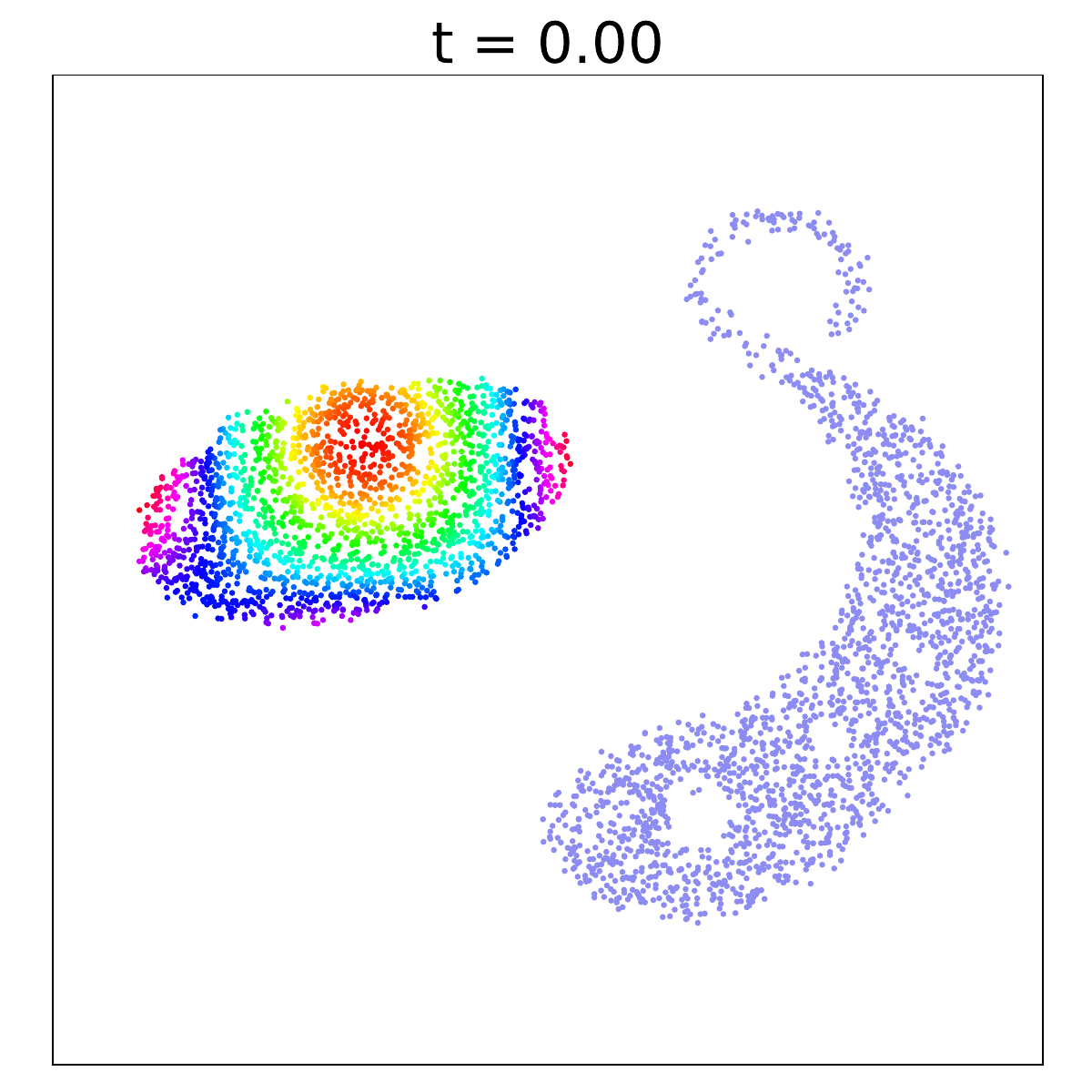} & %
	\myimg{#1_frame5-c} &  %
	\myimg{#1_frame10-c} & %
	\myimg{#1_frame20-c} & %
	\myimg{#1_frame50-c} & %
	\myimg{#1_frame300-c}   %
}
\newcommand{\myimgRowMod}[1]{ %
	\myimg{density_display-c} & %
	\myimg{#1_frame5-c} &  %
	\myimg{#1_frame10-c} & %
	\myimg{#1_frame20-c} & %
	\myimg{#1_frame50-c} & %
	\myimg{#1_frame300-c.png}   %
}

\begin{figure}[p]
	\centering
	\includegraphics[width=\linewidth]{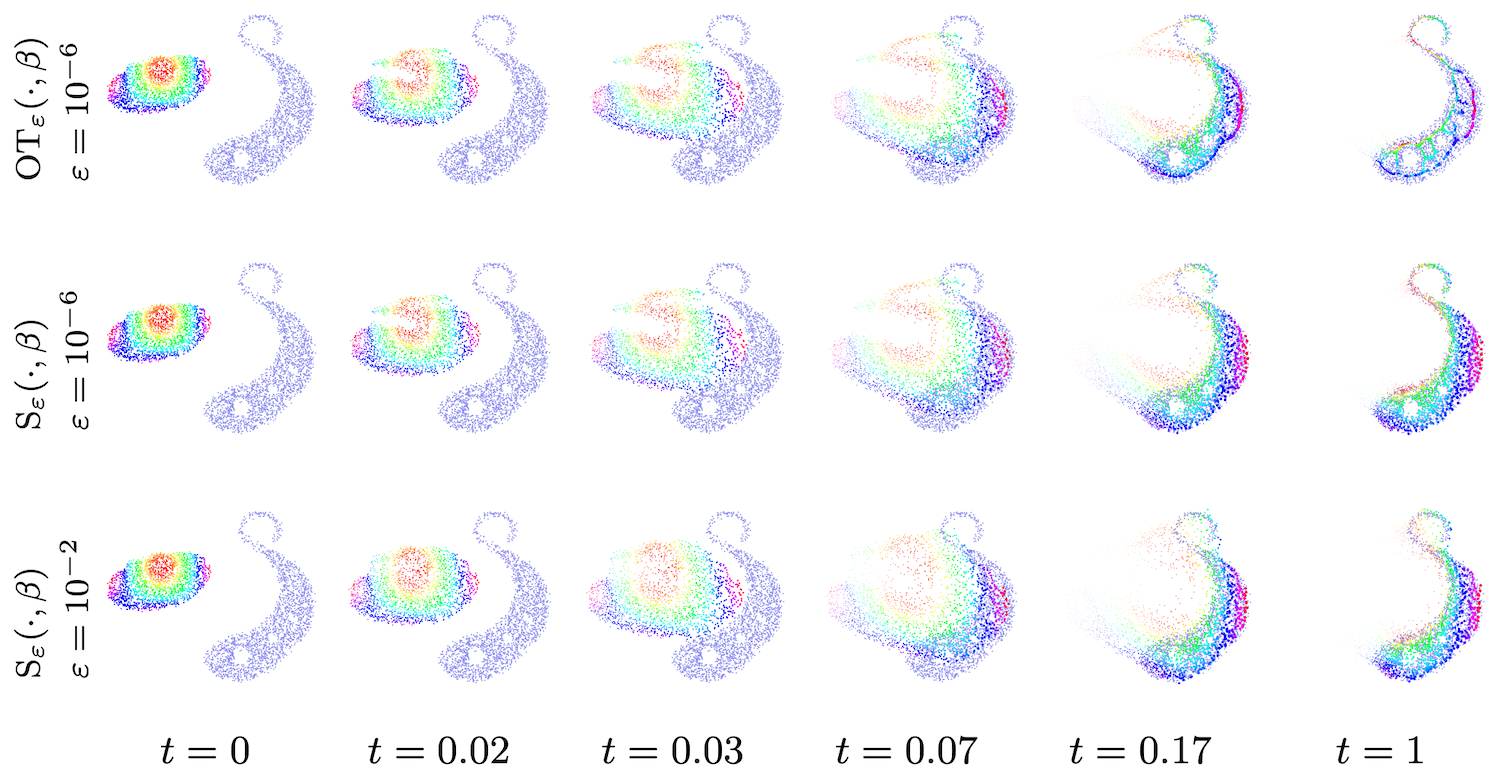}
	\caption{
		Comparison of gradient for of three different discrepancy when using $\D_\phi=\rho\KL$.
		The target measure $\be$ is displayed in blue, while the evolving measure $\al_t$ is displayed using a rainbow color scheme that allows us to track individual particles.
	}
	\label{fig-flow-reg}

	\vspace*{.5cm}

	\includegraphics[width=\linewidth]{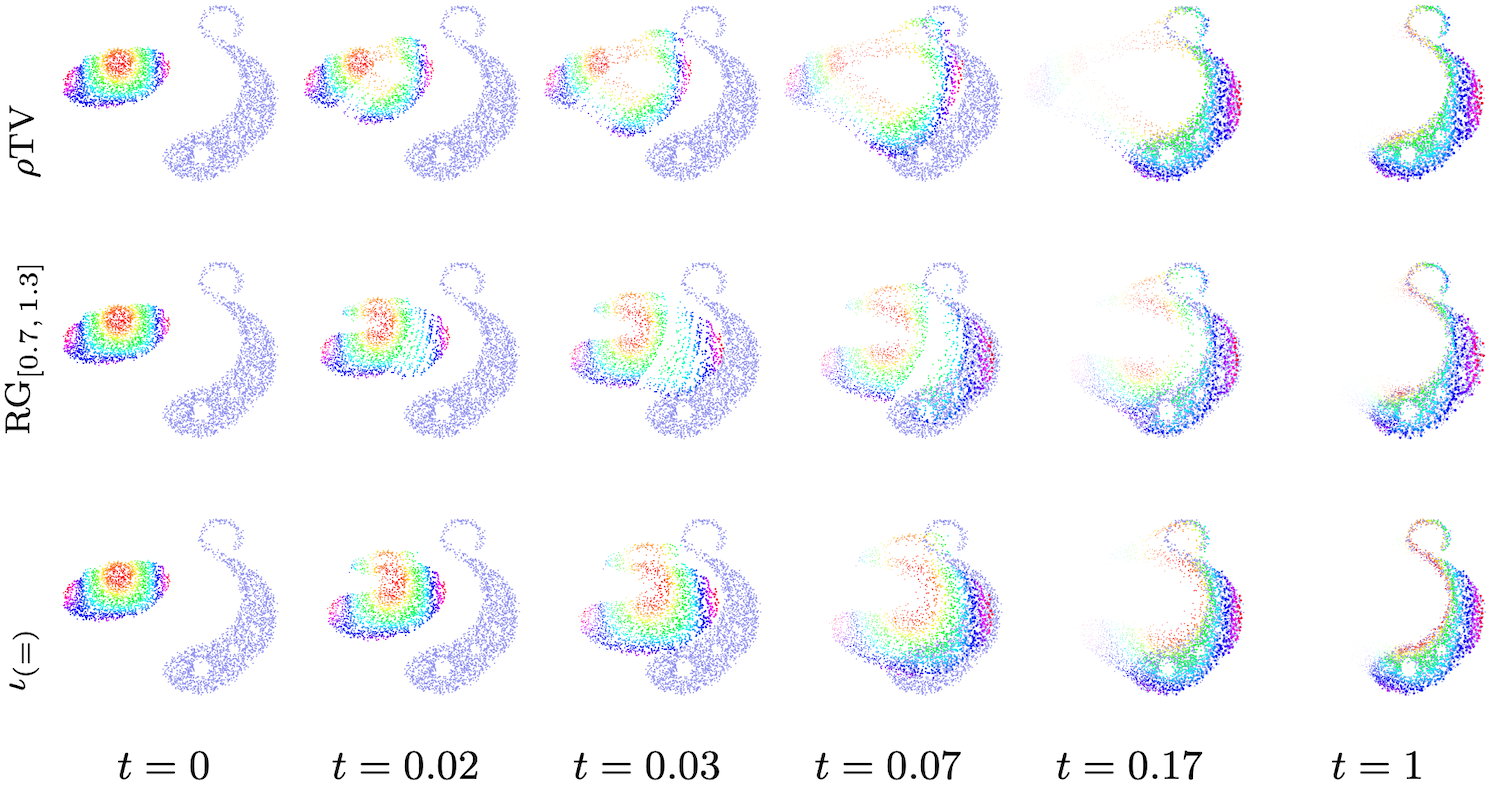}
	\caption{Flow of $\Sb(.,\be)$ with different type of divergence $\D_\phi$, from top to bottom: Total Variation (TV), range constraint and balanced OT. }
	\label{fig-flow-tv}
\end{figure}

\subsection{An application: 3D scene flow estimation}

The theory of unbalanced and entropy-regularized OT is motivated by
applications to noisy data. Our goal is to enable the use of transport-based tools
on problems that are a ``good but imperfect’’ fit
for the standard Monge--Kantorovitch model.

To make this point clear, we showcase the use of our robust OT tools on a real applied problem: the estimation of displacement vectors (``3D flow’’) between two views of the same 3D scene that have been acquired at times $t$ and $t+\Delta t$. This is a fundamental task in computer vision, with major applications to automated driving \cite{vedula1999three}.

As illustrated in Fig.~\ref{fig:kitti}, we consider two point clouds
$x_1$, \ldots, $x_\text{N}$ (``source frame’’) and
$y_1$, \dots, $y_\text{M}$ (``target frame’’) that have been acquired
by a binocular device or a LiDAR scanner.
For this experiment, we rely on a cropped scene from
the standard KITTI dataset \cite{kitti,flownet3d}.
We intend to estimate the positions
$p_1$, \dots, $p_\text{N}$ of all points $x_i$ at time $t+\Delta t$.
We stress that both 3D frames have
been sampled independently from each other,
which means that the Ground Truth results
$p_1^\text{GT}$, \dots, $p_\text{N}^\text{GT}$
that are provided in the dataset
are \emph{not} in perfect correspondence with
the target points $y_1$, \dots, $y_\text{M}$.

\begin{figure}[p]
	\centering
	\captionsetup{width=\linewidth}
	\includegraphics[width=\linewidth]{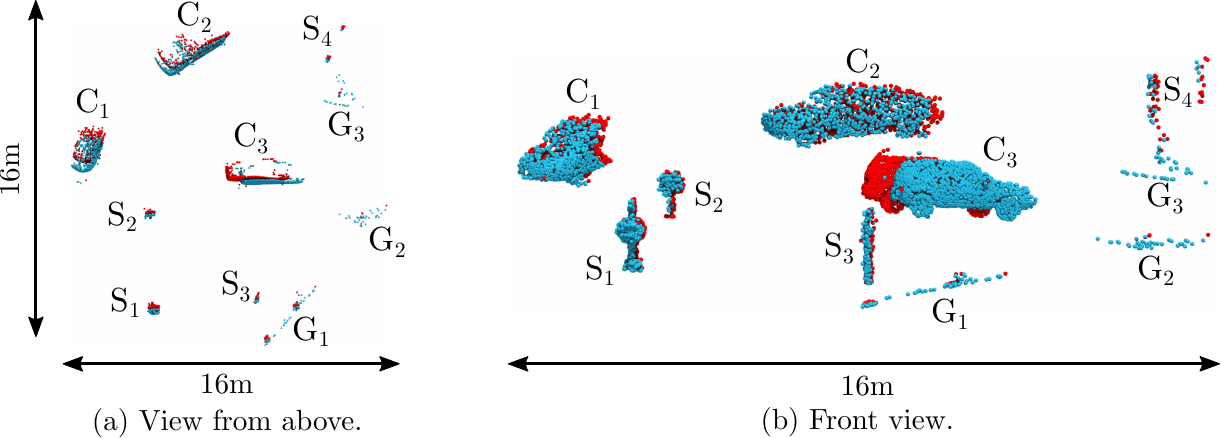}
	\caption{\textbf{3D scene extracted from the KITTI dataset \cite{kitti}.}
		The source (red) and target (blue) point clouds are sampled with
		4,000 points each: they are not in perfect correspondence with each other.
		In the coordinate system of the acquisition device,
		we observe 10 solid objects:
		3 cars in rigid motion ($\text{C}_1$, $\text{C}_2$, $\text{C}_3$);
		4 immobile traffic signs and poles ($\text{S}_1$, $\text{S}_2$, $\text{S}_3$, $\text{S}_4$);
		3 parts of the ground that have been correctly removed
		from the source frame in the pre-processing step but remain visible
		in the target frame ($\text{G}_1$, $\text{G}_2$, $\text{G}_3$).}
	\label{fig:kitti}

	\vspace*{.5cm}
	\vfill
	\includegraphics[width=\linewidth]{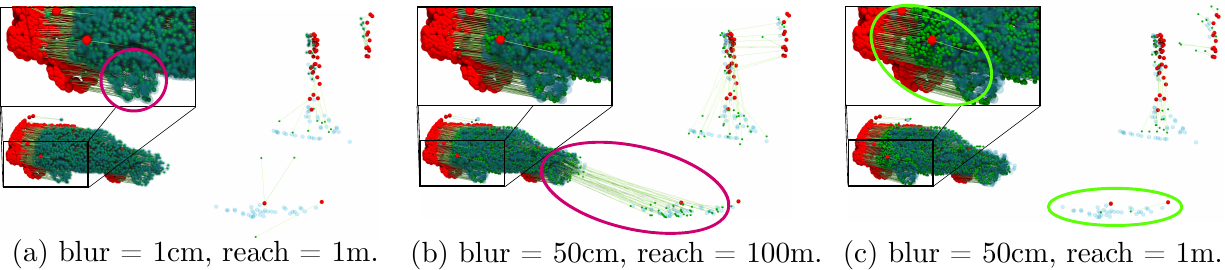}
	\caption{\textbf{Estimation of the 3D scene flow} with different values for the
		blur ($\sqrt{\epsilon}$) and reach ($\sqrt{\rho}$) parameters.
		We focus on a detail of Fig.~\ref{fig:kitti}
		($\text{C}_3$, $\text{S}_4$, $\text{G}_2$, $\text{G}_3$)
		and display the registration result $p_1$, \dots, $p_\text{N}$
		with green points and green arrows that link them to the source points $x_i$.
	}
	\label{fig:plan}

	\vspace*{.5cm}
	\vfill
	\includegraphics[width=\linewidth]{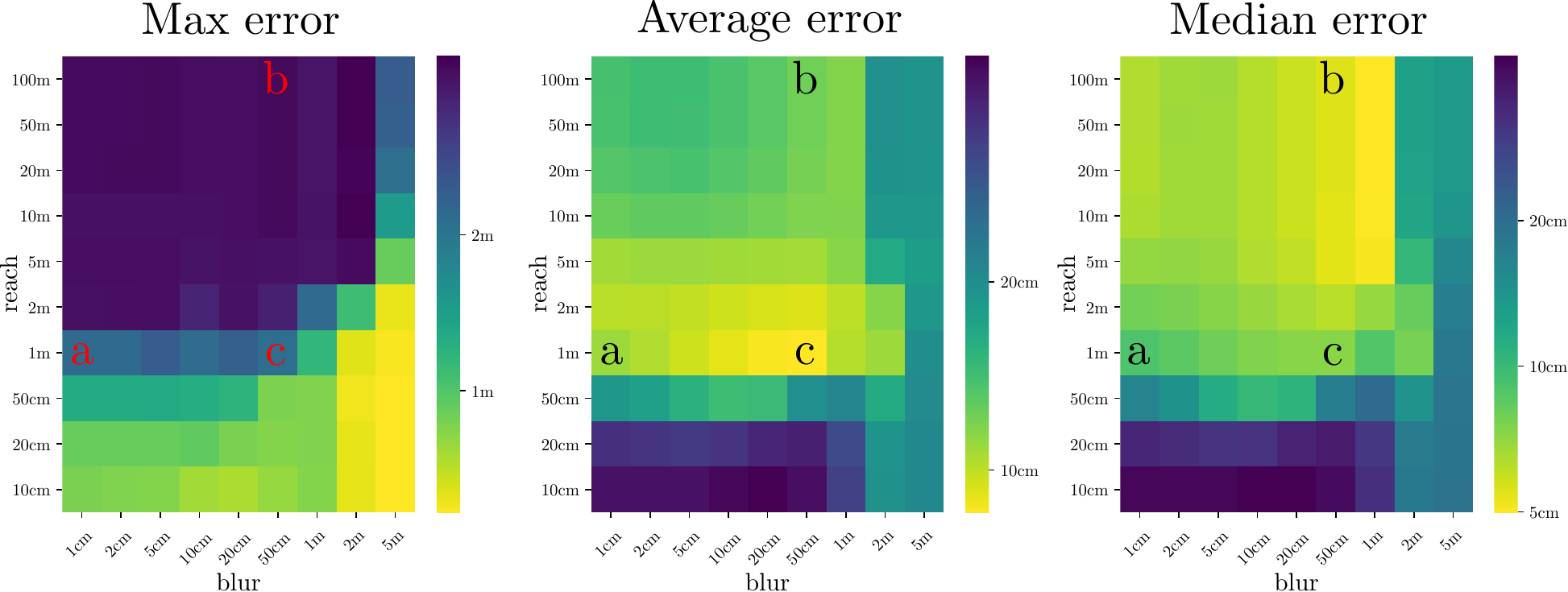}
	\caption{\textbf{Quantitative evaluation.} We display the maximum, average and median
		3D error between the final registrations $p_i$
		and the ground truth target points $p_i$ for varying values of
		the blur ($\sqrt{\varepsilon}$) and reach ($\sqrt{\rho}$) parameters
		in Eq.~(\ref{eq:scene_flow_descent}).
		Letters ``a'', ``b'' and ``c'' correspond to the three visualizations
		of Fig.~\ref{fig:plan}.}
	\label{fig:kitti_tables}
\end{figure}

In the context of automated driving,
a pre-processing step (the ``segmentation'')
removes points that correspond to
the pavement on the ground.
As a consequence, we can understand the scene flow between any two
frames as a collection of small, independent and rigid transformations
of solid objects such as cars, trees and bikes.
OT theory is perfectly suited to this class of geometric
deformations, and recent progress on numerical solvers
have opened the door
to real-time processing for this data \cite{robot}.

To demonstrate the influence of entropic regularization
and of the softening of the marginal constraints
on the scene flow estimation, we study a descent-based algorithm
along the lines of the previous Section.
We work with a quadratic cost $\C(x,y)=\tfrac{1}{2}\norm{x-y}^2_2$
and a Kullback--Leibler penalty on the marginal constraints.
We initialize our flowing point cloud on the source frame
($x_i^{(0)} = x_i$) with uniform weights equal to $1/\text{N}$
and update the point positions with:

\begin{align}
	\forall\, i \in \llbracket 1, \text{N} \rrbracket, ~~
	x_i^{(t+1)} = x_i^{(t)} - \text{N}\, \nabla_{x_i^{(t)}} \Sb\big(\tfrac{1}{\text{N}} \textstyle\sum_{k=1}^\text{N} \delta_{x_k^{(t)}}, \tfrac{1}{\text{M}}\textstyle\sum_{k=1}^\text{M} \delta_{y_k} \big)~.
	\label{eq:scene_flow_descent}
\end{align}
The final registration corresponds to the point cloud
$p_1=x_1^{(10)}$, \dots, $p_\text{N}=x_\text{N}^{(10)}$
after 10 iterations.
The main parameters of our method are the blur ($\sqrt {\epsilon}$)
and reach ($\sqrt{\rho}$) scales for the Sinkhorn divergence $\Sb$,
which are both homogeneous to distances in 3D space.

In this experiment, we rely on the GeomLoss library \cite{feydy2018interpolating,feydy2020thesis} to evaluate
the debiased Sinkhorn divergence $\Sb$ and its gradient.
We keep the GeomLoss ``scaling'' parameter equal to 0.9
to ensure a high precision in the OT solver.
As detailed in \cite[Section 3.3]{feydy2020thesis},
this implementation relies on
symmetrized iterations and an annealing
heuristic to speed up computations
beyond the fully rigorous Algorithm~\ref{alg:sinkhorn}
that is presented in this paper.
We display registration results in Fig.~\ref{fig:plan}
and make the following observations:

\begin{itemize}
	\item Unbalanced OT corresponds to the limit case where the reach
	      parameter ($\sqrt{\rho}$) is finite and the blur parameter ($\sqrt
		      {\epsilon}$) is smaller than the typical distance between any two
	      samples. This setting is illustrated in Fig.~\ref{fig:plan}.a: on the one
	      hand, the registration is robust to the presence of segmentation
	      artifacts for the pavement in the target frame; but on the other hand,
	      the final registration ($p_i$, green) overfits to the target point cloud ($y_j$, blue).
	      The estimated scene flow is unrealistically non-smooth.

	\item Entropy-regularized OT corresponds to the limit case where the blur
	      parameter ($\sqrt{\epsilon}$) is significantly larger than zero and the
	      reach parameter ($\sqrt{\rho}$) is equal to $+\infty$ or is much larger
	      than the diameter of the 3D scene. As illustrated in Fig.~\ref{fig:plan}.b, the registration is smooth but is highly impacted by artifacts that are present
	      in the data: our method matches the front-end of the car to a part of the
	      pavement that was (erroneously) left visible in the target frame.

	\item Unbalanced, entropy-regularized OT is robust to both types of perturbations. As
	      illustrated in Fig.~\ref{fig:plan}.c, picking intermediate values
	      for both of the blur ($\sqrt{\epsilon}$) and reach ($\sqrt{\rho}$) parameters allows us to recover a smooth displacement field
	      that is not thrown in disarray by segmentation artifacts.

\end{itemize}

\noindent
We provide a quantitative analysis of this experiment in Fig.~\ref{fig:kitti_tables} and note that:

\begin{itemize}
	\item The maximum error is primarily a function of the reach parameter: when $\sqrt{\rho}$ is too large, the model is highly sensitive to segmentation errors in the input data. The theory of unbalanced OT is thus required to make our model robust to \textbf{outliers}.

	\item For sensible values of the reach parameter ($\sqrt{\rho} \geqslant 1\,\text{m}$), the median error is a function of the entropic blur $\sqrt{\epsilon}$ that prevents overfitting to the target point cloud. The theory of  entropy-regularized OT is thus needed to make our model robust to \textbf{sampling noise}.

	\item The average error behaves as an intermediate statistic between the maximum and median errors -- which focus on outliers and inliers, respectively.
	      Overall, unbalanced and entropy-regularized OT produces optimal results when the blur parameter is equal to the typical size of the moving objects ($\sqrt{\epsilon} \simeq 50\,\text{cm}$ in our experiment) and the reach parameter is equal to the maximum plausible displacement for a point between any two frames ($\sqrt{\rho} \simeq 1\,\text{m}$ in our experiment).

\end{itemize}

These results show that unbalanced, entropy-regularized OT inherits from two
types of ``robustness’’ that are both relevant to the study of real-world
datasets.
Please note that we include this experiment as an
illustrative example:
in-depth discussions about run times, performance metrics and the interaction
of OT theory with state-of-the-art point neural networks are outside of the
scope of this theoretical paper. For a detailed presentation of the
applications of OT theory to point cloud
registration, we refer to the recent experimental paper \cite{robot}
and its bibliography.


\section*{Conclusion}

We presented in this article the Sinkhorn divergences for unbalanced optimal transport. 
We provided a theoretical analysis of both these divergences and the associated Sinkhorn's algorithm in the setting of continuous measures with compact support.
This shows how key properties from the balanced setting carry over to the unbalanced case. 
This however requires some non-trivial adaptations of both the definition of the divergences and the proof technics, in order to cope with a wide range of entropy functions.
The resulting unbalanced Sinkhorn divergences offer a versatile tool hybridizing OT and MMD distances which can readily be used in many applications in imaging sciences and machine learning.

\section*{Acknowledgments}

The work of Gabriel Peyr\'e is supported by the European Research Council (ERC project NORIA) and by the French government under management of Agence Nationale de la Recherche as part of the ``Investissements d’avenir'' program, reference ANR19-P3IA-0001 (PRAIRIE 3IA Institute).
The authors thank Th\'eo Lacombe for its feedbacks that considerably helped in writing this paper.

\appendix
\section{Additional proofs}
\label{appendix-proofs}

\subsection{Proof of proposition~\ref{prop-nonexp}}
Assume $\epsilon=1$.
Take two pairs $(p_1,q_1)$, $(p_2,q_2)$ such that for $i\in\{1,2\}$, $q_i = \aprox{\phi^*}(p_i)$. This is equivalent to $e^{p_i-q_i}\in\partial\phi^*(q_i)$, and because $\partial\phi^*$ is a monotone operator one has
\begin{align*}
  (e^{p_1-q_1} - e^{p_2-q_2})(q_1 - q_2) \geq 0.
\end{align*}
Then one can use the first order convexity condition to get
\begin{align*}
  &e^{p_1-q_1} - e^{p_2-q_2} \geq e^{p_2-q_2}(p_1 - q_1 - p_2 + q_2),\\
  &e^{p_2-q_2} - e^{p_1-q_1} \geq e^{p_1-q_1}(p_2 - q_2 - p_1 + q_1),\\
  &\Rightarrow 0\geq (e^{p_1-q_1} - e^{p_2-q_2})(p_2 - q_2 - p_1 + q_1)\\
  &\Rightarrow (e^{p_1-q_1} - e^{p_2-q_2})(p_1 - p_2) \geq (e^{p_1-q_1} - e^{p_2-q_2})(q_1 - q_2)\geq 0.
\end{align*}
The case $e^{p_1-q_1} = e^{p_2-q_2}$ is trivial, and without loss of generality we can assume $e^{p_1-q_1} - e^{p_2-q_2} > 0$ by swapping indices if necessary. Eventually it gives the pointwise inequality
\begin{align*}
  |p_1-p_2| \geq |q_1 - q_2| = |\aprox{\phi^*}(p_1) - \aprox{\phi^*}(p_2)|.
\end{align*}
The above inequality gives that if $x\mapsto p(x)$ is a continuous function instead of a real number, then $x\mapsto\aprox{\phi^*}(p(x))$ is also a continuous function (Take $p_1 = p(x)$, $p_2 = p(y)$ and let $x\rightarrow y$). Now take $q_1 = \aprox{\phi^*}(\f)$ and $q_2 = \aprox{\phi^*}(\g)$ for some $(\f,\g)\in\Cc(\Xx)$. Since $\Xx$ is compact, suprema are attained and we can take the point $x\in\Xx$ such that
\begin{align*}
  \norm{q_1-q_2}_\infty = |q_1(x) - q_2(x)|\leq |\f(x) - \g(x)|\leq \norm{\f - \g}_\infty.
\end{align*}
This proves the statement for $\epsilon=1$. One has for any $\epsilon >0$
\begin{align*}
  \aprox{\phi^*}(p) = \epsilon\text{\upshape{Aprox}}_{(\phi/\epsilon)^*}^{1}(p/\epsilon).
\end{align*}
This relation allows to conclude for any $\epsilon$.

\subsection{Weak* continuity of $\OTb$}

\begin{theorem}[Convexity and continuity of $\OTb$]
	For any entropy $\phi$, $\OTb$ is convex on $\Mmp(\Xx)$ in $\al$ and $\be$ but not jointly convex.
	For any entropy $\phi$ such that Theorem~\ref{thm-cv-sink-compact} holds, consider a sequence $\al_n\rightharpoonup\al$ and $\be_n\rightharpoonup\be$ with $(\al,\be)\in\Mmp(\Xx)$, and write $(\f_n,\g_n)$ a sequence of optimal potentials for $\OTb(\al_n,\be_n)$. If $(\f_n,\g_n)$ can be uniformly bounded by a constant independent of $n$, then $\OTb(\al_n,\be_n)\rightarrow\OTb(\al,\be)$.
\end{theorem}
\begin{proof}
	With respect to convexity, $\OTb$ is a supremum of functions which are linear in $\al$ and linear in $\be$, but not jointly convex in $(\al,\be)$: it is convex in $\al$ \emph{and} in $\be$.

	With respect to continuity, note that for any $n$, $(\f_n,\g_n)$ are $\gamma$-Lipschitz (Lemma~\ref{lem-smin-cost-regular}) and continuous on a compact set and are thus uniformly equicontinuous. Using the assumption that this sequence is uniformly bounded, the Ascoli-Arzelà Theorem allows us to show the relative compactness of the sequence in $\Cc(\Xx)\times\Cc(\Yy)$. Note that the Softmin and the aprox are $1$-Lipschitz (Lemma~\ref{lem-smin-lipschitz-func} and Proposition~\ref{prop-nonexp}) and the Softmin is weak* continuous in its input measure $\al$ or $\be$, thus for any converging subsequence $\f_{n_k}\rightarrow\f$ and $\g_{n_k}\rightarrow\g$, we get that $(\f,\g)$ is a fixed point of the Sinkhorn mapping for $(\al,\be)$ and is thus an optimal pair of potentials for $\OTb(\al,\be)$. Since the dual functional~\eqref{eq-dual-unb} is continuous in $(\al,\be,\f,\g)$, we get that for any subsequence $\OTb(\al_n,\be_n)\rightarrow\OTb(\al,\be)$, hence the continuity property.
\end{proof}

\begin{corollary}[Continuity of $\OTb$]\label{cor-continuity}
	Write $\al_n\rightharpoonup\al$ and $\be_n\rightharpoonup\be$ with $(\al,\be)\in\Mmpp(\Xx)$ such that for any $n$ there exists optimal dual potentials $(\f_n,\g_n)$. Then for any setting of Section~\ref{sec-exmp-f-div} we can uniformly bound this sequence and show that $\OTb$ is weak*-continuous.
\end{corollary}
\begin{proof}
	The case of strictly convex entropies is proved in Proposition~\ref{prop-uniform-conv}. In the balanced setting, potentials are defined up to a constant, thus we can assume without loss of generality that $\f_n(x^*)=0$ for some $x^*\in\Xx$. Because optimal potentials are $\gamma$-Lipschitz, we have that $\norm{\f_n}_\infty< \gamma\text{diam}(\Xx)$. Because the Sinkhorn update is 1-Lipschitz, we get that $\norm{\f_n}_\infty< 2\gamma\text{diam}(\Xx) + \epsilon|\log(m(\al_n))|$ and because $\al_n\rightharpoonup\al$ the mass term can be uniformly bounded, hence the result. When $\D_\phi=\rho\TV$ the aprox operator implies $\norm{\f_n}_\infty\leq\rho$ and $\norm{\g_n}_\infty\leq\rho$. In the case $\D_\phi=\RG_{[a,b]}$, we need to prove that for any $n$, at least one of the potentials $(\f_n,\g_n)$ is zero at some point of the support of $(\al_n,\be_n)$. If it is not the case, then we can replace $(\f_n,\g_n)$ by $(\f_n+ \lambda,\g_n-\lambda)$ with $\lambda\in\R$, and the expression of $\phi^*$ is such that the dual functional~\eqref{eq-dual-unb} is locally linear. Then we can locally increase the dual cost, which violates the optimality of $(\f_n,\g_n)$. Thus there exists $(x_n^*,y_n^*)$ such that $\f_n(x_n^*)=0$ or $\g_n(y_n^*)=0$, and we can derive a uniform bound similar to the balanced setting. Thus Theorem~\ref{thm-continuity-unb} holds and we get that $\OTb(\al_n,\be_n)\rightarrow\OTb(\al,\be)$.
\end{proof}

\subsection{Proof of Theorem~\ref{thm-diff-unb}}

The proof is mainly inspired from~\cite[Proposition 7.17]{santambrogio2015optimal}. 
Let us consider $\al$, $\delta\al$, $\be$, $\delta\be$
and $t$ in a neighborhood of $0$, as in Definition~\ref{def-diff-meas}.
We define the variation ratio $\tilde{\De}_t$
as $\tilde{\De}_t \eqdef \OTb(\al_t,\be_t) - \OTb(\al,\be) / t$.
%
%
%
we provide lower and upper bounds on $\tilde{\De}_t$
as $t$ goes to $0$. The purpose of the proof is to show that the $\lim\sup$ and $\lim\inf$ coincide, proving the derivative to be well-defined.

\paragraph{Lower bound.}

First, let us remark that $(\f,\g)$ is a
\emph{suboptimal} pair of dual potentials for
$\OTb(\al_t,\be_t)$.
Hence, one has
\begin{align*}
&\OTb(\al_t,\be_t) \geqslant \dotp{\al_t}{-\phi^*(-\f)} + \dotp{\be_t}{-\phi^*(-\g)} -\epsilon\dotp{ \al_t\otimes\be_t}{ \efgc{\f\oplus\g}-1} \nonumber\\
&\OTb(\al,\be)=\dotp{\al}{-\phi^*(-\f)}+\dotp{\be}{-\phi^*(-\g)} -\epsilon\dotp{ \al\otimes\be}{ \efgc{\f\oplus\g}-1}.\\
&\tilde{\De}_t \geqslant
\dotp{\de\al}{-\phi^*(-\f)}
+
\dotp{\de\be}{-\phi^*(-\g)}
- \epsilon 
\langle \delta\al\otimes\be +
\al\otimes\delta\be , \efgc{\f\oplus\g}-1
\rangle+o(1).
\end{align*}

\paragraph{Upper bound.}

Conversely, let us denote the optimal pair 
of potentials for $\OTb(\al_t,\be_t)$ by $(\f_t,\g_t)$.
As $(\f_t,\g_t)$ are suboptimal potentials for $\OTb(\al,\be)$,
we get that
\begin{align*}
&\OTb(\al,\be)
\geqslant
\dotp{\al}{-\phi^*(-\f_t)}
+
\dotp{\be}{-\phi^*(-\g_t)} 
-
\epsilon\langle \al\otimes\be
, \efgc{\f_t\oplus\g_t}-1\rangle
\nonumber\\
 &\OTb(\al_t,\be_t) = \dotp{\al_t}{-\phi^*(-\f_t)}
+
\dotp{\be_t}{-\phi^*(-\g_t)}
 -\epsilon \langle\al_t\otimes\be_t, \efgc{\f_t\oplus\g_t}-1\rangle,\\
&\tilde{\De}_t \leqslant
\langle\delta\al,-\phi^*(-\f_t)\rangle
+
\langle\delta\be,-\phi^*(-\g_t)\rangle
- \epsilon 
\langle \delta\al\otimes\be_t +
\al_t\otimes\delta\be 
, \efgc{\f_t\oplus\g_t} \text{-}1
\rangle+o(1) \nonumber
\end{align*}

\paragraph{Conclusion.}

Now, let us remark that as $t$ goes to $0$, 
$\al+t\delta\al \rightharpoonup \al$ and $\be+t\delta\be \rightharpoonup \be$.
Using Proposition~\ref{prop-uniform-conv}, $\f_t$ and $\g_t$ converge uniformly
towards $\f$ and $\g$.
Combining the lower and upper bound, we get
\begin{align*}
\tilde{\De}_t \xrightarrow{t\rightarrow0}
&\langle\delta\al,-\phi^*(-\f)-\epsilon\dotp{\be}{\efgc{\f\oplus\g}-1}\rangle +
\langle\delta\be,-\phi^*(-\g)-\epsilon\dotp{\al}{\efgc{\f\oplus\g}-1}\rangle.
\end{align*}
One has $\nabla\phi^*(-\f)=\dotp{\be}{\efgc{\f\oplus\g}}$ and $\nabla\phi^*(-\g)=$ $\dotp{\al}{\efgc{\f\oplus\g}}$ when $\phi^*$ is differentiable.
It yields the last result.

\subsection{Proof of Theorem~\ref{thm-Feps_uniqueness}}

Write
$\E_\epsilon(\al,\mu)~\eqdef~\langle\al,\phi^*(-\epsilon\log \tfrac{\d\mu}{\d\al})\rangle
+\tfrac{\epsilon}{2}
\langle \mu, 
k_\epsilon\star \mu \rangle$ for $(\al,\mu)\in\Mmp(\Xx)\times\Mmp(\Xx)$.
Since $\C$ is bounded on the compact set $\Xx\times\Xx$ and $\al$
is a positive measure, we have that $\Fb(\al)\leq \E_\epsilon(\al,\al) <+\infty$.

\hfill\break
\textbf{Strict convexity and lower-semicontinuity of $E_\epsilon$ and $\Fb$.}
We use Equation~\eqref{eq-constrained-entropy-program} from Proposition~\ref{prop-entropy-reform}. It allows to add the constraint set $I = \{(\al,\mu)\in\Mmp(\Xx),\, \al\sim\mu\}$ which is jointly convex in $(\al,\mu)$. The function $\psi = \phi^* \circ (-\epsilon\log)$ is convex because both functions are convex and $\phi^*$ is nondecreasing. On the set $I$, $\al$ verifies $\al^\bot=0$ w.r.t. $\mu$, thus the term $\dotp{\al}{\phi^*\big(- \epsilon \log\big( \frac{\d\mu}{\d\al}\big)\,\big)}$ can be identified as a $\psi$-divergence (except it is not nonnegative) and is thus jointly convex in $(\al,\mu)$. The norm $\norm{.}_{k_\epsilon}$ is jointly convex, thus so is $E_\epsilon$. Eventually, we minimize a jointly convex function over a (jointly) convex set, and we get that $\Fb$ is convex. Since $\psi$-divergences are also l.s.c. then $E_\epsilon$ is also l.s.c.

\hfill\break
\textbf{Coercivity on $\mu$ and existence.}
Since $\Xx\times\Xx$ is compact and $k_\epsilon(x,y)>0$,
there exists $\eta > 0$ such that
$k(x,y)>\eta$ for all $x$ and $y$ in $\Xx$.
We thus get $\norm{\mu}_{k_\epsilon}^2 ~\geqslant~ \dotp{\mu}{1}^2\,\eta$. For $\mu\in\Mmp(\Xx)$ write $\mu_\al$ its restriction to $\text{spt}(\al)$. One has $m(\mu_\al)\geq m(\mu)$ and because $\phi^*\circ(-\epsilon\log)$ is nonincreasing one has thanks to Jensen inequality that
\begin{align*}
\E_\epsilon(\al,\mu) ~\geqslant~&  m(\al). \phi^*(-\epsilon\log(\frac{m(\mu_\al)}{m(\al)})) +\eta m(\mu)^2\\
 ~\geqslant~&  m(\al). \phi^*(-\epsilon\log(\frac{m(\mu)}{m(\al)})) +\eta m(\mu)^2\\
~\geqslant~& - m(\al)\epsilon\log(\frac{m(\mu)}{m(\al)}) +\eta m(\mu)^2.
\end{align*}
Since $1 \in \text{dom}(\phi)$ one has $\phi^*(q) \geq q$. Thus whenever $m(\mu)$ goes to zero or infinity, $\E_\epsilon(\al,\mu)\rightarrow\infty$.
It allows to build a minimizing sequence $(\mu_n)$
for $\Fb(\al)$ such that
$m(\mu_n)$ is uniformly bounded by some constant
$M>0$.

The Banach-Alaoglu theorem holds and
asserts that $\{\, \mu\in\Mmp(\Xx) ~|~ m(\mu) \leqslant M \,\}$
is weakly compact;
we can extract a weakly converging subsequence
$\mu_{n_k}\rightharpoonup \mu_\infty$ from
the minimizing sequence $(\mu_n)$.
Since the map $\mu\mapsto\E_\epsilon(\al,\mu)$ is weakly
l.s.c.,
$\mu_\infty=\mu_\al$ realizes the minimum
of $\E_\epsilon$, proving the existence of minimizers.

\hfill\break
\textbf{Uniqueness.}
We assumed that the kernel $k_\epsilon$
is \emph{positive universal}.
The squared norm $\mu\mapsto\|\mu\|_{k_\epsilon}^2$
is thus a strictly convex functional, thus $\mu\mapsto \E_\epsilon(\al,\mu)$
is \emph{strictly} convex.
This ensures that $\mu_\al$ is uniquely defined.

\hfill\break
\textbf{Optimality of $\f$.}
If we consider the first order optimality in $\E_\epsilon$ we get $\al$-a.e. that
 $\tfrac{\d\mu_\al}{\d\al} k_\epsilon\star\mu_\al \in \partial\phi^*(-\epsilon\log \tfrac{\d\mu_\al}{\d\al}).$
Denoting $\f = \epsilon\log\tfrac{\d\mu_\al}{\d\al}$ this condition reads $e^{\f / \epsilon} \dotp{\al}{\efgc{\f}} \in \partial\phi^*(-\f)$.
Thus the potential $\f$ satisfies the optimality condition of the dual OT problem. The Radon-Nikodym-Lebesgue theorem only gives that $\f$ is $\al$-integrable, while we consider potentials in $\Cc(\Xx)$. Lemma~\ref{lem-smin-cost-regular} gives that $y \mapsto \Smin{\al}(\C(.,y) - \f)$ is continuous. The $\aprox{\phi^*}$ is Lipschitz thus continuous (Proposition~\ref{prop-nonexp}), so $\f=\Tt_\al(\f)$ is also continuous and optimal.

\hfill\break
\textbf{Strict convexity.}
We use Proposition~\ref{lem-injective-sym-pot} which gives that for two measures $\al\neq\be$ one has $\mu_\al\neq\mu_\be$ where both measure are the one attaining the optimal in $\Fb(\al)$ and $\Fb(\be)$. Since $k_\epsilon$ is universal the norm $\norm{.}_{k_\epsilon}$ is strictly convex when $\mu_\al\neq\mu_\be$. Write $\bar{\mu}$ the optimal measure for $t\al + (1-t)\be$ with $t\in(0,1)$. One has
\begin{align*}
\Fb(t\al + (1-t)\be) &= E_\epsilon(t\al + (1-t)\be, \bar{\mu})
\leq E_\epsilon(t\al + (1-t)\be, t\mu_\al + (1-t)\mu_\be)\\
&< t E_\epsilon(\al,\mu_\al) + (1-t) E_\epsilon(\be,\mu_\be)
< t \Fb(\al) + (1-t) \Fb(\be).
\end{align*}
Hence the strict convexity of the Sinkhorn entropy.

\hfill\break
\textbf{Continuity.} 
The continuity is given by Theorem~\ref{thm-continuity-unb} and Corollary~\ref{cor-continuity} in the particular case $\al=\be$.


\subsection{Proof of Proposition~\ref{prop-Seps-ineq-norm}}

In the latter development we identify symetric terms with a kernel norm through 
$\Vert \be e^{\frac{\g_{\be}}{\epsilon}}\Vert^2_{k_\epsilon} = \dotp{\be\otimes\be}{e^{\frac{(\g_{\be}\oplus \g_{\be} - \C )}{\epsilon}}}.$

Under our assumptions, we know from Theorem~\ref{thm-Feps_uniqueness} that $\f_\al$ and $\g_\be$ exist and are unique. The idea of the proof is to say that the pair of potentials $(\f_{\al}, \g_{\be})$ is suboptimal for $\OTb(\al,\be)$. Since its definition is a supremum over $\Cc(\Xx)$ we get a lower bound that gives
\begin{align*}
 & \OTb(\al,\be)  \geq 
  		- \dotp{\al}{\phi^*(-\f_\al )} -  \dotp{\be}{\phi^*(-\g_\be)}
		-\epsilon \dotp{\al\otimes\be}{e^{\frac{(\f_\al\oplus \g_\be - \C )}{\epsilon}} - 1} \\
	&\qquad \geq - \dotp{\al}{\phi^*(-\f_\al )} - \dotp{\be}{\phi^*(-\g_\be )}
		-\epsilon \dotp{\al\otimes\be}{e^{\frac{(\f_\al\oplus \g_\be - \C )}{\epsilon}}} 
		+ \epsilon m(\al)m(\be)\nonumber\\[0.4em]
& \qquad \geq - \dotp{\al}{\phi^*(-\f_\al)} - \dotp{\be}{\phi^*(-\g_\be)} -\epsilon \dotp{\al\otimes\be}{e^{\frac{(\f_\al\oplus \g_\be - \C )}{\epsilon}}} \\
&\quad\qquad-\tfrac{\epsilon}{2}( \Vert \al e^{\frac{\f_\al}{\epsilon}}\Vert^2_{k_\epsilon} - m(\al)^2 + \Vert \be e^{\frac{\g_\be}{\epsilon}}\Vert^2_{k_\epsilon} - m(\be)^2)\\
&\quad \qquad+\tfrac{\epsilon}{2}( \Vert \al e^{\frac{\f_\al}{\epsilon}}\Vert^2_{k_\epsilon} - m(\al)^2 + \Vert \be e^{\frac{\g_\be}{\epsilon}}\Vert^2_{k_\epsilon} - m(\be)^2) + \epsilon m(\al)m(\be)\\ 
& \qquad \geq - \dotp{\al}{\phi^*(-\f_\al)} - \tfrac{\epsilon}{2} (\Vert \al e^{\frac{\f_\al}{\epsilon}}\Vert^2_{k_\epsilon} - m(\al)^2) 
	- \dotp{\be}{\phi^*(-\g_\be)} - \tfrac{\epsilon}{2}(\Vert \be e^{\frac{\g_\be}{\epsilon}}\Vert^2_{k_\epsilon} - m(\be)^2)\nonumber \\
&\quad \qquad +\tfrac{\epsilon}{2}\big( \Vert \al e^{\frac{\f_\al}{\epsilon}}\Vert^2_{k_\epsilon} + \Vert \be e^{\frac{\g_\be}{\epsilon}}\Vert^2_{k_\epsilon}-2 \dotp{\al\otimes\be}{e^{\frac{(\f_\al\oplus \g_\be - \C )}{\epsilon}}}\big)\\
&\quad \qquad	 + \epsilon m(\al)m(\be) - \tfrac{\epsilon}{2}(m(\al)^2 + m(\be)^2) \\
& \qquad  \geq \tfrac{1}{2}\OTb(\al,\al)+  \tfrac{1}{2}\OTb(\be,\be) 
	+ \tfrac{\epsilon}{2} \Vert \al e^{\frac{\f_\al}{\epsilon}} 
	- \be e^{\frac{\g_\be}{\epsilon}}\Vert^2_{k_\epsilon} 
	- \tfrac{\epsilon}{2} (m(\al) - m(\be))^2.
\end{align*}
With the last line we deduce the desired bound from the definition of $\Sb$.
\section{Sample Complexity - Proof of Theorem~\ref{thm-sample-complexity-unb}}
\label{appendix-statistical-complexity}
\subsection{Prerequisites}

We present in this section the material which is necessary to follow the details of the proof.
We first define Sobolev spaces, and then detail the Faà Di Bruno which is extensively applied in the proofs, with the main result on sample complexity in RKHS.

\begin{definition}
The Sobolev space $\Hh^s_\al(\Xx)$, for $s\in\N^*$, is the space of functions $\f:\Xx\rightarrow \R$ such that for every multi-index $k$ with $|k|\leq s$, the mixed partial derivative $\f^{(k)}$ exists and belongs to $\mathbb{L}^2_\al(\Xx)$. It is endowed with the inner-product
\begin{align*}
  \dotp{\f}{\g}_{\Hh^s_\al(\Xx)} \eqdef \sum_{|k|\leq s} \int_\Xx \f^{(k)}(x)\g^{(k)}(x) \d\al.
\end{align*}

We also define the Sobolev ball $\Hh^s_{\al,\lambda}(\Xx) = \{ \f\in\Hh^s_\al(\Xx), \, \norms{\f}_{\Hh^s_\al(\Xx)} \leq \lambda\}$.
\end{definition}

We recall that for $s > \lfloor \tfrac{d}{2} \rfloor$, $\Hh^s(\R^d)$ is a RKHS. Furthermore the Sobolev extension theorem~\cite{calderon1961lebesgue} gives that $\norms{.}_{\Hh^s_\al(\R^d)} \leq C \norms{.}_{\Hh^s_\al(\Xx)}$ provided that $\Xx$ is a bounded Lipschitz domain. Thus, in what follows it suffices to control the dual potentials with respect to the norm $\norms{.}_{\Hh^s_\al(\Xx)}$ over the compact $\Xx$.

We now state the PAC-learning result we apply in $\Hh^s(\R^d)$. It is a combination of the proofs of Theorem (8), (12.4) and Lemma 22 in~\cite{bartlett2002rademacher}.
\begin{proposition}\label{prop-pac-rkhs}
\cite{bartlett2002rademacher} 
Consider $\al\in\Mmpo(\Xx)$, a $B$-Lipschitz loss $L$ and $\Gg$ a given class of functions. Then
\begin{align*}
  \mathbb{E}_\al \big[ \sup_{\f\in\Gg} \mathbb{E}_\al L(\f) - \mathbb{E}_{\al_n} L(\f) \big] \leq 2B \mathbb{E}_\al\Rr(\Gg)
\end{align*}
where $\mathbb{E}_\al[\f]=\dotp{\al}{\f}$ and $\Rr(\Gg)$ denotes the Rademacher complexity of the class of functions $\Gg$. When $\Gg$ is a ball of radius $\lambda$ in a RKHS with kernel $k$ the Rademacher complexity is bounded by
\begin{align*}
  \mathbb{E}_\al\Rr(\Gg) \leq \frac{\lambda}{n}\sqrt{\sum_{i=1}^n k(X_i,X_i)} \leq \frac{\lambda}{\sqrt{n}} \sqrt{\max_{x\in\Xx} k(x,x)}.
\end{align*}
\end{proposition}

The loss defined in this property will be the identity which is 1-Lipschitz, while the function used in~\cite{genevay2018sample} had a Lipschitz constant depending exponentially in $\epsilon$. 

In order to prove that the potentials are in such RKHS, we need to explicit the derivatives of the potentials through a differentiation of the Sinkhorn mapping. Since it is a composition of several functions, we need to use the Faà Di Bruno formula. It has been generalised for the composition of multivariate functions in~\cite{constantine1996multivariate}. We detail a corollary of the general formula because we only need a composition of function where only the first one is multivariate.

\begin{proposition}\label{prop-faadibruno}
\cite[Corollary 2.10]{constantine1996multivariate}
Define the functions $\g:\R^d \rightarrow \R$, $\f:\R\rightarrow\R$ and $h=\f\circ\g$. Take $x\in\R^d$, $y=\g(x)\in\R$ and $n=|\nu|\in\mathbb{N}\setminus\{0\}$ where $\nu$ is a multi-index. Assume $\g$ is $\Cc^\nu$ at $x$ and $\f$ is $\Cc^n$ at $y$. Then
\begin{align*}
  h^{(\nu)}(x) = \sum_{\lambda=1}^n \f^{(\lambda)}(y) \sum_{p(\nu,\lambda)} (\nu!) \prod_{j=1}^n \frac{\big(g^{(l_j)}(x)\big)^{k_j}}{(k_j!)(l_j!)^{k_j}},
\end{align*}
where
\begin{align*}
  p(\nu,\lambda) = \{& (k_1,...,k_n)\in (\N)^n, \, (l_1,...,l_n)\in (\N^d)^n,\, \exists s\in\llbracket 1,n\rrbracket, \forall i \in\llbracket s,n\rrbracket,\,\\ &k_i>0 \text{ and } 0\prec l_s \prec ... \prec l_n \, \text{ such that } 
  \sum_{i=1}^n k_i = \lambda,\, \sum_{i=1}^n k_i l_i = \nu\}.
\end{align*}

The $0$-th derivative is the function itself. The factorial of a vector is the product of the factorial of the coordinates. One has $l\prec \tilde{l}$ when either $|l| < |\tilde{l}|$ or when $|l| = |\tilde{l}|$ it is larger w.r.t. the lexicographic order. In the monovariate setting, we necessarily have $l_s = s$.
\end{proposition}

\subsection{Proof of the sample complexity}

Terms of the form $\phi^{*(k)}(-\f)$ will appear in the derivation of the bounds. 
Since we are looking at the dependence in $\epsilon$, and because the optimal potential $\f$ implicitly depends on it, we need this first lemma which asserts that its norm is uniformly bounded independently of $\epsilon$.
Knowing that the dual potentials do not diverge with respect to $\epsilon$ allows to consider a compact $\Yy$ in which $\norms{\phi^{*(k)}}_\infty$ is finite.
In what follows, the norms $\norms{\C^{(k)}}_\infty$ and $\norms{\phi^{*(k)}}_\infty$ are meant to be estimated on $\Xx$ and $\Yy$ respectively. Since $\C$ and $\phi^*$ are $\Cc^\infty$, those norms are all finite.

\begin{proposition}\label{prop-unif-bound-pot-eps}
Take any pair of measures $(\al,\be)\in\Mmpp(\Xx)$. Under Assumption~(\ref{as:1},\ref{as:2}), the potentials are uniformly bounded by a bound which is independent of $\epsilon$.
\end{proposition}
\begin{proof}
Lemma~\ref{lem-uniq-tensor-sum} holds and asserts that the dual functional is strictly convex and coercive in $\f\oplus\g$. Though, coercivity when $\f\oplus\g \rightarrow + \infty$ seems to depend on $\epsilon$ because of the term $\epsilon(\tefgc - 1)$. Since $\epsilon(e^{x / \epsilon}-1) \geq x$ for any $x$, coercivity is guaranteed independently of $\epsilon$, and one gets that $\norms{\f\oplus\g}_\infty$ is uniformly bounded independently of $\epsilon$.
It remains to prove the same property for $\f$ and $\g$. Any optimal potential $\f$ is $\gamma$-Lipschitz. thus if one writes $\f = \lambda + h$ with $h(x_0) = 0$, $h$ is also Lipschitz, thus $\norms{h}_\infty \leq \gamma\text{diam}(\Xx)$. It remains to prove that $\lambda$ can be uniformly bounded independently of $\epsilon$.
The proof of Lemma~\ref{lem-compact-dual} shows that under Assumptions~\ref{as:2}, the dual functional is coercive under translations, due to the terms involving $\phi^*$ which do not depend on $\epsilon$.
Thus coercivity holds independently of $\epsilon$. We have $\norms{\f}_\infty \leq |\lambda| + \gamma\text{diam}(\Xx)$ where $\lambda$ is in a compact set independent of $\epsilon$.
\end{proof}

Before stating the result on the regularity of the dual potentials, we prove a technical proposition that explicits the expression of derivatives of the $\aprox{\phi^*}$ operator.
We introduce a generic notation by expressing some terms implicitly as polynomials of the parameter $\epsilon$ of order $k$, written $P_{k}(\epsilon)$. In some calculations the same notation $P$ is used to represent different objects from one line to another.

\begin{proposition}\label{prop-deriv-aprox}
Assume that $\phi^*$ is $\Cc^\infty$. Then the operator $\aprox{\phi^*}$ is also $\Cc^\infty$, and its n-th derivative verifies for any $n$
\begin{align*}
  (\aprox{\phi^*})^{(n)}(x) = \frac{P_{n-1}(\epsilon)(x)}{(\phi^{*\prime} + \epsilon\phi^{*\prime\prime})^{2n-1}}
\end{align*}
where $P_{n-1}(\epsilon)$ represents a polynomial in $\epsilon$ of order $n-1$ whose coefficients are functions which only depend on the derivatives of $\phi^*$ up to the order $n$. The dependance of $P_{n-1}(\epsilon)$ in $x$ only appears through the derivatives of $\phi^*$.
\end{proposition}
\begin{proof}
For sake of conciseness we will write $p(x) = \aprox{\phi^*}(x)$ in this proof.
The regularity of $\aprox{\phi^*}$ is given by the optimality condition of its definition, i.e. $\phi^{*\prime} (p(x)) = e^{(x - p(x)) / \epsilon}$.
This expression is a $\Cc^\infty$ function in $(x,p(x))$ whose derivatives are never nonzero, thus the implicit function theorem gives that $\aprox{\phi^*}$ is $\Cc^\infty$.

We prove the bound on the derivatives of $\aprox{\phi^*}$ by a strong induction. Differentiating this equation yields
\begin{align}
  p^\prime\phi^{*\prime\prime}\circ &p = \frac{1 - \phi^{*\prime}\circ p}{\epsilon} e^{\frac{x - p(x)}{\epsilon}}
  = \frac{1 - p^{\prime}}{\epsilon} \phi^{*\prime}\circ p\nonumber\\
&p^\prime(\phi^{*\prime}\circ p + \epsilon\phi^{*\prime\prime}\circ p) = \phi^{*\prime}\circ p. \label{eq-diff-optim-prox} 
\end{align}
This relation proves the statement for $n=1$. Let's assume now that the property is true up to a given integer $n$. Applying the Faà Di Bruno and Leibniz formulas~\ref{prop-faadibruno} to the above equation~\eqref{eq-diff-optim-prox} gives
\begin{align}
&p^{(n+1)} (\phi^{*\prime}\circ p + \epsilon\phi^{*\prime\prime}\circ p) = \label{eq-faaleibniz-line0}\\
  &\quad\sum_{\lambda=1}^n \phi^{*(\lambda+1)}\circ p \sum_{p(\nu,\lambda)} (\nu!) \prod_{j=1}^n \frac{\big(p^{(l_j)}\big)^{k_j}}{(k_j!)(l_j!)^{k_j}} \label{eq-faaleibniz-line1}\\
  &-\sum_{k=0}^{n-1}\binom{n}{k} p^{(k+1)}\sum_{\lambda=1}^{n-k} (\phi^{*(\lambda + 1)}\circ p + \epsilon\phi^{*(\lambda + 2)}\circ p) \sum_{p(\nu,\lambda)} (\nu!) \prod_{j=1}^n \frac{\big(p^{(l_j)}\big)^{k_j}}{(k_j!)(l_j!)^{k_j}}.\label{eq-faaleibniz-line2}
\end{align}

Note that in the above formula the last derivative~\eqref{eq-faaleibniz-line0} in the leibniz formula has been separated from the rest of the sum~\eqref{eq-faaleibniz-line2}. Applying the induction hypothesis, one gets that for any $\lambda$ line~\ref{eq-faaleibniz-line1} is a polynomial of order
\begin{align*}
  \prod_{j=1}^n \big(p^{(l_j)}\big)^{k_j} 
  &= \prod_{j=1}^n \frac{P_{k_j(l_j -1)}(\epsilon)}{(\phi^{*\prime} + \epsilon\phi^{*\prime\prime})^{k_j (2l_j - 1)}} \\
  &= \frac{P_{n - \lambda}(\epsilon)}{(\phi^{*\prime} + \epsilon\phi^{*\prime\prime})^{2n - \lambda}}  \times \frac{(\phi^{*\prime} + \epsilon\phi^{*\prime\prime})^{\lambda}}{(\phi^{*\prime} + \epsilon\phi^{*\prime\prime})^{\lambda}}
  = \frac{P_{n}(\epsilon)}{(\phi^{*\prime} + \epsilon\phi^{*\prime\prime})^{2n}}.
\end{align*}
As the same term appears line~\ref{eq-faaleibniz-line2} with a Faà Di Bruno formula that stops at the order $n-k$, one gets for any $(k,\lambda)$ a term of order
\begin{align*}
  \frac{P_{k}(\epsilon)}{(\phi^{*\prime} + \epsilon\phi^{*\prime\prime})^{2k+1}} &\times P_{1}(\epsilon) \times \frac{P_{n-k-\lambda}(\epsilon)}{(\phi^{*\prime} + \epsilon\phi^{*\prime\prime})^{2(n-k) - \lambda}}\\
&= \frac{P_{n+1 - \lambda}(\epsilon)}{(\phi^{*\prime} + \epsilon\phi^{*\prime\prime})^{2n - \lambda + 1}} \times \frac{(\phi^{*\prime} + \epsilon\phi^{*\prime\prime})^{\lambda - 1}}{(\phi^{*\prime} + \epsilon\phi^{*\prime\prime})^{\lambda - 1}} 
= \frac{P_n(\epsilon)}{(\phi^{*\prime} + \epsilon\phi^{*\prime\prime})^{2n}}.
\end{align*}
Eventually, dividing $p^{(n+1)} (\phi^{*\prime} + \epsilon\phi^{*\prime})$ by $(\phi^{*\prime} + \epsilon\phi^{*\prime\prime})$ gives the right denominator and ends the proof by strong induction.
\end{proof}

\begin{proposition} \label{prop-dual-pot-sob}
Assume that $\phi^*$ and  $\C$ are $\Cc^\infty$ and that Assumptions~(\ref{as:1}, \ref{as:2}) hold.
One has $\phi^*(-\f) + \epsilon\nabla\phi^*(-\f)\in\Hh^s_{\al,\lambda}(\R^d)$ and $\phi^*(-\g) + \epsilon\nabla\phi^*(-\g)\in\Hh^s_{\be,\lambda}(\R^d)$,
where the radius of the ball $\lambda$ is a rational fraction of $\epsilon$ with coefficients depending on the norms $\norm{\phi^{*(k)} }$ and $\norms{\C^{(k)} }$ for derivatives $k$ up to the order $s$, but is independent of the measures' masses. Its asymptotics for $\epsilon$ going to either $0$ or $+\infty$ read
\begin{align*}
  For \,\,\epsilon\rightarrow 0, \quad \lambda = O(1/ \epsilon^{s-1}), \qandq
  for \,\,\epsilon\rightarrow \infty, \quad \lambda = O(1).
\end{align*}
\end{proposition}
\begin{proof}
This proof applies several times the Faà Di Bruno formula~\ref{prop-faadibruno} to the function
\begin{align*}
  x\mapsto (\phi^* + \epsilon\nabla\phi^*)\circ\aprox{\phi^*}\circ(\epsilon\log)\circ\dotp{\be}{e^{\frac{\g-\C(x,.)}{\epsilon}}}.
\end{align*}

\hfill\break
\textbf{Differentiation under the integral.}
We differentiate the operator $x\mapsto\dotp{\be}{e^{(\g-\C(x,.)) / \epsilon}}$. An application of the dominated convergence theorem similar to Lemma~\ref{lem-smin-cost-regular} proves that it is as smooth as the cost $\C$ and that the differentiation and integration can be swapped. In other words
\begin{align*}
  \partial^{(k)}\dotp{\be}{e^{(\g-\C(x,.)) / \epsilon}} = \dotp{\be}{\partial^{(k)} e^{(\g-\C(x,.)) / \epsilon}}.
\end{align*}
Applying Proposition~\ref{prop-faadibruno} to $h = \exp\circ(-\C / \epsilon)$ defined on $\R^d\rightarrow\R\rightarrow\R$ gives
\begin{align*}
  h^{(\nu)}(x) &= \sum_{\lambda=1}^n e^{-\C(x,.) / \epsilon} \sum_{p(\nu,\lambda)} (\nu!) \prod_{j=1}^n \frac{\big(-\C^{(l_j)}(x,.) / \epsilon\big)^{k_j}}{(k_j!)(l_j!)^{k_j}}\\
  &= e^{-\C(x,.) / \epsilon} \sum_{\lambda=1}^n (\tfrac{1}{\epsilon})^\lambda \sum_{p(\nu,\lambda)} (\nu!) \prod_{j=1}^n \frac{\big(-\C^{(l_j)}(x,.)\big)^{k_j}}{(k_j!)(l_j!)^{k_j}}\\
  &\leq e^{-\C(x,.) / \epsilon} \sum_{\lambda=1}^n (\tfrac{1}{\epsilon})^\lambda \sum_{p(\nu,\lambda)} (\nu!) \prod_{j=1}^n \frac{\big(\norms{\C^{(l_j)}}_\infty\big)^{k_j}}{(k_j!)(l_j!)^{k_j}}
\end{align*}
Note that the norm $\norms{.}_\infty$ verifies $\norms{\f\g}_\infty\leq\norms{\f}_\infty\norms{\g}_\infty$.

Thus one can bound the derivative of the integral
\begin{align}\label{eq-bound-faa-int}
  \partial^{(k)}\dotp{\be}{e^{(\g-\C(x,.)) / \epsilon}} \leq \dotp{\be}{e^{(\g-\C(x,.)) / \epsilon}} Q_k(\tfrac{1}{\epsilon}),
\end{align}
where $Q_k$ is a polynomial in $1/\epsilon$ of order $k$ with no constant term (it is important when $\epsilon\rightarrow\infty$), whose coefficients only depend on the norm of the derivatives of $\C$.

\hfill\break
\textbf{Differentiation of the Sinkhorn mapping.}
We now differentiate the composition of $T(x) = -\aprox{\phi^*}(\epsilon\log(x))$ for any smooth $\aprox{\phi^*}$ operator. Given that $\log^{(\nu)}(x) = (-1)^\nu (\nu - 1)! x^{-\nu}$, the Faà Di Bruno formula~\ref{prop-faadibruno} with Proposition~\ref{prop-deriv-aprox} formula gives
\begin{align}
  T^{(\nu)}(x) &= -\sum_{\lambda=1}^n (\aprox{\phi^*})^{(\lambda)}(\epsilon\log(x)) \sum_{p(\nu,\lambda)} (\nu!) \prod_{j=1}^n \frac{\big( \epsilon(-1)^j (j - 1)! x^{-j} \big)^{k_j}}{(k_j!)(j!)^{k_j}} \label{eq-faa-sink-1} \\ 
  &= -(-1)^\nu x^{-\nu} \sum_{\lambda=1}^n (\epsilon)^\lambda(\aprox{\phi^*})^{(\lambda)}(\epsilon\log(x)) \sum_{p(\nu,\lambda)} (\nu!)  \prod_{j=1}^n \frac{1}{(k_j!)(j)^{k_j}}\label{eq-faa-sink-2}\\
  &= -(-1)^\nu x^{-\nu} \sum_{\lambda=1}^n \frac{(\epsilon)^\lambda P_{\lambda-1}(\epsilon)(x)}{(\phi^{*\prime} + \epsilon\phi^{*\prime\prime})^{2\lambda - 1}}\label{eq-faa-sink-3}\\
  &\leq x^{-\nu} \sum_{\lambda=1}^n \frac{(\epsilon)^\lambda P_{\lambda-1}(\epsilon)}{(\inf\phi^{*\prime} + \epsilon\inf\phi^{*\prime\prime})^{2\lambda - 1}}.\label{eq-bound-faa-sink}
\end{align}
We recall that the Faà Di Bruno formula imposes $\sum k_j = \lambda$ and $\sum j k_j=\nu$, hence the simplification from line~\eqref{eq-faa-sink-1} to line~\eqref{eq-faa-sink-2}. Line~\eqref{eq-faa-sink-3} is an application of Proposition~\ref{prop-deriv-aprox} which simplifies the expression. Eventually we can bound this term as displayed line~\eqref{eq-bound-faa-sink}. In this last line the polynomial is meant to depend on the norms $\norms{\phi^{*(k)}}$ and $\epsilon$ but not on $x$ (since $x$ appears through the derivatives of $\phi^*$).

\hfill\break
\textbf{Differentiation of the dual potential.}
Applying the Faà di Bruno formula~\ref{prop-faadibruno} to $\f = T(\dotp{\be}{e^{(\g-\C) / \epsilon}})$ yields
\begin{align}
  \f^{(\nu)}(x) &= \sum_{\lambda=1}^n T^{(\lambda)}(\dotp{\be}{e^{(\g-\C) / \epsilon}}) \sum_{p(\nu,\lambda)} (\nu!) \prod_{j=1}^n \frac{\big(\partial^{(l_j)}\dotp{\be}{e^{(\g-\C(x,.)) / \epsilon}} \big)^{k_j}}{(k_j!)(j!)^{k_j}}\label{eq-faa-pot-1}\\
  &\leq \sum_{\lambda=1}^n \bigg(\dotp{\be}{e^{(\g-\C) / \epsilon}}^{-\lambda} \sum_{k=1}^n \frac{(\epsilon)^\lambda P_{k-1}(\epsilon)}{(\inf\phi^{*\prime} + \epsilon\inf\phi^{*\prime\prime})^{2k - 1}}\bigg) \times\nonumber\\
  &\qquad\qquad\qquad\bigg( \dotp{\be}{e^{(\g-\C(x,.)) / \epsilon}}^{\sum k_j} Q_{\sum j k_j}(\tfrac{1}{\epsilon})\bigg)\label{eq-faa-pot-2}\\
  &\leq \sum_{\lambda=1}^n \dotp{\be}{e^{(\g-\C) / \epsilon}}^{-\lambda} \sum_{k=1}^\lambda \frac{(\epsilon)^k P_{k-1}(\epsilon)\dotp{\be}{e^{(\g-\C(x,.)) / \epsilon}}^{\lambda} Q_{\nu}(\tfrac{1}{\epsilon})}{(\inf\phi^{*\prime} + \epsilon\inf\phi^{*\prime\prime})^{2k - 1}}   \label{eq-faa-pot-3}\\
  &\leq \sum_{\lambda=1}^n \sum_{k=1}^\lambda \frac{(\epsilon)^k P_{k-1}(\epsilon)}{(\inf\phi^{*\prime} + \epsilon\inf\phi^{*\prime\prime})^{2k - 1}} Q_{\nu}(\tfrac{1}{\epsilon}).\label{eq-bound-faa-pot}
\end{align}

Line~\eqref{eq-faa-pot-2} combines Inequalities~\eqref{eq-bound-faa-int} and~\eqref{eq-bound-faa-sink}.
The notation $P_\nu(\epsilon)$ represents a polynomial of order $\nu$ in $\epsilon$ whose coefficients depend on the norms $\norms{\phi^{*(k)} }_\infty$, and $Q_\nu(1 / \epsilon)$ represents a polynomial of order $\nu$ in $ 1 / \epsilon$ with no constant term and whose coefficients depend on the norms $\norms{\C^{(k)} }_\infty$.
Note that under Assumption~\ref{as:1}, $\phi^*$ is increasing and strictly convex on the compact $\Yy$, thus $\inf \phi^{*\prime} >0$ and $\inf \phi^{*\prime\prime} >0$. 
An important fact is that all terms $\dotp{\be}{e^{(\g-\C) / \epsilon}}$ disappear in the bound~\eqref{eq-bound-faa-pot}. Since all other contributions of this form disappear by bounding with $\norms{.}_\infty$, it means that $\norms{\f}_\infty$ is bounded independently of the mass of the input measure $\be$.

Thus the norm of the dual potential is bounded by 
\begin{align*}
  \norm{\f^{(\nu)}}_\infty &\leq \sum_{\lambda=1}^n \sum_{k=1}^\lambda \frac{\epsilon^k P_{k-1}(\epsilon)}{(\inf\phi^{*\prime} + \epsilon\inf\phi^{*\prime\prime})^{2k - 1}} Q_{\nu}(1/\epsilon).
\end{align*}
Again, note that the bound on the norm of $\f^{(\nu)}$ does not depend on $\dotp{\be}{e^{(\g-\C) / \epsilon}}$, thus it does not depend on the mass of the input measures $(\al,\be)$.

\hfill\break
\textbf{Proof of asymptotics.} 
Eventually we apply one last time the Faà Di Bruno formula~\ref{prop-faadibruno} to $h = \phi^*(-\f) + \epsilon\nabla\phi^*(-\f)$ with Inequality~\eqref{eq-bound-faa-pot} to get
\begin{align*}
  h^{(\nu)} &= \sum_{\lambda=1}^n (\phi^{*(\lambda)}(-\f) + \epsilon\phi^{*(\lambda+1)}(-\f)) \sum_{p(\nu,\lambda)} (\nu!) \prod_{j=1}^n \frac{(-\f^{(l_j)})^{k_j}}{(k_j!)(l_j!)^{k_j}},\\
  \norm{ h^{(\nu)} }_\infty &\leq \sum_{\lambda=1}^n \big(\norm{ \phi^{*(\lambda)} }_\infty + \epsilon\norm{ \phi^{*(\lambda+1)} }_\infty \big) \sum_{p(\nu,\lambda)} (\nu!) \prod_{j=1}^n \frac{\norm{ \f^{(l_j)} }_\infty^{k_j}}{(k_j!)(l_j!)^{k_j}}.
\end{align*}
Let's focus on the case $\epsilon\rightarrow 0$. In that case
\begin{gather*}
  \big(\norm{ \phi^{*(k)} }_\infty + \epsilon\norm{ \phi^{*(k+1)} }_\infty \big) \rightarrow \norm{ \phi^{*(k)} }_\infty,\\
  \frac{P_{k-1}(\epsilon)}{(\inf\phi^{*\prime} + \epsilon\inf\phi^{*\prime\prime})^{2k - 1}} \rightarrow cste,\\
  \epsilon^k Q_{\nu}(1/\epsilon) = O(1 / \epsilon^{\nu - k})
\end{gather*}
As for any $l_j$, $k$ varies from $1$ to $l_j$, we get that $\norms{\f^{(l_j)}}_\infty = O(1/ \epsilon^{l_j-1})$ and that the product of the norms is $O(1/ \epsilon^{|\nu| - \lambda})$. Since $\epsilon\rightarrow 0$, the principal term is given by the largest $|\nu|$ and smallest $\lambda$, i.e. $|\nu|=s$ and $\lambda=1$ (we are in $\Hh^s_{\al,\lambda}(\R^d)$). It gives that the Sobolev norm of $h$ is $O(1 / \epsilon^{s-1})$. Concerning the asymptotic $\epsilon\rightarrow\infty$, it gives
\begin{gather*}
  \frac{\epsilon^k P_{k-1}(\epsilon)}{(\inf\phi^{*\prime} + \epsilon\inf\phi^{*\prime\prime})^{2k - 1}} \rightarrow cste,\\
  \big(\norm{ \phi^{*(k)} }_\infty + \epsilon\norm{ \phi^{*(k+1)} }_\infty \big) Q_{\nu}(1/\epsilon) \rightarrow cste.
\end{gather*}
The second limit holds because $Q_\nu$ has no constant term. All in all, it gives that $\norms{ h^{(\nu)} }_\infty = O(1)$
\end{proof}

Now that the regularity of the dual potentials has been proved, we prove a bound on $|\OTb(\al,\be) - \OTb(\al_n,\be_n)|$ which allows to apply the PAC-framework results in RKHS.

\begin{proposition}\label{prop-ineq-ot-sob}
Assume that Assumption~\ref{as:1},\ref{as:2} hold. One has 
\begin{align}
  |\OTb(\al,\be) - \OTb(\al_n,\be_n)| &\leq\, 2 \sup_{\f\in\Hh^s_{\al,\lambda}(\R^d)} |\dotp{\al - \al_n}{\f}|
  +\, 2 \sup_{\f\in\Hh^s_{\be,\lambda}(\R^d)} |\dotp{\be - \be_n}{\g}|.\nonumber
\end{align}
\end{proposition}
\begin{proof}
Write as $\Aa(\al,\be,\f,\g)$ the functional optimized in the dual program~\eqref{eq-dual-unb}. The assumptions give that the optimal dual potentials exist, such that we write $\OTb(\al,\be) = \Aa(\al,\be,\f,\g)$ and $\OTb(\al_n,\be) = \Aa(\al_n,\be,\f_n,\g_n)$. The optimality of those potentials give the following suboptimality inequalities
\begin{align*}
  \Aa(\al,\be,\f_n,\g_n) - \Aa(\al_n,\be,\f_n,\g_n) 
  &\leq \Aa(\al,\be,\f,\g) - \Aa(\al_n,\be,\f_n,\g_n)\\
  &\leq \Aa(\al,\be,\f,\g) - \Aa(\al_n,\be,\f,\g).
\end{align*}
The central term is $\OTb(\al,\be) - \OTb(\al_n,\be)$, thus these bounds give
\begin{align*}
  |\OTb(\al,\be) - \OTb(\al_n,\be)| \leq &|\Aa(\al,\be,\f,\g) - \Aa(\al_n,\be,\f,\g)|\\
   &+ |\Aa(\al,\be,\f_n,\g_n) - \Aa(\al_n,\be,\f_n,\g_n)|.
\end{align*}
 We now bound each term. The proof is similar for both. Concerning the first term, one has
\begin{align*}
  |\Aa(\al,\be,\f,\g) &- \Aa(\al_n,\be,\f,\g)|
  = |\dotp{\al - \al_n}{-\phi^*(-\f)} - \epsilon\dotp{(\al - \al_n)\otimes\be}{e^{\frac{\f\oplus\g - \C}{\epsilon}} - 1}|.
\end{align*}
The measure $\al - \al_n$ has zero mean, thus constant terms cancel out. The dual optimality condition under Assumption~\ref{as:2} is $\dotp{\be}{\tefgc}= \nabla\phi^*(-\f)$. It yields
\begin{align*}
  |\Aa(\al,\be,\f,\g) - \Aa(\al_n,\be,\f,\g)|&= |\dotp{\al - \al_n}{-\phi^*(-\f) - \epsilon\nabla\phi^*(-\f)}|
  \leq \sup_{\f\in\Hh^s_{\al,\lambda}(\R^d)} |\dotp{\al - \al_n}{\f}|.
\end{align*}
Proposition~\ref{prop-dual-pot-sob} gives that $\phi^*(-\f) + \epsilon\nabla\phi^*(-\f)\in\Hh^s_{\al,\lambda}(\R^d)$, hence the last inequality with a supremum.

The proof is the same for the second term. The inequality for $|\OTb(\al,\be) - \OTb(\al_n,\be_n)|$ is obtained via a triangle inequality
\begin{align*}
  |\OTb(\al,\be) - \OTb(\al_n,\be_n)|\leq &|\OTb(\al,\be) - \OTb(\al_n,\be)|
   + |\OTb(\al_n,\be) - \OTb(\al_n,\be_n)|.
\end{align*}
The bound detailed previously applies for both terms, since it holds when one argument is fixed and the other is empirically estimated. 
\end{proof}

\bibliographystyle{alpha}
\bibliography{biblio}

\newcommand{\etalchar}[1]{$^{#1}$}
\begin{thebibliography}{RJBVE19}

\bibitem[ACB17]{arjovsky2017wasserstein}
Martin Arjovsky, Soumith Chintala, and L{\'e}on Bottou.
\newblock Wasserstein {GAN}.
\newblock {\em arXiv preprint arXiv:1701.07875}, 2017.

\bibitem[Ale81]{alefeld1981convergence}
Georg Alefeld.
\newblock On the convergence of {Halley's} method.
\newblock {\em The American Mathematical Monthly}, 88(7):530--536, 1981.

\bibitem[BL21]{baradat2021regularized}
Aymeric Baradat and Hugo Lavenant.
\newblock Regularized unbalanced optimal transport as entropy minimization with
  respect to branching brownian motion.
\newblock {\em arXiv preprint arXiv:2111.01666}, 2021.

\bibitem[BM02]{bartlett2002rademacher}
Peter~L Bartlett and Shahar Mendelson.
\newblock {Rademacher} and {Gaussian} complexities: Risk bounds and structural
  results.
\newblock {\em Journal of Machine Learning Research}, 3(Nov):463--482, 2002.

\bibitem[Cal61]{calderon1961lebesgue}
Alberto~P Calder{\'o}n.
\newblock Lebesgue spaces of differentiable functions and distributions.
\newblock In {\em Proc. Sympos. Pure Math}, volume~4, pages 33--49, 1961.

\bibitem[CB18]{chizat2018global}
Lenaic Chizat and Francis Bach.
\newblock On the global convergence of gradient descent for over-parameterized
  models using optimal transport.
\newblock {\em Advances in neural information processing systems},
  31:3036--3046, 2018.

\bibitem[CDM17]{chizat2017tumor}
L{\'e}na{\"\i}c Chizat and Simone Di~Marino.
\newblock A tumor growth model of {Hele--Shaw} type as a gradient flow.
\newblock {\em arXiv preprint arXiv:1712.06124}, 2017.

\bibitem[CFG{\etalchar{+}}20]{charlier2020kernel}
Benjamin Charlier, Jean Feydy, Joan~Alexis Glaun{\`e}s, Fran{\c{c}}ois-David
  Collin, and Ghislain Durif.
\newblock Kernel operations on the {GPU}, with autodiff, without memory
  overflows.
\newblock {\em arXiv preprint arXiv:2004.11127}, 2020.

\bibitem[CGH{\etalchar{+}}96]{corless1996lambertw}
Robert~M Corless, Gaston~H Gonnet, David~EG Hare, David~J Jeffrey, and Donald~E
  Knuth.
\newblock On the {Lambert-W} function.
\newblock {\em Advances in Computational mathematics}, 5(1):329--359, 1996.

\bibitem[Chi19]{chizat2019sparse}
Lenaic Chizat.
\newblock Sparse optimization on measures with over-parameterized gradient
  descent.
\newblock {\em arXiv preprint arXiv:1907.10300}, 2019.

\bibitem[CPSV15]{chizat2015unbalanced}
Lenaic Chizat, Gabriel Peyr{\'e}, Bernhard Schmitzer, and Fran{\c{c}}ois-Xavier
  Vialard.
\newblock Unbalanced optimal transport: geometry and {Kantorovich} formulation.
\newblock {\em arXiv preprint arXiv:1508.05216}, 2015.

\bibitem[CPSV18]{chizat2016scaling}
Lenaic Chizat, Gabriel Peyr{\'e}, Bernhard Schmitzer, and Fran{\c{c}}ois-Xavier
  Vialard.
\newblock Scaling algorithms for unbalanced transport problems.
\newblock {\em to appear in Mathematics of Computation}, 2018.

\bibitem[CR13]{combettes2013moreau}
Patrick~L Combettes and Noli~N Reyes.
\newblock {Moreau}’s decomposition in {Banach} spaces.
\newblock {\em Mathematical Programming}, 139(1-2):103--114, 2013.

\bibitem[CRL{\etalchar{+}}20]{chizat2020faster}
Lenaic Chizat, Pierre Roussillon, Flavien L{\'e}ger, Fran{\c{c}}ois-Xavier
  Vialard, and Gabriel Peyr{\'e}.
\newblock Faster {Wasserstein} distance estimation with the {Sinkhorn}
  divergence.
\newblock {\em Advances in Neural Information Processing Systems}, 33, 2020.

\bibitem[CS96]{constantine1996multivariate}
Gregory~M. Constantine and Thomas~H. Savits.
\newblock A multivariate {Faa di Bruno} formula with applications.
\newblock {\em Transactions of the American Mathematical Society},
  348(2):503--520, 1996.

\bibitem[Csi67]{csiszar1967information}
Imre Csisz{\'a}r.
\newblock Information-type measures of difference of probability distributions
  and indirect observation.
\newblock {\em Studia Scientiarum Mathematicarum Hungarica}, 2:229--318, 1967.

\bibitem[Cut13]{cuturi2013lightspeed}
Marco Cuturi.
\newblock Sinkhorn distances: Lightspeed computation of optimal transport.
\newblock In {\em Adv. in Neural Information Processing Systems}, pages
  2292--2300, 2013.

\bibitem[Dud69]{dudley1969speed}
Richard~M. Dudley.
\newblock The speed of mean {Glivenko-Cantelli} convergence.
\newblock {\em The Annals of Mathematical Statistics}, 40(1):40--50, 1969.

\bibitem[FCVP17]{feydy2017optimal}
Jean Feydy, Benjamin Charlier, Fran{\c{c}}ois-Xavier Vialard, and Gabriel
  Peyr{\'e}.
\newblock Optimal transport for diffeomorphic registration.
\newblock In {\em International Conference on Medical Image Computing and
  Computer-Assisted Intervention}, pages 291--299. Springer, 2017.

\bibitem[Fey20]{feydy2020thesis}
Jean Feydy.
\newblock {\em Geometric data analysis, beyond convolutions}.
\newblock PhD thesis, Universit{\'e} Paris-Saclay, 2020.

\bibitem[FGCB20]{feydy2020fast}
Jean Feydy, Joan Glaun{\`e}s, Benjamin Charlier, and Michael Bronstein.
\newblock Fast geometric learning with symbolic matrices.
\newblock {\em Proc. NeurIPS}, 2(4):6, 2020.

\bibitem[Fig10]{figalli2010optimal}
Alessio Figalli.
\newblock The optimal partial transport problem.
\newblock {\em Archive for rational mechanics and analysis}, 195(2):533--560,
  2010.

\bibitem[FSFC21]{fatras2021unbalanced}
Kilian Fatras, Thibault S{\'e}journ{\'e}, R{\'e}mi Flamary, and Nicolas Courty.
\newblock Unbalanced minibatch optimal transport; applications to domain
  adaptation.
\newblock In {\em International Conference on Machine Learning}, pages
  3186--3197. PMLR, 2021.

\bibitem[FSV{\etalchar{+}}19]{feydy2018interpolating}
Jean Feydy, Thibault S{\'e}journ{\'e}, Fran{\c{c}}ois-Xavier Vialard, Shun-ichi
  Amari, Alain Trouv{\'e}, and Gabriel Peyr{\'e}.
\newblock Interpolating between optimal transport and {MMD} using {Sinkhorn}
  divergences.
\newblock In {\em The 22nd International Conference on Artificial Intelligence
  and Statistics}, pages 2681--2690. PMLR, 2019.

\bibitem[GBR{\etalchar{+}}06]{gretton2007kernel}
Arthur Gretton, Karsten Borgwardt, Malte Rasch, Bernhard Sch{\"o}lkopf, and
  Alex Smola.
\newblock A kernel method for the two-sample-problem.
\newblock {\em Advances in neural information processing systems}, 19:513--520,
  2006.

\bibitem[GBR{\etalchar{+}}12]{gretton2012kernel}
Arthur Gretton, Karsten~M Borgwardt, Malte~J Rasch, Bernhard Sch{\"o}lkopf, and
  Alexander Smola.
\newblock A kernel two-sample test.
\newblock {\em Journal of Machine Learning Research}, 13(Mar):723--773, 2012.

\bibitem[GCB{\etalchar{+}}19]{genevay2018sample}
Aude Genevay, L{\'e}naic Chizat, Francis Bach, Marco Cuturi, and Gabriel
  Peyr{\'e}.
\newblock Sample complexity of {Sinkhorn} divergences.
\newblock In {\em The 22nd International Conference on Artificial Intelligence
  and Statistics}, pages 1574--1583. PMLR, 2019.

\bibitem[GPC18]{genevay2018learning}
Aude Genevay, Gabriel Peyr{\'e}, and Marco Cuturi.
\newblock Learning generative models with {Sinkhorn} divergences.
\newblock In {\em International Conference on Artificial Intelligence and
  Statistics}, pages 1608--1617, 2018.

\bibitem[Han92]{hanin1992kantorovich}
Leonid~G Hanin.
\newblock {Kantorovich--Rubinstein} norm and its application in the theory of
  {Lipschitz} spaces.
\newblock {\em Proceedings of the American Mathematical Society},
  115(2):345--352, 1992.

\bibitem[Han99]{hanin1999extension}
Leonid~G Hanin.
\newblock An extension of the {Kantorovich} norm.
\newblock {\em Contemporary Mathematics}, 226:113--130, 1999.

\bibitem[Hun07]{matplotlib}
J.~D. Hunter.
\newblock Matplotlib: A {2D} graphics environment.
\newblock {\em Computing in Science \& Engineering}, 9(3):90--95, 2007.

\bibitem[JCG20]{janati2020debiased}
Hicham Janati, Marco Cuturi, and Alexandre Gramfort.
\newblock Debiased {Sinkhorn} barycenters.
\newblock {\em arXiv preprint arXiv:2006.02575}, 2020.

\bibitem[KMV{\etalchar{+}}16]{kondratyev2016new}
Stanislav Kondratyev, L{\'e}onard Monsaingeon, Dmitry Vorotnikov, et~al.
\newblock A new optimal transport distance on the space of finite {Radon}
  measures.
\newblock {\em Advances in Differential Equations}, 21(11/12):1117--1164, 2016.

\bibitem[KRU14]{knight2014symmetry}
Philip~A Knight, Daniel Ruiz, and Bora U{\c{c}}ar.
\newblock A symmetry preserving algorithm for matrix scaling.
\newblock {\em SIAM journal on Matrix Analysis and Applications},
  35(3):931--955, 2014.

\bibitem[LBR19]{lee2019parallel}
John Lee, Nicholas~P Bertrand, and Christopher~J Rozell.
\newblock Parallel unbalanced optimal transport regularization for large scale
  imaging problems.
\newblock {\em arXiv preprint arXiv:1909.00149}, 2019.

\bibitem[LMS15]{liero2015optimal}
Matthias Liero, Alexander Mielke, and Giuseppe Savar{\'e}.
\newblock Optimal entropy-transport problems and a new {Hellinger--Kantorovich}
  distance between positive measures.
\newblock {\em Inventiones mathematicae}, pages 1--149, 2015.

\bibitem[LN12]{lemmens2012nonlinear}
Bas Lemmens and Roger Nussbaum.
\newblock {\em Nonlinear Perron-Frobenius Theory}, volume 189.
\newblock Cambridge University Press, 2012.

\bibitem[LQG19]{flownet3d}
Xingyu Liu, Charles~R Qi, and Leonidas~J Guibas.
\newblock {FlowNet3D}: Learning scene flow in {3D} point clouds.
\newblock In {\em Proceedings of the IEEE/CVF Conference on Computer Vision and
  Pattern Recognition}, pages 529--537, 2019.

\bibitem[MG15]{kitti}
Moritz Menze and Andreas Geiger.
\newblock Object scene flow for autonomous vehicles.
\newblock In {\em Conference on Computer Vision and Pattern Recognition
  (CVPR)}, 2015.

\bibitem[MGS{\etalchar{+}}21]{mukherjee2021outlier}
Debarghya Mukherjee, Aritra Guha, Justin~M Solomon, Yuekai Sun, and Mikhail
  Yurochkin.
\newblock Outlier-robust optimal transport.
\newblock In {\em International Conference on Machine Learning}, pages
  7850--7860. PMLR, 2021.

\bibitem[MW19]{mena2019statistical}
Gonzalo Mena and Jonathan Weed.
\newblock Statistical bounds for entropic optimal transport: sample complexity
  and the central limit theorem.
\newblock {\em arXiv preprint arXiv:1905.11882}, 2019.

\bibitem[Oli06]{numpy}
Travis Oliphant.
\newblock A guide to numpy, 2006.

\bibitem[PGC{\etalchar{+}}17]{pytorch}
Adam Paszke, Sam Gross, Soumith Chintala, Gregory Chanan, Edward Yang, Zachary
  DeVito, Zeming Lin, Alban Desmaison, Luca Antiga, and Adam Lerer.
\newblock Automatic differentiation in pytorch.
\newblock 2017.

\bibitem[RJBVE19]{rotskoff2019global}
Grant Rotskoff, Samy Jelassi, Joan Bruna, and Eric Vanden-Eijnden.
\newblock Global convergence of neuron birth-death dynamics.
\newblock {\em arXiv preprint arXiv:1902.01843}, 2019.

\bibitem[RTC17]{ramdas2017wasserstein}
Aaditya Ramdas, Nicolas~Garcia Trillos, and Marco Cuturi.
\newblock On {Wasserstein} two-sample testing and related families of
  nonparametric tests.
\newblock {\em Entropy}, 19(2), 2017.

\bibitem[RW09]{rockafellar2009variational}
R~Tyrrell Rockafellar and Roger J-B Wets.
\newblock {\em Variational analysis}, volume 317.
\newblock Springer Science \& Business Media, 2009.

\bibitem[San15]{santambrogio2015optimal}
Filippo Santambrogio.
\newblock {\em Optimal Transport for applied mathematicians}, volume~87 of {\em
  Progress in Nonlinear Differential Equations and their applications}.
\newblock Springer, 2015.

\bibitem[SFL{\etalchar{+}}21]{robot}
Zhengyang Shen, Jean Feydy, Peirong Liu, Ariel Curiale, Ruben San Jose~Estepar,
  Raul San Jose~Estepar, and Marc Niethammer.
\newblock Accurate point cloud registration with robust optimal transport.
\newblock {\em Advances in Neural Information Processing Systems}, 34, 2021.

\bibitem[Sin64]{sinkhorn1964relationship}
Richard Sinkhorn.
\newblock A relationship between arbitrary positive matrices and doubly
  stochastic matrices.
\newblock {\em Ann. Math. Statist.}, 35:876--879, 1964.

\bibitem[SST{\etalchar{+}}17]{schiebinger2017reconstruction}
Geoffrey Schiebinger, Jian Shu, Marcin Tabaka, Brian Cleary, Vidya Subramanian,
  Aryeh Solomon, Siyan Liu, Stacie Lin, Peter Berube, Lia Lee, et~al.
\newblock Reconstruction of developmental landscapes by optimal-transport
  analysis of single-cell gene expression sheds light on cellular
  reprogramming.
\newblock {\em BioRxiv}, page 191056, 2017.

\bibitem[SW19]{schmitzer2019framework}
Bernhard Schmitzer and Benedikt Wirth.
\newblock A framework for {Wasserstein}-1-type metrics.
\newblock {\em to appear in Journal of Convex Analysis}, 2019.

\bibitem[SZRM18]{salimans2018improving}
Tim Salimans, Han Zhang, Alec Radford, and Dimitris Metaxas.
\newblock Improving {GAN}s using optimal transport.
\newblock In {\em International Conference on Learning Representations}, 2018.

\bibitem[Teb92]{teboulle1992entropic}
Marc Teboulle.
\newblock Entropic proximal mappings with applications to nonlinear
  programming.
\newblock {\em Mathematics of Operations Research}, 17(3):670--690, 1992.

\bibitem[Tse01]{tseng2001convergence}
Paul Tseng.
\newblock Convergence of a block coordinate descent method for
  nondifferentiable minimization.
\newblock {\em Journal of optimization theory and applications},
  109(3):475--494, 2001.

\bibitem[VBR{\etalchar{+}}99]{vedula1999three}
Sundar Vedula, Simon Baker, Peter Rander, Robert Collins, and Takeo Kanade.
\newblock Three-dimensional scene flow.
\newblock In {\em Proceedings of the Seventh IEEE International Conference on
  Computer Vision}, volume~2, pages 722--729. IEEE, 1999.

\bibitem[WB17]{weed2017sharp}
Jonathan Weed and Francis Bach.
\newblock Sharp asymptotic and finite-sample rates of convergence of empirical
  measures in {Wasserstein} distance.
\newblock {\em arXiv preprint arXiv:1707.00087}, 2017.

\bibitem[YU18]{yang2018scalable}
Karren~D Yang and Caroline Uhler.
\newblock Scalable unbalanced optimal transport using generative adversarial
  networks.
\newblock {\em arXiv preprint arXiv:1810.11447}, 2018.

\end{thebibliography}

\end{document}